\documentclass[twoside,12pt]{article}
\usepackage{amsmath,amssymb,amsthm,amsfonts}
\usepackage{graphics,graphicx,picture} 
\usepackage{epsfig} 
\usepackage{extarrows}
\usepackage{tikz}
\usepackage{subcaption}
\usepackage{epstopdf}
\usepackage[colorlinks=true]{hyperref}
\usepackage{txfonts}
\usepackage{enumerate}
\usepackage[title]{appendix}
\usepackage[numbers,sort&compress]{natbib}
\usepackage{subcaption}
\graphicspath{{figures/}}		  

\numberwithin{equation}{section}

\newcommand{\be}{\begin{equation}}
\newcommand{\ee}{\end{equation}}

\newcommand{\ben}{\begin{eqnarray*}}
\newcommand{\enn}{\end{eqnarray*}}

\newcommand{\eps}{\epsilon}

\newtheorem{proposition}{Proposition}[section]
\newtheorem{theorem}{\textbf Theorem}[section]
\newtheorem{lemma}{\textbf Lemma}[section]
\newtheorem{remark}{\textbf Remark}[section]

 \numberwithin{equation}{section}

\def\sech{\mathrm{sech}}

\renewcommand{\theequation}{\arabic{section}.\arabic{equation}}

\begin{document}

\title{\textbf{Existence, Stability and Slow Dynamics of Spikes in a 1D
    Minimal Keller--Segel Model with Logistic Growth}}

 \author{Fanze Kong\thanks{Fanze Kong: fzkong@math.ubc.ca}, \hspace{0.15cm}
 Michael J. Ward\thanks{Michael Ward: ward@math.ubc.ca},\hspace{0.15cm}
 and
 Juncheng Wei\thanks{Juncheng Wei: jcwei@math.ubc.ca}\\
 \fontsize{11pt}{9pt}\selectfont { Dept. of Mathematics, University of British
  Columbia, Vancouver, Canada, V6T 1Z2}}
\date{}
\maketitle
\vspace{-0.25in} We analyze the existence, linear stability, and slow
dynamics of localized 1D spike patterns for a Keller--Segel model of
chemotaxis that includes the effect of logistic growth of the cellular
population.  Our analysis of localized patterns for this two-component
reaction-diffusion (RD) model is based, not on the usual limit of a
large chemotactic drift coefficient, but instead on the singular limit
of an asymptotically small diffusivity $d_2=\epsilon^2\ll 1$ of the
chemoattractant concentration field.  In the limit $d_2\ll 1$,
steady-state and quasi-equilibrium 1D multi-spike patterns are
constructed asymptotically. To determine the linear stability of
steady-state $N$-spike patterns we analyze the spectral properties
associated with both the ``large'' ${\mathcal O}(1)$ and the ``small''
$o(1)$ eigenvalues associated with the linearization of the
Keller-Segel model. By analyzing a nonlocal eigenvalue problem
characterizing the large eigenvalues, it is shown that $N$-spike
equilibria can be destabilized by a zero-eigenvalue crossing leading
to a competition instability if the cellular diffusion rate $d_1$
exceeds a threshold, or from a Hopf bifurcation if a relaxation time
constant $\tau$ is too large. In addition, a matrix eigenvalue problem
that governs the stability properties of an $N$-spike steady-state
with respect to the small eigenvalues is derived. From an analysis of
this matrix problem, an explicit range of $d_1$ where the $N$-spike
steady-state is stable to the small eigenvalues is
identified. Finally, for quasi-equilibrium spike patterns that are
stable on an ${\mathcal O}(1)$ time-scale, we derive a differential
algebraic system (DAE) governing the slow dynamics of a collection of
localized spikes. Unexpectedly, our analysis of the KS model with
logistic growth in the singular limit $d_2\ll 1$ is rather closely
related to the analysis of spike patterns for the Gierer-Meinhardt RD
system.

\vspace{0.3cm}
\maketitle

\section{Introduction}\label{sec:intro}
The study of pattern formation phenomena for RD
systems originates from the pioneering work of Alan M. Turing
\cite{Turing}. In an attempt to understand the mechanism underlying
biological morphogenesis, he discovered that spatially homogeneous
steady-states of reaction kinetics for multi-component systems that
are linearly stable can be destabilized in the presence of
diffusion. This diffusion-induced instability, now commonly referred
to as a Turing instability, typically leads to the formation of stable
spatial patterns that break the symmetry of the spatially uniform
state. Based on this insight, modern bifurcation-theoretic tools such
as weakly nonlinear multi-scale analysis and Lyapunov-Schmidt
reductions have been used ubiquitously to characterize pattern
formation near onset for RD systems. However, to analyze localized
patterns for RD systems away from the onset of where a Turing instability
occurs, new theoretical approaches are needed. Over the past two
decades, there has been a focus on developing such novel analytical tools
to study the existence, stability, and dynamical behavior of
``far-from-equilibrium'' spatially localized patterns, such as stripes
and spots, for two-component RD systems that combine only diffusion
and nonlinear reactions (see \cite{WeiWinterBook}, \cite{dkp},
\cite{chenwan2011stability}, \cite{ward_survey} and references
therein).

In contrast, the analytical study of localized pattern formation for
RD systems that combine diffusion, nonlinear reactions, and {\em
  advection} poses many new theoretical challenges
(cf.~\cite{wang2013spiky,Chen2014,tsecrime,bast_veg}.  The most common
such RD models are chemotaxis-type systems, such as the prototypical
Keller--Segel (KS) system \cite{Keller1970,Keller1971}, that are
widely used to model how cells or bacteria direct their movements in
response to an environmental chemical stimulus, such as observed in
some foundational experiments
(cf.~\cite{engelmann1882sauerstoffausscheidung},
\cite{adler1975chemotaxis}, \cite{brown1974temporal}). Chemotactic
effects have been shown to play a key role in a wide variety of other
biological processes such as, cell-cell interactions in the immune
system, the organization of tissues during embryogenesis, and the
growth of tumor cells
\cite{brinkmann2004neutrophil,oppenheim2005alarmins,rossi2000biology}.

In 1971, Keller and Segel \cite{Keller1970,Keller1971} proposed the
following coupled RD system to model chemotaxis:
 \begin{equation}\tag{K-S}\label{K-S}
\left\{
\begin{array}{ll}
  \tau u_t= \overbrace{d_1\Delta u}^{\text{cellular diffusion}}-
  \overbrace{\chi \nabla\cdot (S(u,v) \nabla v)}^{\text{advection}}+
  \overbrace{f(u)}^{\text{source}}\,,&x \in \Omega\,,\, t>0\,, \\
  v_t=\overbrace{d_2\Delta v}^{\text{chemical signal diffusion}}+
  \overbrace{g(u,v)}^{\text{chemical production/consumption}}\,,&x \in \Omega\,,\,
                                                                  t>0\,.
\end{array}
\right.
\end{equation}
Here $\Omega$ is either a bounded domain with smooth boundary
$\partial\Omega$ or the whole space $\mathbb R^N$ with $N\geq 1$.  In
(\ref{K-S}), $u$ is the cellular density, $v$ is the chemotactic
concentration, $\tau$ is the reaction time constant, $d_1$ and $d_2$
are diffusivities of $u$ and $v$, respectively, while $S(u,v)$ models
the chemotactic or directed movement. The chemotactic drift coefficient
$\chi$ measures the relative strength of this directed motion.  In a bounded
domain, no-flux boundary conditions are usually imposed on (\ref{K-S})
to ensure that the cellular aggregation is spontaneous.

One main research focus on the chemotaxis PDE system (\ref{K-S}) is
the study of self-aggregating pattern-formation phenomena and the
determination of whether finite-time singularities can occur.  There
are two well-known approaches to study the possibility of such blow-up
behavior.  The first approach is to analyze the well-posedness and
global existence of solutions, which can rule out the trivial
dynamics.  The other approach is to construct spatially inhomogeneous
patterns and to study their local and long time behaviors. For a 
survey of diverse applications and some mathematical results for
(\ref{K-S}) and its variants see
\cite{horstmann2001,Horstmann2003,horstmann2004,hillen2009user,wink_rev,
  painter_review}.

Our goal is to analyze certain pattern-formation properties for
a KS model with logistic growth, given by
\begin{equation}\label{timedependent1}
   \begin{cases}
     \tau u_t=d_1\Delta u - \chi\nabla\cdot (u \nabla v)+\mu
     u(\bar u-u) \,,
     &x\in\Omega,\, t>0\,,  \\
     v_t=d_2\Delta v- v + u, &x\in\Omega\,,\, t>0\,,\\
     \frac{\partial u}{\partial\textbf{n}}(x,t)= \frac{\partial
       v}{\partial \textbf{n}}(x,t)=0\,,&x\in\partial\Omega\,,\,
     t>0\,,
   \end{cases}
 \end{equation}
 where $u(x,0)=u^{0} (x)$ and $v(x,0)=v^{0}(x)$ are non-negative
 initial data. Here $S$ and $g$ in (\ref{K-S}) are taken to be linear,
 i.~e.~$S(u,v)=u$ and $g(u,v)=u-v$. In (\ref{timedependent1}),
 $f(u)=\mu u(\bar u-u)$ describes the cellular population growth
 dynamics, where $\mu>0$ denotes the logistic growth rate and
 $\bar u>0$ represents the carrying capacity of the habitat for cells.
 Before discussing some previous results for (\ref{timedependent1}), we will
 highlight some results for the case $f(u)\equiv 0.$

Without logistic growth, (\ref{timedependent1}) in 2D admits blow-up
phenomenon, which depends on the cellular mass
$M:=\int_{\Omega}u(x,0)\,dx$.  In particular, if $M<M_0:={4\pi/\chi}$
for the bounded domain or $M<M_0:={8\pi/\chi}$ for the whole space
$\mathbb R^2$, the solution to (\ref{timedependent1}) will globally
exist \cite{nagai1997application}; otherwise (\ref{timedependent1})
admits finite time blow-up solutions
\cite{nanjundiah1973,childress1981,herrero1996,senba2000some,wang2002steady}.
For the steady-state problem of (\ref{timedependent1}) in 2D, the
pioneering study of Lin, Ni and Takagi
\cite{lin1988large,ni1993locating} constructed large amplitude
stationary solutions analytically.  Motivated by this seminal work, it
has been subsequently revealed that non-constant steady states with
$f(u)\equiv 0$ can exhibit a wide range of solution behaviors
\cite{Gui1999,del2006collapsing,carrillo2021boundary}.  In particular,
Wei and Del Pino \cite{del2006collapsing} constructed a multi-spike
equilibrium to (\ref{timedependent1}) in 2D via the ``localized energy
method".  In contrast to the 2D case, the solution to
(\ref{timedependent1}) in 1D with $f\equiv 0$ is uniformly bounded in
time \cite{osaki2001finite,nagai1995}.  For the stationary
counterpart, spatially non-uniform steady states were constructed in
\cite{hillen2004one,kang2007stability,wang2013spiky,Chen2014}.  In
particular, Wang and Xu \cite{wang2013spiky} adopted an innovative
bifurcation-theoretic approach to directly treat the steady-state
problem for (\ref{timedependent1}) in 1D without relying significantly
on the special structure of (\ref{timedependent1}).

With the logistic source term, i.e. when $f(u)=\mu u(\bar u-u)$,
Winkler et
al. \cite{winkler2008,winkler2010log,winkler2014,xiang2018,xiang2018sub,lin2017convergence}
showed that the solution of (\ref{timedependent1}) globally exists in
any dimension when the effect of the logistic growth is strong enough.
Some results regarding the construction of spatially inhomogeneous
equilibria for (\ref{timedependent1}) are given in
\cite{jin2016pattern,wang2016qualitative,kolokolnikov2014,KWX2022,KWX20221}.
However, the dynamics of (\ref{timedependent1}) can be highly
intricate and are not nearly as well understood as for the case where
$f(u)\equiv 0$.  Hillen and Painter et
al. \cite{painter2011spatio,hillen2013merging} studied
(\ref{timedependent1}) numerically and revealed the possibility of
periodic and chaotic dynamics consisting of repeated spike nucleation
and amalgamation events.

From a formal asymptotic analysis together with numerical simulations,
Kolokolnikov et al. \cite{kolokolnikov2014} showed that there exists
three types of spiky steady states to (\ref{timedependent1}) in 1D.
In particular, they constructed a locally stable single interior spike
solution, which does not occur in the minimal KS model without the
logistic source term.  To more fully understand how a logistic source
term allows for spiky patterns, the focus in this paper is to study
the existence, stability, and dynamics of spiky solutions to the 1D
version of (\ref{timedependent1}), which is formulated as
\begin{subequations}\label{timedependent}
\begin{align}
  \tau u_t&=d_1u_{xx} - \chi (u v_x)_x + \mu u(\bar u - u) \,, \quad
            |x|<1\,, \,\, t>0\,; \qquad u_x(\pm 1,t)=0\,,
  \label{timedependent_a}\\
  v_t&=d_2v_{xx} - v + u\,, \quad |x|<1\,,\,\, t>0;
       \qquad v_x(\pm 1,t)=0 \,,\label{timedependent_b}
\end{align}
\end{subequations}
with $u(x,0)=u^{0}(x)$ and $v(x,0)=v^{0}(x)$.  Our main goal is to
construct $N$-spike equilibria for (\ref{timedependent}) with equal heights
in the limit where the diffusivity $d_2$ is small, and to analyze the
linear stability properties of these localized steady-state patterns.
Labeling $d_2=\epsilon^2\ll 1$, the steady-state problem for
(\ref{timedependent}) on $|x|<1$ is
\begin{align}\label{ss1}
  d_1u_{xx} - \chi(u v_x)_x + \mu u(\bar u - u) = 0\,, \qquad
        \epsilon^2 v_{xx} - v + u =  0\,; \qquad
        u_{x}(\pm 1)=v_x(\pm 1)=0\,.
\end{align}
We will also explicitly construct quasi-equilibrium patterns for
(\ref{timedependent}) where the spike locations evolve dynamically on
some asymptotically long time scale as $\epsilon\to 0$ towards their
steady-state locations. 

We emphasize that our analysis of localized pattern formation for
(\ref{timedependent}) in the singular limit $d_2\to 0$ is in distinct
contrast to the previous analytical and numerical studies of pattern
formation properties for (\ref{timedependent}) that were undertaken in
the more traditional large chemotactic drift limit $\chi\gg 1$ (cf.~
\cite{KWX2022}, \cite{KWX20221}, \cite{jin2016pattern}, and
\cite{wang2016qualitative}). In the singular limit $d_2\ll 1$, our
analysis and results for the existence, linear stability, and slow
dynamics of 1D spike patterns for (\ref{timedependent}) will be shown
to be rather closely related to  corresponding studies of 1D spike
patterns for the Gierer-Meinhardt (GM) RD model
(cf.~\cite{iron2001stability}, \cite{iw2}, \cite{wwhopf},
\cite{wei2007existence}).

The outline of this paper is as follows.  In \S \ref{sec2}, we
construct $N$-spike quasi-equilibrium spike patterns for
(\ref{timedependent}) using the method of matched asymptotic
expansions in the limit $\epsilon\ll 1$. Our analysis reveals a novel,
analytically tractable, sub-inner asymptotic structure that
characterizes the spatial profile of a localized spike. With regards
to the linear stability analysis, in \S \ref{sec3} and \S \ref{sec4}
we analyze the large and small eigenvalues in the discrete spectrum of
the linearization of (\ref{timedependent}) around an $N$-spike
steady-state, respectively. The spectral properties for the large
eigenvalues are shown to be governed by a nonlocal eigenvalue problem
(NLEP), which has a somewhat similar form to the NLEP that arises in
the study of spike stability in the GM model. For $\tau=0$, our NLEP
linear stability analysis will provide a range of $d_1$ values for
which the $N$-spike equilibrium is linearly stable. Moreover, for
$\tau>0$, we show that spike amplitude oscillations can occur via a
Hopf bifurcation associated with the NLEP. Hypergeometric functions
are shown to be key for accurately calculating the stability
thresholds from the NLEP. For the small eigenvalues, in \S
\ref{sec:small_thresh} we will determine analytically an explicit
range of $d_1$ for which the steady-state solution is linearly stable
to translation instabilities.

A differential-algebraic (DAE) system characterizing slow spike
dynamics for quasi-equilibrium patterns is derived in \S
\ref{sec5}. The slow dynamics obtained from this DAE system are
favorably compared with corresponding results computed from full PDE
numerical simulations of (\ref{timedependent}). Moreover, in \S
\ref{sec:balance_dae} we show that the explicit expressions for the
small eigenvalues that are obtained in \S \ref{sec4} can also be
derived from a linearization of the DAE dynamics around the
steady-state spike locations. In \S \ref{sec:discussion}, we suggest a
few open problems, and we compare and contrast the analytical approach
used and results previously obtained for spike patterns in the 1D GM
model with that obtained herein for the KS model
(\ref{timedependent}).
 
\section{Asymptotic Analysis of the $N$-Spike Quasi-equilibrium}\label{sec2}
In this section, we construct $N$-spike quasi-steady state solutions to
(\ref{ss1}) in the limit $\epsilon\ll 1$ by using the method of matched
asymptotic expansions. We define the centers of the spikes as $x_j$, for
$j=1,\ldots,N$, and assume that they are well-separated in the sense that
$|x_j-1|={\mathcal O}(1)$, $|x_j+1|={\mathcal O}(1)$, and
$|x_i-x_j|={\mathcal O}(1)$ for $i\not= j.$

\subsection{Inner Solution}
In the inner region near each $x_j$ where the cellular density $u$ and
the chemical concentration $v$ are localized, we introduce new local
variables $y=\epsilon^{-1}(x-x_j)$, $U_j(y)=u(x_j+\epsilon y)$, and
$V_j(y)=v(x_j+\epsilon y)$, and we expand
\begin{align}\label{expansion}
 U_j(y)=
  U_{0j}(y)+\epsilon^2 U_{1j}(y)+\ldots\,, \quad V_j(y)=
  V_{0j}(y)+\epsilon^2V_{1j}(y)+\ldots\,, \quad y=\epsilon^{-1}(x-x_j)\,.
\end{align}
Here the subscripts $0$ and $1$ in $(U_{0j},V_{0j})$ and
$(U_{1j},V_{1j})$ are the orders of the expansion, while $j$ refers
to the $j^{\mbox{th}}$ inner region.   The leading order terms, found by
substituting (\ref{expansion}) into (\ref{ss1}), yield on
$-\infty<y<\infty$ that
\begin{align}\label{leading}
  (U_{0j}^{\prime}-\bar\chi U_{0j}V_{0j}^{\prime})^{\prime}=0\,, \quad
  U_{0j}^{\prime}(0)=0 \,; \qquad
  V_{0j}^{\prime\prime}-V_{0j}+U_{0j}=0\,,\quad V_{0j}^{\prime}(0)=0\,.
\end{align}
In this so-called {\em core problem}, we define
$\bar\chi:={\chi/d_1}$.  Below, for simplicity, we omit the subscript
$j$ for $U_{0j}$ and $V_{0j}$. Moreover, in the analysis below, terms
such as $h_x$ and $H^{\prime}$ denote differentiation in $x$ and $y$,
respectively.

Upon imposing $U_0\rightarrow 0$ as $\vert y \vert\rightarrow \infty$,
(\ref{leading}) yields $U_0=C_j e^{\bar \chi V_0},$ where the constant
$C_j>0$ will be determined below. Then, from (\ref{leading}), we
conclude that the spike profile is characterized by a homoclinic
solution $V_0$ to
\begin{align}\label{Hamiltonian}
  V_0^{\prime\prime}+Q(V_0)=0\,, \quad -\infty <y<+\infty\,; \qquad
  V_0^{\prime}(0)=0 \,; \qquad V_0\rightarrow s_j\,, \,\,\,
  V_0^{\prime}\,, \, V_0^{\prime\prime}\rightarrow 0\,,\,\,
  \mbox{as} \,\, |y|\to \infty \,,
 \end{align}
 where $Q(V_0):=-V_0+C_je^{\bar \chi V_0}$ and $s_j$ satisfies
 $Q(s_j)=0$, so that $C_je^{\bar \chi s_j}=s_j$.  The first integral
 of (\ref{Hamiltonian}) is
\begin{align}\label{firstintegral}
  \frac{1}{2}\left(V_0^{\prime}\right)^2+K(V_0;C_j)=0\,, \qquad
  K(V_0;C_j):=\int_{s_j}^{V_0}Q(\xi)\, d\xi=\frac{1}{2}\left(s^2_j-
  V_0^2\right)+\frac{C_j}{\bar \chi}\left(e^{\bar \chi V_0}-
   e^{\bar \chi s_j}\right)\,.
\end{align}
Imposing that $V_0$ is even and monotone decreasing in $y>0$, we
obtain from (\ref{firstintegral}) that
\begin{align}\label{V0prime}
V_0^{\prime}=-\sqrt{-2K(V_0;C_j)}\quad\text{for}\quad 0\leq y <\infty \,.
\end{align}
By separating variables in (\ref{V0prime}), we obtain an
implicit equation for $V_0$, defined on $y\geq 0$,  given by
\begin{align}\label{implicit}
y=\int_{V_0}^{v_{\max j}}\frac{d\xi}{\sqrt{-2K(\xi;C_j)}} \,,
\end{align}
where $v_{\max j}:=V_0(0)$ is the amplitude of $V_0$. By setting
$V_0^{\prime}(0)=0$ in (\ref{V0prime}) and using
$s_j=C_j e^{\bar{\chi}s_j}$, it follows that $v_{\max j}$ satisfies
the nonlinear algebraic equation
\begin{align}\label{keyrelation}
  -\frac{1}{2} v_{\max j}^2+\frac{1}{2}s^2_j+
  \frac{C_j}{\bar\chi}e^{\bar\chi v_{\max j}}-\frac{s_j}{\bar\chi}=0\,.
\end{align}

In the outer region, defined at ${\mathcal O}(1)$ distances from the
centers of the spikes, we will construct a solution where $u$ and $v$
are $o(1)$ as $\epsilon\to 0$. As a result, we anticipate from
the matching condition that $s_j=o(1)$ as $\epsilon\to 0$,
which allows us to approximate our implicit form of $V_0$ given by
(\ref{implicit}). To this end, we suppose that $s_j=o(1)$ and
$v_{\max j}^{-1}=o(1)$ when $\epsilon\ll 1$, so that
(\ref{keyrelation}) reduces to
\begin{align}\label{keyrelationship1}
\frac{C_j}{\bar\chi}e^{\bar\chi v_{\max j}}\sim \frac{1}{2}v_{\max j}^2 \,.
\end{align}
Next, we introduce new variables $z$ and $\tilde{V}_0$, which
constitute a sub-layer within the inner region, defined by
\begin{equation*}
  z= y \, v_{\max j} \,, \qquad  V_0(y)=v_{\max j}+\tilde V_0(z)\,.
\end{equation*}
By using (\ref{keyrelationship1}), and where primes now indicate
derivatives in $z$, we can rewrite the $V_0$-equation in
\eqref{Hamiltonian} as
\begin{align}\label{innerinerleadingbefore}
  v_{\max j}^2\tilde V_{0}^{\prime\prime}-(v_{\max j}+\tilde V_0)+
  \frac{\bar\chi}{2}  v_{\max j}^2 e^{\bar\chi \tilde V_0}=0\,, \qquad
  \infty<z<+\infty\,.
\end{align}
In this $j^{\mbox{th}}$ sub-inner region, we expand $\tilde V_0(z)$ as
$\tilde V_0=\tilde V_{00}(z)+o(1)$. From
(\ref{innerinerleadingbefore}), and assuming that $v_{\max j}\gg 1$,
we obtain an explicitly solvable leading order sub-inner equation
\begin{align}\label{tiltev00}
  \tilde V_{00}^{\prime\prime} +\frac{1}{2}\bar\chi e^{\bar\chi \tilde V_{00}}=0 \,,
  \qquad \mbox{so that} \qquad
    \tilde V_{00}=\frac{1}{\bar\chi}
  \log\left[\sech^2\left(\frac{\bar\chi}{2}z\right)
  \right]\,, \quad U_0=C_je^{\bar\chi V_0}\sim
C_j e^{\bar\chi(v_{\max j}+\tilde V_{00}) }\,.
\end{align}

We now summarize our results in the inner and sub-inner regions for
the leading order profile of a quasi-equilibrium spike when $s_j\ll 1$
and $v_{\max j}\gg 1$. In the sub-inner region, where
$\vert x-x_j\vert \leq {\mathcal O}\left(\frac{\epsilon}{v_{\max
      j}}\right)$, we have
\begin{align}\label{innersolution1}
  u\sim  U_0\sim \frac{1}{2}\bar\chi v_{\max j}^2e^{\bar\chi\tilde V_{00}(z) }\,, \qquad
  v\sim V_0\sim  v_{\max j}+\tilde V_{00}(z)\,,
\end{align}
where $z=v_{\max j}\epsilon^{-1}(x-x_j)$ and $\tilde V_{00}$ is given
by (\ref{tiltev00}). In the inner region, where $\vert x-x_j\vert\leq
{\mathcal O}(\epsilon)$, we have
\begin{align}\label{innersol2}
u\sim C_je^{\bar\chi V_0(y)}\,, \qquad v\sim V_0(y)\,,
\end{align}
where $y=\epsilon^{-1}(x-x_j)$ and $V_0(y)$ is determined implicitly
by (\ref{implicit}). Here, the three constants $C_j$, $s_j$ and
$v_{\max j}$ satisfy the two nonlinear algebraic equations
$C_je^{\bar \chi s_j}=s_j$ and (\ref{keyrelation}). The required
additional equation arises below from matching the far-field of each
inner solution to an outer solution.  The far-field of the
leading order inner solution (\ref{innersol2}) gives only that
$u\sim v\rightarrow s_j$ as $\vert y\vert \rightarrow\infty$, but has
no gradient information.

As such, we must refine our inner analysis to one higher order so as
to match with any spatial gradients of the outer solution at the spike
locations.  To this end, we substitute (\ref{expansion}) into
(\ref{ss1}) and obtain at next order that
\begin{subequations}\label{nextorder}
\begin{align}
 &      \left[U_{1j}^{\prime}-\bar\chi (U_0V_{1j}^{\prime})-
         \bar\chi (U_{1j}V_0^{\prime})\right]^{\prime} =
  \frac{\mu}{d_1}(U_0^2-\bar uU_0)\,,\quad -\infty<y<\infty\,;
       \qquad U_{1j}^{\prime}(0)=0\,, \label{nextorder_a}\\
 &  V_{1j}^{\prime\prime}-V_{1j}+U_{1j}=0\,, \quad -\infty<y<\infty\,;\qquad
                       V_{1j}^{\prime}(0)=0\,.\label{nextorder_b}
\end{align}
\end{subequations}
Upon integrating (\ref{nextorder_a}) over $(0,y)$, we obtain the
flux-balancing condition
\begin{align}\label{fluxbalancing}
  U_{1j}^{\prime}-\bar\chi U_{1j}V_0^{\prime}-\bar\chi U_0V_{1j}^{\prime}=
  \frac{\mu}{d_1}\int_0^{y}\left[ U_0^2(\xi)-\bar u U_0(\xi)\right]\, d\xi\,.
\end{align}
By letting $y\to \infty$, and assuming that $U_{1j}V_{0}^{\prime}$ and
$U_0V_{1j}^{\prime}$ are negligible in this limit, we obtain from
(\ref{fluxbalancing}) that
\begin{align}\label{u1farfield}
U_{1j}^{\prime}\rightarrow  \frac{\mu}{d_1}\int_0^\infty (U_0^2-\bar u U_0) \, dy\,,
\qquad \mbox{as} \quad y\to + \infty \,.
\end{align}

To explicitly determine $U_{1j}^{\prime}$ as $y\rightarrow +\infty,$ we must
estimate the two integrals in (\ref{u1farfield}) involving $U_0$ and $U_0^2.$
By using the sub-inner solution (\ref{innersolution1}), we readily calculate
that
\begin{align}\label{u1farfieldafter}
  \int_0^\infty U_0 \, dy\sim v_{\max j} \,, \qquad \int_0^\infty U_0^2 \, dy
  \sim \frac{1}{3}\bar\chi v_{\max j}^3 \,.
\end{align}
In this way, by substituting (\ref{u1farfieldafter}) into
(\ref{u1farfield}), we obtain that
$U_{1j}^{\prime}\rightarrow {\mu \bar\chi v_{\max j}^3/(3d_1)}$ as
$y\rightarrow +\infty$. In a similar way, we obtain from
(\ref{nextorder}) that
$U_{1j}^{\prime}\rightarrow -{\mu \bar\chi v_{\max j}^3/(3d_1)}$ as
$y\to-\infty$. Since the inner solution is expanded as
$u\sim U_0+\epsilon^2 U_{1j}$, and $U_{0y}$ is exponentially small as
$|y|\to\infty$, we obtain for the outer solution that
$u_{x}\sim \epsilon U_{ij}^{\prime}$ as $x\to x_j$ and
$|y|\to \infty$. By using the expressions above for $U_{1j}^{\prime}$ as
$y\to\pm \infty$, we conclude by matching the far-field behavaior of the
inner solution to the outer solution that the outer solution must satisfy
the limiting behavior
\begin{align}\label{jumpx}
u_x\sim \pm  \epsilon\frac{\mu \bar\chi}{3 d_1} v_{\max j}^3\, \quad
\mbox{as} \quad x\rightarrow x_{j}^{\pm} \,.
\end{align}
This matching condition shows that, in the outer region, $u_x$ must
have a jump discontinuity across each $x=x_j$.

Finally, we must confirm, through a self-consistency argument, that
$U_{1j}V_{0}^{\prime}$ and $U_0V_{1j }^{\prime}$ for $y\to \infty$ can
be neglected in (\ref{fluxbalancing}). To do so, we observe that,
although $U_{1j}$ grows linearly for $\vert y\vert$ sufficiently
large, the exponential decay of $V_{0}^{\prime}$ ensures that
$U_{1j} V_{0}^{\prime}$ can be neglected as $y\to\infty$. Moreover,
since $u\sim v$ in the outer region when $\epsilon\ll 1$, we obtain
that $U_0\sim V_0$ and $U_{1j}\sim V_{1j}$ as
$y\rightarrow \pm\infty$.  Combining these estimates with
$U_0\rightarrow s_j$ as $\vert y\vert\rightarrow +\infty,$ we obtain
$U_0V_{1j}^{\prime}\approx s_j U_{1j}^{\prime}\ll U_{1j}^{\prime}$ in
(\ref{fluxbalancing}) since $s_j\ll 1$. As a result, our assumptions
that $U_{1j}V_{0}^{\prime}$ and $U_0V_{1j}^{\prime}$ can be neglected
in (\ref{fluxbalancing}) as $|y|\to \infty$ are self-consistent.


\subsection{Outer Solution and Matching}
Next, we construct the solution in the outer region.  When
$\epsilon\ll 1$, we expand $u$ and $v$ as $u=u_o+o(1)$ and
$v=v_o+o(1)$.  From (\ref{ss1}) for $v$ we get $v_o=u_o$, so
that (\ref{ss1}) for $u$ reduces to
\begin{align}\label{outereq}
  u_{oxx}-\bar\chi(u_o u_{ox})_x+\frac{\mu}{d_1}u_o(\bar u-u_o)=0\,,
  \quad x\in (-1,1)\backslash \bigcup\limits_{j=1}^N x_j \,.
\end{align}
There are two ways to approximate the solution to (\ref{outereq}).
The first approach is to introduce the new variable
\begin{align}\label{w:wdef}
  w:=u_o-\frac{\bar \chi}{2}u_o^2 \,.
\end{align}
In terms of $w$, we obtain from (\ref{outereq}) that $w$ satisfies
\begin{align}\label{wouterproblembefore}
  w_{xx}+\frac{\mu}{d_1}\bigg[\frac{\bar u}{\bar\chi}-\frac{2}{\bar\chi^2}+
  \bigg(\frac{2}{\bar\chi^2}-\frac{\bar u}{\bar\chi}\bigg)
  \sqrt{1-2w\bar\chi}+\frac{2w}{\bar\chi}\bigg]=0\,, \qquad
x\in (-1,1)\backslash \bigcup\limits_{j=1}^N x_j \,.
\end{align}
We have $\sqrt{1-2\bar\chi w}\sim 1-\bar\chi w$ since $u_o$ is small
in the outer region. With this approximation, and setting
$w\sim w_o$, we obtain from (\ref{w:wdef}) and
(\ref{wouterproblembefore}) that $w_o\sim u_o$ in the outer region, where
$w_o$ solves the leading order problem
\begin{align}\label{outerproblem1}
  w_{oxx}+\frac{\bar u \, \mu }{d_1}w_o=0\,, \qquad x\in (-1,1)\backslash
  \bigcup\limits_{j=1}^N x_j\,.
\end{align}
Observe that (\ref{outerproblem1}) follows {\em exactly} from
(\ref{wouterproblembefore}), with no approximation, for the special
parameter set $\bar{u}={2/\bar{\chi}}$.  The second way to approximate
(\ref{outereq}) is to collect the leading order terms in
(\ref{outereq}) directly.  In fact, since $u_o$ is small,
$(u_ou_{ox})_x$ and $u_o^2$ are higher order terms in the outer
region.  By neglecting these terms in (\ref{outereq}), we also obtain
(\ref{outerproblem1}) since $u_o\sim w_o.$ Finally, for
(\ref{outerproblem1}), we require from (\ref{jumpx}) that $w_{o}$ must
satisfy the jump condition
$w_{ox}(x_j^{+})-w_{ox}(x_j^{-})= \frac{2\bar\chi \mu}{3d_1}v_{\max
  j}^3\epsilon$ across each $x_j$. In this way, we obtain the
leading order outer problem
\begin{align}\label{outerproblem}
     {\mathcal L}_0 w_o := \frac{d_1}{\mu}w_{oxx}+\bar u w_o=\frac{2\bar\chi}{3 }
  \epsilon\sum\limits_{k=1}^N v_{\max k}^3\delta(x-x_k)\,, \quad -1<x<1\,;
  \qquad w_{ox}(\pm 1)=0\,.
\end{align}

To analyze the solvability of (\ref{outerproblem}), we first observe that
(\ref{outerproblem}) admits the nontrivial homogeneous solution
\begin{equation}\label{prop:w0h}
  w_{oh}(x):=\cos\left( \frac{m(x+1)\pi}{2} \right) \,, \qquad 
\mbox{when} \quad d_1 = d_{1Tm} := \frac{4\mu \bar u }{m^2\pi^2}\,,
\quad \mbox{for} \,\, m=1,2,\ldots \,.
\end{equation}
As shown in Appendix \ref{app:Turing}, the interpretation of these
critical, or {\em resonant}, values of $d_1$ are that they correspond
precisely to where there is a bifurcation from the spatially uniform
solution $v=u=0$ for (\ref{timedependent}) on $|x|<1$. This trivial
solution for (\ref{timedependent}) on $|x|<1$ is linearly stable only
when $d_1> {4\mu \bar{u}/\pi^2}$.  When $d_1=d_{1Tm}$, there is a
solution (non-unique) to (\ref{outerproblem}) only if a compatibility
condition is satisfied. However, as shown in Appendix \ref{app:Turing} this
condition is automatically satisfied for an $N$-spike steady-state
solution.

To solve (\ref{outerproblem}) when $d_1\neq d_{1Tm}$, we introduce the
Helmholtz Green's function $G(x;x_k)$ satisfying
\begin{align}\label{greenequation}
\frac{d_1}{\mu}G_{xx}+\bar u G=\delta(x-x_k)\,,\quad -1<x<1\,; \qquad
G_x(\pm 1;x_k)=0\,.
\end{align}
For $d_1\neq d_{1Tm}$, the explicit solution to (\ref{greenequation}) is
\begin{align}\label{greenfunction}
  G(x;x_k)=\sqrt{\frac{\mu}{\bar ud_1}}\left[\tan(\theta(1+x_k))+
  \tan(\theta(1-x_k))\right]^{-1} \, \left\{\begin{array}{ll}
\frac{\cos(\theta(1+x))}{\cos(\theta(1+x_k)}\,, &-1<x<x_k\,,\\
\frac{\cos(\theta(1-x))}{\cos(\theta(1-x_k))}\,,&x_k<x<1\,, 
\end{array}
\right. \, \qquad  \theta:=\sqrt{\frac{\mu\bar u}{d_1}}\,.
\end{align}
In terms of (\ref{greenfunction}), the solution to (\ref{outerproblem})
when  $d_1\neq d_{1Tm}$ is
\begin{align}\label{woouter}
 u_o\sim w_o=\frac{2\bar\chi}{3 }\epsilon\sum_{k=1}^N v_{\max k}^3G(x;x_k)\,.
\end{align}

Our final step in the construction is to match the inner and outer solutions
to obtain the third algebraic equation needed to determine
$s_j$, $C_j$ and $v_{\max j}$.  Since $w_o\sim u_o$ in the outer
region, we impose $w_o(x_j)=s_j$ to get
\begin{align}\label{quasisj}
  s_j=\frac{2\bar\chi}{3 }\epsilon \sum_{k=1}^N v_{\max k}^3G(x_j;x_k) \,, \qquad
   j=1,\ldots,N\,,
\end{align}
when $d_1\neq d_{1Tm}$.  Combining (\ref{quasisj}), (\ref{keyrelation}), and
$C_je^{\bar \chi s_j}=s_j$, we obtain the following coupled algebraic system:
\begin{align}\label{algebraicreal}
C_je^{\bar \chi s_j}-s_j=0\,, \qquad
         -\frac{1}{2} v_{\max j}^2+\frac{1}{2}s^2_j+\frac{C_j}{\bar\chi}
         e^{\bar\chi v_{\max j}}-\frac{s_j}{\bar\chi}=0\,, \qquad
s_j=\frac{2\bar\chi}{3 }\epsilon \sum\limits_{k=1}^N v_{\max k}^3G(x_j;x_k)\,.
\end{align}

Finally, we observe that the matching condition (\ref{jumpx})
between the inner and outer solutions holds only when the spike
locations $x_j$ are equally spaced, and are given by $x_j=x_j^0$ where
\begin{align}\label{locations}
x_j^0:=-1+\frac{2j-1}{N}\,, \qquad j=1,\ldots,N\,.
\end{align}
Moreover, in (\ref{realag}) of Appendix \ref{appendixA} we calculate
$\sum\limits_{k=1}^NG(x_j;x_k)$ explicitly to show that it is
independent of $j$ when $x_j=x_j^0$, $x_k=x_k^{0}$, and
$d_1\neq d_{1Tm}$. As a result, for equally-spaced spikes we have
$s_j=s_0$, where $s_0$ is given by
\begin{align}\label{ag}
  s_0=\frac{2\bar\chi}{3 }a_gv_{\max 0}^3 \epsilon \sim \epsilon a_g \,
  \int_{-\infty}^{\infty} U_0^2 \, dy \,, \qquad \mbox{with}\quad
  a_g:=\sum_{k=1}^N G(x_j^0;x_k^0) =\frac{1}{2} \sqrt{\frac{\mu}{d_1\bar{\mu}}}
  \cot\left(\frac{\theta}{N}\right)\,.  
\end{align}
When $x_j=x_j^0$ and $d_1\neq d_{1Tm}$, our $N$-spike
quasi-equilibrium is the approximation to a true steady-state solution
of (\ref{ss1}). Setting $s_j=s_0$ for all $j$, we obtain from
(\ref{algebraicreal}) that $C_j=C_0$ and $v_{\max j}=v_{\max 0}$ for
all $j$, satisfy
\begin{align}\label{algebraic2}
C_0=s_0e^{-\bar \chi s_0}\,, \qquad
  -\frac{1}{2} v_{\max 0}^2+\frac{1}{2}s^2_{0}+
  \frac{C_0}{\bar\chi}e^{\bar\chi v_{\max 0}}-\frac{s_0}{\bar\chi}=0\,, \quad
\mbox{where} \quad s_0=\frac{2\bar\chi}{3 }\epsilon v_{\max 0}^3 a_g \,,
\end{align}
with $a_g$ as given in (\ref{ag}).  By combining (\ref{algebraic2}),
we obtain a single nonlinear equation for $v_{\max 0}$ given by
\begin{align}\label{singlevmax}
  -\frac{1}{2} v_{\max 0}^2+\frac{2}{9}\bar\chi^2v_{\max 0}^6a_g^2\epsilon^2+
  \frac{2}{3}a_gv_{\max 0}^3\epsilon
  e^{\bar\chi v_{\max 0}-\frac{2}{3}a_g\bar\chi^2v_{\max 0}^3\epsilon}-{\frac{2}{3 }
  a_gv_{\max 0}^3 \epsilon}=0 \,.
\end{align}
In terms of the solution $v_{\max 0}$ to (\ref{singlevmax}), $s_0$ and
$C_0$ are given by (\ref{algebraic2}).  Moreover, assuming
$v_{\max 0}\gg 1$, $v_{\max 0}^3\epsilon\ll 1$, and $a_g>0$, a
dominant balance argument on (\ref{singlevmax}) for $\epsilon\ll 1$
yields that
\begin{equation}\label{vm:dominant}
    \frac{1}{2}v_{\max 0}^2 \sim \frac{s_0}{\bar{\chi}} e^{\bar{\chi}v_{\max 0}} \sim
  \frac{2}{3}a_g v_{\max 0}^3 \epsilon e^{\bar{\chi}v_{\max 0}} \,.
\end{equation}
This shows that $v_{\max 0}={\mathcal O}(-\log\eps)
\gg 1$, so that the consistency condition $v_{\max 0}^3\epsilon\ll 1$ is
satisfied. We summarize our results regarding the
construction of the $N$-spike steady-state in the following formal
proposition:

\begin{proposition}\label{prop1}
  Let $\epsilon\ll 1$, assume that $d_1\neq d_{1Tm}$, where $d_{1Tm}$
  is defined in (\ref{prop:w0h}). Label the set
  $\mathcal I:=\{1,2,\ldots,N\}$. Then, the $N$-spike
  quasi-equilibrium to (\ref{ss1}), defined by $(u_{q},v_q)$, has the
  following asymptotic behavior in $-1<x<1$:
\begin{align}\label{Nspikebutquasi}
\left\{\begin{array}{ll}
u_q(x)&\sim\left\{\begin{array}{ll}
       \frac{1}{2}\bar\chi v_{\max  k}^2\sech^2\Big(\frac{\bar\chi}{2}
       \frac{v_{\max k}(x-x_k)}{\epsilon}\Big)\,,&x\in\Big\{x\in\mathbb R\big|\,
    |x-x_k|\leq {\mathcal O}\Big(\frac{\epsilon}{\vert\log\epsilon\vert}\Big)\,,
                 ~\exists k\in\mathcal I\Big\}\,,\\
    C_ke^{\bar\chi V_0\big(\frac{x-x_k}{\epsilon}\big)}\,,&x\in\Big\{x\in\mathbb R\big|\,
    {\mathcal O}\Big(\frac{\epsilon}{\vert \log\epsilon\vert}\Big)\ll
     |x-x_k|\leq {\mathcal O}(\epsilon),~\exists k\in\mathcal I\Big\}\,,\\
    \frac{2\bar\chi\epsilon}{3 }\sum\limits_{k=1}^N v_{\max k}^3G(x;x_k)\,,&
  x\in\Big\{x\in\mathbb R\big|\, {\mathcal O}(\epsilon)\ll\vert x-x_k\vert\,,
        ~\forall  k\in\mathcal I\Big\}\,,\\
\end{array}
\right.\\
\\
v_q(x)&\sim \left\{\begin{array}{ll}
  v_{\max k}+\frac{1}{\bar\chi}\log\Big[\sech^2\Big(\frac{\bar\chi}{2}
     \frac{v_{\max k}(x-x_k)}{\epsilon}\Big)\Big]\,,&x\in\Big\{x\in\mathbb R
    \big|\, |x-x_k|\leq {\mathcal O}\Big(\frac{\epsilon}{\vert
       \log\epsilon\vert}\Big),~\exists k\in\mathcal I\Big\}\,,\\
     V_0\big(\frac{x-x_k}{\epsilon}\big)\,,& x\in\Big\{x\in\mathbb R\big|\,
   {\mathcal O}\big(\frac{\epsilon}{\log\epsilon}\big)\ll\vert x-x_k\vert\leq
   {\mathcal O}(\epsilon)\,,~\exists  k\in\mathcal I\Big\}\,,\\
   \frac{2\bar\chi\epsilon}{3 }\sum\limits_{k=1}^Nv_{\max k}^3G(x;x_k)\,,&
  x\in\Big\{x\in\mathbb R\big| \, {\mathcal O}(\epsilon)\ll\vert
   x-x_k\vert\,,~\forall  k\in\mathcal I\Big\}\,.
\end{array}
\right.
\end{array}
\right.
\end{align}
Here $\bar\chi={\chi/d_1}$, $G(x;x_k)$ is defined by
(\ref{greenfunction}) and $V_0$ is given implicitly by
(\ref{implicit}).  Moreover, the constants $v_{\max j}$, $s_j$ and
$C_j$ are determined by (\ref{algebraicreal}).  When $x_j=x_j^0$, as
given in (\ref{locations}), the spikes are equally spaced and
$(u_q,v_q)$ becomes an approximation to the true $N$-spike equilibrium
solution $(u_e,v_e)$ to (\ref{timedependent}), in which
\begin{equation*}
 v_{\max j}=v_{\max 0}\,,\quad s_j=s_0\,, \quad C_j=C_0, \quad
 \mbox{for} \quad j=1,\ldots,N\,.
\end{equation*}
In terms of the solution $v_{\max 0}$ to (\ref{singlevmax}), $s_0$ and
$C_0$ are given by (\ref{algebraic2}) where $a_g$ is defined in
(\ref{ag}).
\end{proposition}

When $d_1=d_{1Tm}$, for $m=1,\ldots,N-1$, we show in Appendix
\ref{app:Turing} that for a steady-state solution where the spike
locations $x_j$ satisfy (\ref{locations}), the outer problem
(\ref{outerproblem}) has a non-unique solution that can be found using
a generalized Green's function.  Finally, to establish the range of
$d_1$ where our steady-state analysis is valid we must also ensure
that the outer solution $w_o$ is positive on $|x|<1$.  This
constraint, discussed in Appendix \ref{app:Turing}, motivates the
following key remark that introduces the notion of an {\em admissible set}
${\mathcal T}_e$ for $d_1$.

\begin{remark}\label{remark:pos}
  For an $N$-spike steady-state solution, where the spikes are
  centered at (\ref{locations}), the range of $d_1$ where
  (\ref{outerproblem}) has a unique and positive solution $w_o$ is
  characterized by an admissible set $\mathcal{T}_e$, which we
  define by
  \begin{equation}\label{d1:admiss}
    {\mathcal T}_e :=\lbrace{ \,\,  d_1 \,\, \vert \,\, d_1> d_{1pN}:=
      \frac{4\mu \bar u}{N^2\pi^2}\,, \quad d_1\neq 
        d_{1Tm} := \frac{4\mu \bar u}{m^2\pi^2}\,, \quad m=1,\ldots,N-1
          \rbrace}\,.
  \end{equation}
  For $d_1>d_{1pN}$, we have $a_g>0$ in (\ref{ag}), so that
  $s_0>0$. As $d_1\to d_{1pN}$ from above, $a_g\to 0^{+}$ and
  $v_{\max 0}\to +\infty$.  Moreover, when $d_1\in \mathcal{T}_e$, the
  outer solution $w_o$ on the interval of width ${2/N}$ between two
  adjacent spikes, which is asymptotically close to the uniform state
  $u=0$, is linearly stable.  At the positivity threshold
  $d_1=d_{1pN}$, the trivial solution $u=0$ on a domain of length
  ${2/N}$ undergoes a Turing instability and this threshold appears to
  trigger a nonlinear spike nucleation event for (\ref{timedependent})
  between adjacent spikes (see Figure \ref{figurenucleation} in \S
  \ref{sec:small_thresh} below). In contrast, for an $N$-spike
  quasi-equilibrium pattern, the outer solution $w_o$ between spikes
  is positive when
  \begin{equation}\label{d1:qe}
    d_1>\frac{L_{max}^2\mu \bar u}{\pi^2} \,, \qquad
    L_{max} :=\max\lbrace{  |x_1+1|\,; \, |x_{N}-1|\,; \,
      |x_{j+1}-x_j|\,, \,\, j=1,\ldots,N \rbrace}\,.
  \end{equation}
\end{remark}

\subsection{Global Balancing and Comparison with Numerics}

As an analytical confirmation of our asymptotic results, we show that
they are consistent with a global balancing condition. By integrating 
(\ref{ss1})  for $u$ over $|x|\leq 1$, we obtain that the global
balance condition
\begin{align}\label{globalintegral}
\int_{-1}^1 u(\bar u-u)\, dx=0\,,
\end{align}
must hold. Defining $f(u):=u(\bar u-u)$, we decompose the left-hand side of
(\ref{globalintegral}) into the two terms
\begin{align}\label{globalintegral1}
  \int_{-1}^1 u(\bar u-u)\,dx=\overbrace{\int_{-1}^1 [f(u)-f(s_0)]\,dx}^{I_1}+
  \overbrace{\int_{-1}^1 f(s_0)\,dx}^{I_2} \,.
\end{align}
Since the inner and outer regions both contribute to $I_1,$ we
decompose $I_1$ as $I_{11}+I_{12}$, where $I_{11}$ and $I_{12}$
represent the $N$ inner integrals and the outer integral,
respectively.  For $I_{11}$, since $u\rightarrow s_0$ as
$y\rightarrow \pm \infty,$ we have $f(u)-f(s_0)\rightarrow 0$ as
$y\rightarrow \pm \infty.$ Therefore, by using (\ref{V0prime}) and
since there are $N$ identical inner regions, we identify that
\begin{align}\label{I11inner}
  I_{11}\sim 2N \epsilon\int_{0}^\infty[f(U_0)-f(s_0)]\, dy=2N \epsilon
  \int_{s_0}^{v_{\max 0}}\frac{f(C_0e^{{\bar{\chi}} \xi})-f(s_0)}
  {\sqrt{-2K(\xi;C_0)}}\, d\xi \,,
\end{align}
where $K$ is given by (\ref{firstintegral}). However, to estimate
(\ref{I11inner}) we can more simply use the fact that $U_0\gg 1$ in
each sub-inner region.  In this way, by using (\ref{u1farfieldafter}),
we obtain that
\begin{align}\label{globalbalance1}
  I_{11}\sim &2N \epsilon \int_{0}^\infty (\bar uU_0- U_0^2) \, dy\sim
               2N \bar u v_{\max 0}\epsilon-\frac{2}{3}N \bar
               \chi v_{\max 0}^3\epsilon \,.
\end{align}
Next, by using the outer solution (\ref{woouter}), together with
$\int_{-1}^{1}G(x;x_k^{0})\, dx={1/\bar{u}}$, we estimate the
outer integral as
\begin{align}\label{globalbalance2}
I_{12}\sim &\int_{-1}^1f(w_o)\, dx-\int_{-1}^1f(s_0)\, dx\,, \nonumber\\
  =&\frac{2}{3} \bar\chi\bar u v_{\max 0}^3\epsilon\sum_{k=1}^N
     \int_{-1}^1G(x;x_k^0)\, dx-\Big(\frac{2\bar\chi\bar u}{3 }v_{\max 0}^3
     \epsilon\Big)^2\int_{-1}^1\Big[\sum_{k=1}^NG(x;x_k^0)\Big]^2\, dx
     -\int_{-1}^1f(s_0)\, dx\,, \nonumber\\
  =&\frac{2}{3} N \bar\chi v_{\max 0}^3 \epsilon
        -\int_{-1}^1f(s_0)\, dx +  {\mathcal O}(\epsilon^2 v_{\max 0}^{6})\,.
\end{align}
We substitute (\ref{globalbalance1}) and (\ref{globalbalance2}) into
(\ref{globalintegral1}) to find
$\int_{-1}^1 u(\bar u-u) \, dx = {\mathcal O} \left(\epsilon v_{\max
    0}\right)\ll 1$, and so the global balancing condition is satisfied
to this order as $\epsilon\to 0$.

\begin{table}[htbp]
    \centering
    \begin{tabular}{|c|c|c|c|c|c|c|}
    \hline
        $d_2$& $d_1=\chi$ &$\bar u$ &$u_{\max}$(num) &$u_{\max}$(asy)& $v_{\max}$(num) &$v_{\max}$(asy) \\\hline
         0.02& 1 &2&3.8935& 3.4633& 2.6937& 2.6318\\\hline
        0.004& 1 &2& 5.2575& 5.0329&3.1702&3.1727\\\hline
         0.002& 1 &2&5.9773&5.8239&3.3955&3.4129\\\hline
             0.02& 10 &2&3.8599&3.1702&2.6623&2.5180\\\hline
        0.004& 10 &2 &5.0958&4.6664&3.1099 &3.0550\\\hline
         0.002& 10 &2&5.7514&5.4210&3.3218&3.2927\\\hline
             0.02& 1 &3&5.9159&4.4409&3.3970&2.9802\\\hline
        0.004& 1 &3&7.3629&6.2531& 3.7971&3.5364 \\\hline
         0.002& 1 &3&8.1535&7.1617&4.0023&3.7846\\\hline
    \end{tabular}
    \caption{\fontsize{11pt}{9pt}\selectfont {
      The asymptotic results for $u_{\max}$ and
      $v_{\max}$, obtained from (\ref{Nspikebutquasi}), for various $d_2$,
      $d_1$ and $\bar u$ are compared with FlexPDE7 numerical results.}}
    \label{tab1}
  \end{table}

\begin{table}[htbp]
    \centering
    \begin{tabular}{|c|c|c|c|c|c|c|}
    \hline
        $d_2$& $d_1=\chi$ &$\bar u$ &$u_{\text{bdry}}$(num) &$u_{\text{bdry}}$(asy)& $v_{\text{bdry}}$(num) &$v_{\text{bdry}}$(asy) \\\hline
         0.02& 1 &2&0.4799& 0.5195& 0.5047& 0.5195\\\hline
        0.004& 1 &2& 0.3744& 0.3923&0.3734&0.3923\\\hline
         0.002& 1 &2&0.3340&0.3412&0.3336&0.3412\\\hline
             0.02& 10 &2&0.4295&0.3824&0.4567&0.3824\\\hline
        0.004& 10 &2 &0.3166&0.3047&0.3166 &0.3047\\\hline
         0.002& 10 &2&0.2790&0.2695&0.2790&0.2695\\\hline
             0.02& 1 &3&0.3350&0.5538&0.3537&0.5538\\\hline
        0.004& 1 &3&0.2878&0.3883& 0.2867&0.3883 \\\hline
         0.002& 1 &3&0.2627&0.3305&0.2622&0.3305\\\hline
    \end{tabular}
    \caption{\fontsize{11pt}{9pt}\selectfont {The asymptotic results
        for $u_{\text{bdry}}$ and $v_{\text{bdry}}$, obtained from
        (\ref{Nspikebutquasi}), for various $d_2$, $d_1$ and $\bar u$
        are compared with FlexPDE7 \cite{flex2021} numerical results.}}
    \label{tab2}
\end{table}

\begin{figure}[htbp]
\centering
\begin{subfigure}[t]{0.5\textwidth}
    \includegraphics[width=1.0\linewidth,height=7.0cm]{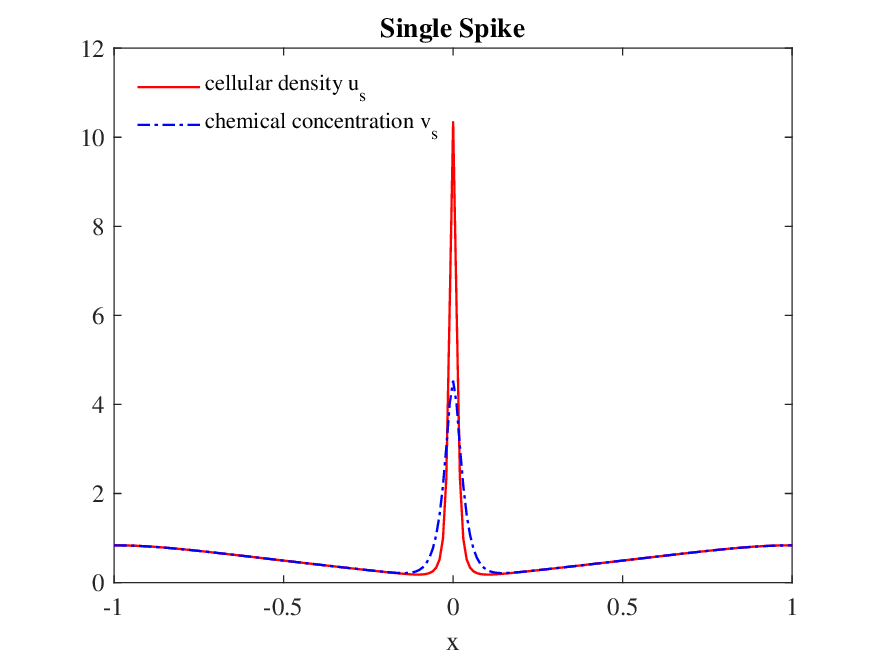}
         \caption{a one-spike steady-state}
\end{subfigure}\hspace{-0.25in}
\begin{subfigure}[t]{0.5\textwidth}
  \includegraphics[width=1.0\linewidth,height=7.0cm]{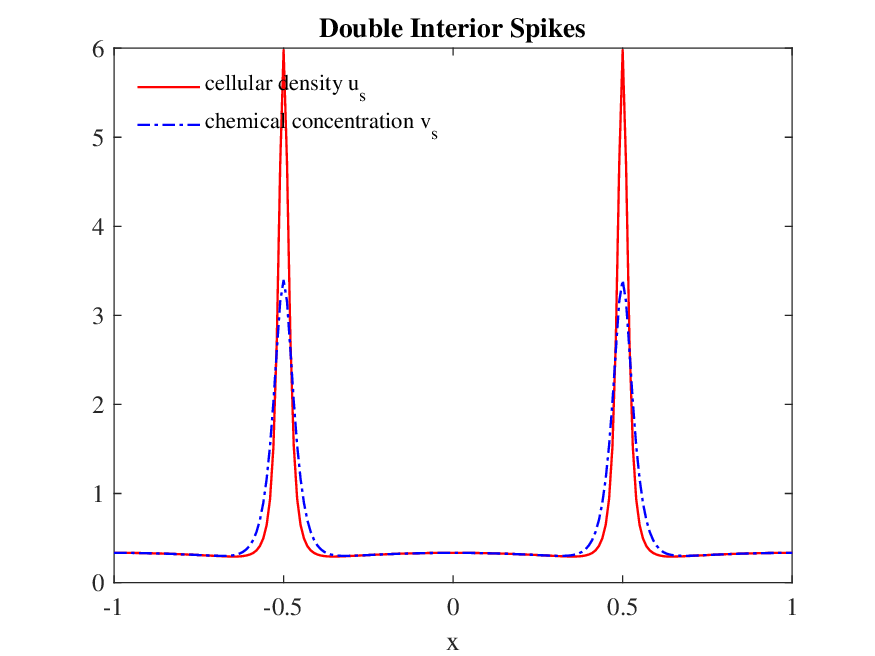}
    \caption{a two-spike steady-state}
\end{subfigure}
\\ \vspace{-0.15in}
\caption{\fontsize{11pt}{9pt}\selectfont { {\em Numerically-computed
      one and two-spike steady-state solutions of (\ref{ss1}) with
      $d_1=\chi=1$, $\bar u=2,$ $d_2=0.0005$ using FlexPDE7
      \cite{flex2021}.}  The solid red curves are the cellular density
    $u$, while the dotted blue curves are the chemical concentration
    $v$. Observe that $u$ and $v$ increase in the outer region.}}
\label{equilibriumpic}
\vspace{-0.0in}
\end{figure}

  For a one-spike steady-state, we now compare our asymptotic results
  with corresponding full numerical results computed using FlexPDE7
  \cite{flex2021}.  For $\mu=0.25$, in Table \ref{tab1} we compare
  asymptotic and numerical results for the maximum values of $u$ and
  $v$ for both $\bar u=2$ and for $\bar{u}=3$. A similar comparison,
  but for the boundary values of $u$ and $v$ are shown in Table
  \ref{tab2}. We observe that the asymptotic results in
  (\ref{Nspikebutquasi}) more closely approximate the numerical result
  when $\bar{u}=2$ than when $\bar{u}=3$. This improved agreement when
  $\bar{u}=2$ is due to the fact that the $\sqrt{1-2 w_o\bar\chi}$
  term in \eqref{wouterproblembefore} vanishes only when
  $\bar{u}={2/\bar\chi}$, and so the error does not include any
  ${\mathcal O}(\vert\log\epsilon\vert^{-1} )$ correction term as it
  does for the case when $\bar{u}=3$. In Figure \ref{equilibriumpic}
  we plot the numerically-computed one-spike and two-spike
  steady-state solutions computed using FlexPDE7 \cite{flex2021}. We
  observe that the half-profiles of $u$ and $v$ are not monotone
  decreasing and so their spatial behavior is rather different than
  for spike patterns of the classical KS model
  \cite{wang2013spiky} without the logistic growth term.

\subsection{Formulation of the Linear Stability Problem}
To formulate the linear stability problem for the steady-state
solution, denoted by $(u_e,v_e)$, we introduce the following
time-dependent perturbation $(u,v)$ to (\ref{timedependent}):
\begin{align}\label{perturbation}
u(x,t)=u_e(x)+e^{\lambda t}\phi(x)\,,\qquad v(x,t)=v_e(x)+e^{\lambda t}\psi(x)\,,
\end{align}
where $\phi\ll 1$ and $\psi\ll 1.$ Upon substituting
(\ref{perturbation}) into (\ref{timedependent}) and linearizing, we
obtain the spectral problem
\begin{subequations}\label{lep}
\begin{align}
           {\frac{\tau\lambda}{d_1}}\phi&=\phi_{xx}-\bar\chi(u_e\psi_x)_x
           -\bar\chi(v_{ex}\phi)_x+
      \frac{\mu}{d_1}(\bar u-2u_e)\phi\,, \quad -1<x<1\,; \quad
         \phi_x(\pm 1)=0\,,\label{lep_a} \\
           \lambda\psi&=\epsilon^2\psi_{xx}-\psi+\phi\,, \quad -1<x<1\,;
           \quad \psi_x(\pm 1)=0\,.\label{lep_b}
\end{align}
\end{subequations}
It is well known that linearized eigenvalue problems arising from the
analysis of localized spike patterns of RD systems have two classes of
eigenvalues (cf.~\cite{iron2001stability}).  The first type is
referred to as the large eigenvalues since they are bounded away from
zero as $\epsilon\rightarrow 0.$ The second type are the small
eigenvalues of order $o(1)$ as $\epsilon\rightarrow 0.$

In \S \ref{sec3} and \S \ref{sec4} we will analyze the large and small
eigenvalues for (\ref{lep}), where the cellular diffusion rate $d_1$
is the main bifurcation parameter. Recall that $d_1\in \mathcal{T}_e$
where the admissible set is defined in (\ref{d1:admiss}). Our main
goal is to determine critical thresholds for $d_1\in \mathcal{T}_e$,
depending on $N$, that will provide the range of $d_1$ for which all
large and small eigenvalues satisfy $\mbox{Re}(\lambda)<0$. On this
range, $N$-spike steady-states are linearly stable as $\epsilon\to
0$. Oscillatory instabilities in the amplitude of a one-spike
steady-state are also shown to be possible as $\tau$ is increased from
a Hopf bifurcation of the large eigenvalues.
 
\section{Analysis of the Large Eigenvalues}\label{sec3}
This section is devoted to the study of large eigenvalues for an
$N$-spike steady-state. These eigenvalues are bounded away from zero
as $\epsilon\rightarrow0.$ To begin, we introduce local variables
defined in the $j^{\mbox{th}}$ inner region by
\begin{align}\label{largenewvariable}
  y=\epsilon^{-1}(x-x_j)\,, \qquad \Phi_j(y):=\phi(x_j+\epsilon  y)\,,
  \qquad \Psi_j(y):=\psi(x_j+\epsilon  y)\,,
\end{align}
and we expand
\begin{align}\label{largeexpansionphipsi}
  \Phi_j(y)=\Phi_{0j}(y)+\epsilon^2\Phi_{1j}(y)+\ldots\, \quad
  \Psi_j(y)=\Psi_{0j}(y)+\epsilon^2\Psi_{1j}(y)+\ldots\,, \quad
  \lambda\sim\lambda_0 \,.
\end{align}
Since the spike profile $(U_j,V_j)$ for the steady-state is the same
for each $j$, as similar to (\ref{expansion}) we expand
\begin{align}\label{largeexpansionUV}
  U_j(y)=U_{0}(y)+\epsilon^2U_{1}(y)+\ldots\, \quad V_j(y)=V_{0}(y)+
  \epsilon^2V_{1}(y)+\ldots \,.
\end{align}
Upon substituting (\ref{largenewvariable})--(\ref{largeexpansionUV})
into (\ref{lep}), we obtain the following leading order problem on
$-\infty<y<\infty$:
\begin{subequations}\label{leadingphilarge}
\begin{align}
         0&=\Phi_{0j}^{\prime\prime}-\bar\chi(U_0\Psi^{\prime}_{0j})^{\prime}-
  \bar\chi(V_0^{\prime}\Phi_{0j})^{\prime}\,; \qquad
\Phi_{0j}^{\prime}(0) =0 \,, \label{leadingphilarge_a}\\
  \lambda_0\Psi_{0j}&=\Psi_{0j}^{\prime\prime}-\Psi_{0j}+\Phi_{0j}\,;
                    \qquad  \Psi_{0j}^{\prime}(0)=0\,.
\end{align}
\end{subequations}
Recalling that $U_{0}^{\prime}=\bar \chi U_{0} V_{0}^{\prime}$ from
the core problem (\ref{leading}), it is convenient to define
$g_{0j}$ by
\begin{align}\label{sec3g0}
g_{0j}:=\frac{\Phi_{0j}}{U_0}-\bar\chi \, \Psi_{0j}\,.
\end{align}
In terms of $g_{0j}$, the two problems in (\ref{leadingphilarge}) are
transformed on $-\infty<y<\infty$ to
\begin{align}\label{leadingglarge}
  \left(U_0g_{0j}^{\prime}\right)^{\prime} =0 \,, \quad
  g_{0j}^{\prime}(0)=0\,; \qquad
     \lambda_0\Psi_{0j}=\Psi_{0j}^{\prime\prime}-\Psi_{0j}+\bar\chi U_0\Psi_{0j}+
         U_0g_{0j}\,, \quad \Psi_{0j}^{\prime}(0)=0\,.
\end{align}
Imposing that $g_{0j}$ is bounded as $|y|\to\infty$, we obtain from
the first equation of (\ref{leadingglarge}) that
$g_{0j}=\tilde C_j$, where $\tilde C_j$ is to be determined. Then,
the second equation in (\ref{leadingglarge}) becomes
\begin{align}\label{jthlargeep}
  \lambda_0\Psi_{0j}=\Psi_{0j}^{\prime\prime}-\Psi_{0j}+\bar \chi U_0\Psi_{0j}+
  \tilde C_j U_0\,, \quad -\infty<y<\infty\,; \qquad
  \Psi_{0j}^{\prime}(0)=0 \,.
\end{align}
Before formulating the outer problem, we must
determine the far-field behavior of the inner solution. In the outer region,
we obtain from (\ref{lep}) that, for $\epsilon\ll 1,$
$\phi\sim (\lambda_0+1)\psi$. As a result, we must have
$\Phi_{0j}\sim (\lambda_0+1)\Psi_{0j}$ as $y\rightarrow \pm\infty$.  By using
this relation, together with $g_{0j}=\tilde C_j$ and $U_0\sim s_0$ as $|y|
\to\infty$, (\ref{sec3g0}) yields that
$$\Phi_{0j}=\tilde C_jU_0+\bar\chi U_0\Psi_{0j}\sim \tilde C_j s_0+\bar\chi
s_0\frac{\Phi_{0j}}{\lambda_0+1}\,,$$
as $|y|\to\infty$. Since $s_0\ll 1$, this expression provides the
leading order far-field behavior
\begin{align}\label{largeinnerfarfieldbehavior}
\Phi_{0j}\sim \tilde C_j s_j \,, \qquad \mbox{as} \quad |y|\to \infty\,.
\end{align}

Next, we construct the outer solution. Since
$u_{e}=v_{e}={\mathcal O}(s_0)\ll 1$ in the outer region, 
(\ref{lep_a}) yields that $\phi\sim\phi_o$, where
\begin{equation}\label{phio:nj}
  \frac{d_1}{\mu} \phi_{oxx} + \hat{u} \phi_o =0 \,, \quad -1<x<1 \,, \quad
  x\neq x_{j}^{0}\,, \qquad \mbox{where} \quad
  \hat{u}= \bar u - \frac{\tau\lambda_0}{\mu} \,.
\end{equation}
From (\ref{largeinnerfarfieldbehavior}), one matching condition is
$\phi_0(x_{j}^{0})=\tilde{C}_js_0$, while the other is obtained by deriving
the appropriate
jump condition for $\lbrack \phi_{ox}\rbrack_{j}:=\phi_{ox}(x_j^{0+})-
\phi_{ox}(x_j^{0-})$. To derive this jump condition we write (\ref{lep_a})
as
\begin{equation}\label{hatu}
  \frac{d_1}{\mu} \phi_{xx} + \hat{u} \phi = 2 u_e\phi + \bar\chi(u_e\psi_x)_x
  +\bar\chi(v_{ex}\phi)_x \,.
\end{equation}
We integrate (\ref{hatu}) over an intermediate scale
$x_{j}^{0}-\sigma<x<x_{j}^{0}+\sigma$ where
$\epsilon\ll\sigma\ll 1$ and we pass to the limit $\sigma\to 0$, but
with ${\sigma/\epsilon}\to\infty$. In terms of the inner coordinate
$y$, where $u_e\sim U_0$ and $\phi\sim \Phi_{0j}$, and upon using
the facts that $u_e=v_e={\mathcal O}(s_0)\ll 1$ at
$x=x_{j}^{0}\pm\sigma$, we obtain that the jump condition for the outer
solution is
\begin{equation*}
  \frac{d_1}{\mu}  \lbrack \phi_{ox}\rbrack_j \sim
  2\epsilon \int_{-\infty}^{\infty} U_0\Phi_{0j}\,
  dy \,.
\end{equation*}
Then, by using $\Phi_{0j}\sim U_0\tilde{C}_j + \bar{\chi} U_0 \Psi_{0j}$,
as derived from (\ref{sec3g0}) with $g_{0j}=\tilde{C}_j$, we conclude
that
\begin{equation}\label{f:jump}
  \frac{d_1}{\mu}  \lbrack \phi_{ox}\rbrack_j \sim
    2\epsilon \int_{-\infty}^{\infty}\left(U_0^2 \tilde{C}_j
      + \bar{\chi} \Psi_{0j} U_0^2 \right)\, dy\,.
\end{equation}
In this way, we obtain the following multi-point boundary-value
problem (BVP) for the outer solution $\phi_o$:
\begin{subequations}\label{phio:prob}
\begin{align}
  \frac{d_1}{\mu} \phi_{oxx} + \hat{u} \phi_o &=0\,, \quad -1<x<1 \,,
  \,\,\, x\neq x_{j}^{0}\,, \,\, j=1,\ldots,N \,; \qquad \phi_{ox}(\pm 1) =0\,, \\
  \frac{d_1}{\mu} \lbrack \phi_{ox}\rbrack_j &=
  2\epsilon\int_{-\infty}^{\infty}\left(U_0^2 \tilde{C}_j
 + \bar{\chi} \Psi_{0j} U_0^2 \right)\, dy \,,
    \quad  \phi_0(x_{j}^{0})=\tilde{C_j} s_0 \,, \quad j=1,\ldots,N\,.
\end{align}
\end{subequations}

To solve (\ref{phio:prob}), we introduce an
eigenvalue-dependent Green's function $G_\lambda(x;x_k)$ defined by
\begin{align}\label{lam:greenequation}
  \frac{d_1}{\mu}G_{\lambda xx}+\hat u G_{\lambda} =\delta(x-x_k)\,,\quad
  -1<x<1\,; \qquad
G_{\lambda x}(\pm 1;x_k)=0\,,
\end{align}
which exists provided that $d_1\neq {4\mu \hat{u}/(m^2\pi^2)}$ for
$m=1,2,\ldots$. When these constraints are satisfied, the solution to
(\ref{phio:prob}) is represented as the superposition
\begin{align}\label{largeouter}
  \phi_o=  2\epsilon \sum_{k=1}^N \int_{-\infty}^\infty
  \left(\bar\chi  U^2_0
  \Psi_{0k} + \tilde C_k U^2_0 \right)\, dy \,\, G_{\lambda}(x;x_k^{0})  \,.
\end{align}
By imposing $\phi_o(x_{j}^{0})=s_0\tilde{C_j}$, and recalling from
(\ref{ag}) that $s_0=\epsilon a_g \int_{-\infty}^\infty U_0^2 \, dy$,
we obtain from (\ref{largeouter}) that
\begin{equation}\label{matchinglarge1}
  \tilde C_j=\frac{2}{ a_g\int_{-\infty}^\infty U_0^2 \, dy}
  \sum_{k=1}^N\left(\int_{-\infty}^\infty\bar\chi  U^2_0\Psi_{0k} \, dy\right)
  G_{\lambda}\left(x_j^0;x_k^0\right) +
    \frac{2}{a_g}\sum_{k=1}^N\tilde C_k\, G_{\lambda}\left(x_j^0;x_k^0\right)\,,
\end{equation}
where $a_{g}$ was defined in (\ref{ag}). Then, by letting $I$ be the
$N\times N$ identity matrix, and introducing
\begin{align}\label{nlep:glambda}
\mathcal G_{\lambda}:=\left( \begin{array}{ccc}
 G_{\lambda}\big(x_1^0;x_1^0\big) & \cdots &  G_{\lambda}\big(x_1^0;x_N^0\big)\\
  \vdots & \ddots & \ddots\\
   G_{\lambda}\big(x_N^0;x_1^0\big) &\cdots &  G_{\lambda}\big(x_N^0;x_N^0\big)\\
    \end{array}
  \right)\,,~~~~\boldsymbol{\tilde C}:=\left( \begin{array}{c}
\tilde C_1\\
  \vdots  \\
\tilde C_N\\
    \end{array}
  \right)\,,~~~\boldsymbol{\Psi_{0}}:=\left( \begin{array}{c}
\Psi_{01}\\
  \vdots  \\
\Psi_{0N}\\
    \end{array}
  \right)\,,
\end{align}
we can write the linear algebraic system (\ref{matchinglarge1}) for
$\tilde{C}_j$, with $j=1,\ldots,N$, in matrix form as
\begin{equation}\label{matform}
  \left(\frac{2}{a_{g}}{\mathcal G}_{\lambda} - I\right)\boldsymbol{\tilde C}
  = \frac{2}{a_{g}} {\mathcal G}_{\lambda} \left(
  -\frac{\bar{\chi}}{\int_{-\infty}^{\infty} U_0^2\, dy}
  \int_{-\infty}^{\infty} U_0^2 \boldsymbol{\Psi_0} \,dy \right)\,.
\end{equation}

By combining (\ref{jthlargeep}) with (\ref{matform}), we obtain a
vector nonlocal eigenvalue problem (NLEP) given by
\begin{align}\label{vectorpsi0lep}
  \boldsymbol{ \Psi}_{0}^{\prime\prime}-\boldsymbol{\Psi_0}+
  \bar{\chi} U_0\boldsymbol{\Psi_0}-\bar{\chi}U_0
  \frac{\int_{-\infty}^\infty U_0^2{\mathcal
  B}\boldsymbol{\Psi}_0 \, dy }{\int_{-\infty}^\infty U_0^2 \, dy}
  =\lambda_0\boldsymbol{\Psi}_0\,,
\end{align}
where, to leading order, we have $\boldsymbol{\Psi}_0\rightarrow 0$ as
$|y|\rightarrow \infty$. Here $\mathcal B$ is the $N\times N$
matrix defined by
\begin{align}\label{mathcalB}
  { \mathcal B}:=\frac{2}{a_g}\left(\frac{2}{a_{g}}
  \mathcal G_{\lambda}
  -I\right)^{-1}\mathcal{G}_\lambda \,.
\end{align}

Next, we diagonalize the vector NLEP (\ref{vectorpsi0lep}) by
introducing the orthogonal eigenspace of ${\mathcal B}$ as
\begin{equation}\label{matrixBeigenvalueproblem}
  {\mathcal B} {\bf q}_j=\alpha_j {\bf q}_j \,, \qquad j=1,\ldots,N\,.
\end{equation}
Denoting ${\mathcal Q}$ as the matrix of eigenvectors ${\bf q}_j$ (as
columns), we obtain ${\mathcal B}={\mathcal Q}\,\mbox{diag}(\alpha_1,\ldots,
\alpha_N) {\mathcal Q}^{-1}$. By defining
$\tilde{\boldsymbol{\Psi}}_0={\mathcal Q}^{-1}{\boldsymbol{\Psi}}_0$,
we obtain that (\ref{vectorpsi0lep}) reduces to the following
$N$-scalar NLEPs, where $\alpha$ is any
eigenvalue of ${\mathcal B}$:
\begin{align}\label{NLEP}
 \Psi^{\prime\prime}-\Psi + \bar\chi U_0 \Psi -\alpha \bar{\chi}U_0
  \frac{\int_{-\infty}^\infty U_0^2 \Psi \, dy}{\int_{-\infty}^\infty U_0^2 \, dy}
  =\lambda_0 \Psi \,,  \qquad \Psi\to 0 \,\, \mbox{as} \,\, |y|\to \infty\,.
\end{align}

Since $U_0\gg {\mathcal O}(1)$ in the sub-inner region, we will
transform (\ref{NLEP}) to the $z$-variable. Recall that in the
$j^{\mbox{th}}$ sub-inner region, we have from Proposition \ref{prop1}
that
\begin{align}\label{gm_anal}
  z=v_{\max 0} y \,, \qquad y=\epsilon^{-1}(x-x_j) \,,
  \qquad U_0\sim \frac{1}{2}\bar\chi v_{\max 0}^2\sech^{2}\left(
  \frac{\bar{\chi}z}{2}\right) \,.
\end{align}
By introducing the re-scaled coordinate $\bar{z}:={\bar{\chi} z/2}$,
and defining
\begin{equation}\label{sech2}
  \bar{U_0} := 2\sech^{2}(\bar{z}) \,,
\end{equation}
we readily derive from (\ref{NLEP}) that we must analyze, on
$-\infty<\bar{z}<+\infty$, the approximating NLEP given by
\begin{equation}\label{NLEPbarzvar}
   \Psi_{\bar z\bar z}+ \bar{U_0}\Psi -\alpha \bar{U_0}\frac{\int_{-\infty}^\infty
     {\bar U_0}^2\Psi \, d\bar z }{\int_{-\infty}^\infty {\bar U_0}^2 \,d\bar z}
   =\frac{4}{\bar\chi^2 v_{\max 0}^2}(\lambda_0+1) \Psi \,, \qquad
   \Psi \,\,\, \mbox{bounded as} \,\,\,  |\bar{z}|\rightarrow \infty\,.
\end{equation}

\subsection{Competition Instabilities: $\tau=0$}\label{nlep:compet}

From the NLEP (\ref{NLEPbarzvar}), we now determine the conditions on
$d_1\in{\mathcal T}_e$, $\mu$, $\bar u$ and $N$ such that the
$N$-spike equilibrium is linearly stable with respect to the large
eigenvalues when $\tau=0$.  To do so, we must first determine explicit
formulae for the eigenvalues of the matrix $\mathcal B$ in
(\ref{mathcalB}). Then, by analyzing the NLEP, we must calculate the
critical threshold $\alpha_c>0$ such that in the restricted subset for
which $\lambda_0\not=0$, we can guarantee that when $\alpha<\alpha_c$
the principal eigenvalue of (\ref{NLEPbarzvar}) has a positive real
part, and that when $\alpha>\alpha_c$ it has a negative real part.

One can immediately conclude that when the minimum eigenvalue of
matrix $\mathcal B$, labeled by $\alpha_{min}$, satisfies
$\alpha_{min}>\alpha_c$, the NLEP (\ref{NLEPbarzvar}) with $\tau=0$
has no eigenvalue with a positive real part in the subset for which
$\lambda_0\not=0.$ We will calculate the explicit range of parameter
values $d_1\in {\mathcal T}_e$, $\mu$, $\bar u$ and $N$ to ensure that
the condition $\alpha_{min}>\alpha_c$ holds, which guarantees that the
$N$-spike equilibrium is linearly stable with respect to the large
eigenvalues when $\tau=0$. Our results will be expressed in terms of a
threshold value in the diffusivity $d_1$.

In Appendix \ref{appendixA} we show that when $d_1\in {\mathcal T}_e$,
the eigenvalues $\alpha_j$ of ${\mathcal B}$ when $\tau=0$ are related
to the eigenvalues $\sigma_j$ of the Green's matrix ${\mathcal G}$ by
\begin{equation}\label{eig:alpha_all}
  \alpha_j = \frac{2\sigma_j}{2\sigma_j - \sigma_1}\,, \qquad \mbox{for}
  \quad j=1,\ldots,N \,,
\end{equation}
where $\sigma_j$ is defined in (\ref{app:sigma_j}) upon setting
$\tau=0$. The minimum such eigenvalue is $\alpha_{min}=\alpha_N$.

Next, we focus on the computation of the critical threshold
$\alpha_c.$ In fact, if we entirely follow the method in
\cite{Wei1999single} to study (\ref{NLEPbarzvar}), we readily obtain
that $\alpha_c\sim 1$.  However, the next order term in
$\alpha_c$ is ${\mathcal O}(|\log\epsilon|^{-1})$ since it involves
$v_{\max 0}$. This term is key for obtaining accurate predictions of the
stability threshold when $\tau=0$. To obtain this refined asymptotic
formula of $\alpha_c$, in Appendix \ref{appensec3} we transform the
NLEP into an ODE that can be solved with the use of hypergeometric
functions.  We summarize our rigorous results in  the following
theorem:

\begin{theorem}\label{theorem1}
  Consider the following nonlocal eigenvalue problem (NLEP):
\begin{align}\label{NLEP1}
  \left\{\begin{array}{ll}
   \Psi_{\bar z\bar z}+ \bar{U_0}\Psi -\gamma_0 \bar{U_0}\frac{\int_{-\infty}^\infty
     {\bar U_0}^2\Psi \, d\bar z }{\int_{-\infty}^\infty {\bar U_0}^2 \,d\bar z}
           =\frac{4}{\bar\chi^2 v_{\max 0}^2}(\lambda_0+1) \Psi \,,
            &-\infty<z<+\infty\,,\\
 \Psi  \quad \mbox{bounded as } \quad \vert z\vert\rightarrow \infty\,.
\end{array}
\right.
\end{align}
Here $\gamma_0\geq 0$ and $\bar{U}_0$ is given in (\ref{sech2}). Let
$\lambda_0\not=0$ be the eigenvalue of (\ref{NLEP1}) with the largest
real part.  Then for $v_{\max 0}\gg 1$, we have $\mbox{Re}(\lambda_0)>0$ when
$\gamma_0<\gamma_c:=1-\frac{3}{2\bar\chi v_{\max 0}}$. Alternatively,
we have $\mbox{Re}(\lambda_0)<0$ when $\gamma_0>\gamma_c$. 
\end{theorem}
\begin{proof}
  The proof of Theorem \ref{theorem1} is given in Appendix
  \ref{appensec3}.
\end{proof}

We observe from Theorem \ref{theorem1} that when $v_{\max 0}$ is
sufficiently large, we have $\gamma_c\sim 1$, However, the correction
term is needed to obtain an improved result. Since the minimum
eigenvalue of ${\mathcal B}$ in (\ref{eig:alpha_all}) occurs when $j=N$,
we use Theorem \ref{theorem1} to conclude for $\tau=0$ and
$d_1\in {\mathcal T}_e$, that $\mbox{Re}(\lambda_0)=0$ when
\begin{align}\label{lambdagndividelambda1}
\frac{2\sigma_N}{2\sigma_{N}-\sigma_{1}}=1-\frac{3}{2\bar\chi v_{\max 0}}\,,
\end{align}
where $\sigma_{1}$ and $\sigma_{N}$ are given in (\ref{app:sigma_j}) when
$\tau=0$. This yields that
\begin{equation}\label{jac:thresh}
  \frac{\sigma_N}{\sigma_1} =
  \frac{ {e/(2f)} + 1 }{{e/(2f)} -\cos\left({\pi/N}\right)} =
  \frac{1}{2} - \frac{\bar{\chi}v_{\max 0}}{3}\,,
\end{equation}
where ${e/(2f)}=-\cos\left({2\theta/N}\right)$ with
$\theta=\sqrt{\frac{\mu\bar u}{d_1}}$ can be calculated from
(\ref{def}). By isolating $\cos\left({2\theta/N}\right)$, we get
\begin{equation*}
  \cos\left(\frac{2\theta}{N} \right) = \frac{ 1-a\cos\left({\pi/N}\right)}
  { a+1} \,, \quad \mbox{where}
  \quad a := \frac{\bar\chi v_{\max 0} }{3}-\frac{1}{2}\,.
\end{equation*}
Upon solving this expression for $d_1$, we can obtain a critical
threshold in terms of $\mu$, $\bar u$, $\bar\chi$ and $N$. In this
way, owing to Theorem \ref{theorem1}, we summarize our results for the
case $\tau=0$ as follows:

\begin{proposition}\label{prop33} Assume that $d_1\in {\mathcal T}_e$ and
  $\tau=0$. Let $\lambda_0\not=0$ be the eigenvalue of
  (\ref{NLEPbarzvar}) with the largest real part when $\tau=0$.  Then,
  for $N=1,2,\ldots$, $\mbox{Re}(\lambda_0)<0$ when
\begin{align}\label{d1d1c}
  d_1< d_{1cN} := \frac{4\mu \bar u}{N^2
  \left(\arccos(\eta_N)\right)^2} \,, \quad \mbox{where} \quad
  \eta_N:=\frac{1-a\cos(\pi/N)}{a+1} \,, \quad
   a:=\frac{\bar\chi v_{\max 0} }{3}-\frac{1}{2}\,.
\end{align}
Here $v_{\max 0}$ is determined by (\ref{singlevmax}).  Alternatively,
when $d_1>d_{1cN}$, we have $\mbox{Re}(\lambda_0)>0$. Since
$d_{1c1}=\infty$ when $N=1$, we conclude that a single interior spike
is always linearly stable with respect to the large eigenvalues for
any $d_1={\mathcal O}(1)$ when $\tau=0$.
\end{proposition}

Proposition \ref{prop33} provides the stability criterion for an
$N$-spike equilibrium with respect to the large eigenvalues when
$\tau=0$. To relate $d_{1cN}$ to the thresholds $d_{1pN}$ and
$d_{1Tm}$ of the admissible set ${\mathcal T}_e$, as defined in
(\ref{d1:admiss}), we observe from (\ref{d1d1c}) that since
$v_{\max 0}\gg 1$, we have $\eta_N>0$ for $N=2$, and $\eta_N<0$ for
$N\geq 3$. Therefore, $0<\arccos(\eta_2)<{\pi/2}$, while
${\pi/2}<\arccos(\eta_N)<\pi$ for any $N\geq 3$. As a result, for
$\epsilon\to 0$, we conclude that
\begin{align}\label{d1c:ineq}
  d_{1p2}<d_{1T1} <d_{1c2} \,, \quad
  \mbox{for} \,\, N=2 \,; \qquad
  d_{1pN}<d_{1cN}<\ d_{1Tm} \,, \quad 
 \mbox{for} \,\, N\geq 3  \,\, \mbox{and} \,\,  m\leq {N/2} \,.
\end{align}

However, since $v_{\max 0}$ depends weakly on $d_1$, the
threshold $d_{1cN}$ in (\ref{d1d1c}) is a weakly nonlinear implicit
expression that must be solved numerically.  To illustrate our
results, we chose $d_2=0.0004=\epsilon^2$, $\bar u=2$, $\mu=1$ and
$\bar\chi=1,$ and we calculate the thresholds $d_{1cN}$ for $N=2,3,4$
as
\begin{align}\label{d1234num_a}
   d_{1c2}\approx  2.36\,\,\, (N=2);
  \qquad d_{1c3}\approx 0.74\,\,\, (N=3); \qquad d_{1c4}\approx 0.39\,\,\,
  (N=4)\,.
\end{align}
When $N\gg 1,$ $d_{1cN}$ has the limiting behavior
$d_{1cN}\sim {4\mu \bar u
  N^{-2}/\left[\arccos(\eta_{\infty})\right]^2}$ where
$\eta_{\infty}:={(1-a)/(1+a)}$.  This limiting result is valid only
for $N\ll {1/\epsilon}$, owing to the fact that steady-state analysis
in \S \ref{sec2} requires that ${d_1/\epsilon^2}\gg 1.$

In summary, our analysis has shown that a sufficiently large cellular
diffusion rate $d_1$ will trigger a competition instability for an
$N$-spike steady-state solution when $\tau=0$, To partially confirm
our theory, in Figure \ref{figurecompetition} we show full numerical
results computed from (\ref{timedependent}) showing a competition
instability for a two-spike quasi steady-state solution as $d_1$
slowly increases in time. This initial instability is found to lead to
a nonlinear process that annihilates one of the two spikes. This
observation suggests that competition instabilities for the KS model
(\ref{timedependent}) are in fact subcritical, as is well-known for
the 1D Gierer-Meinhardt RD model \cite{paquin}.

\begin{figure}[h!]
\centering
\begin{subfigure}[t]{0.33\textwidth}
    \includegraphics[width=1.0\linewidth,height=7.0cm]{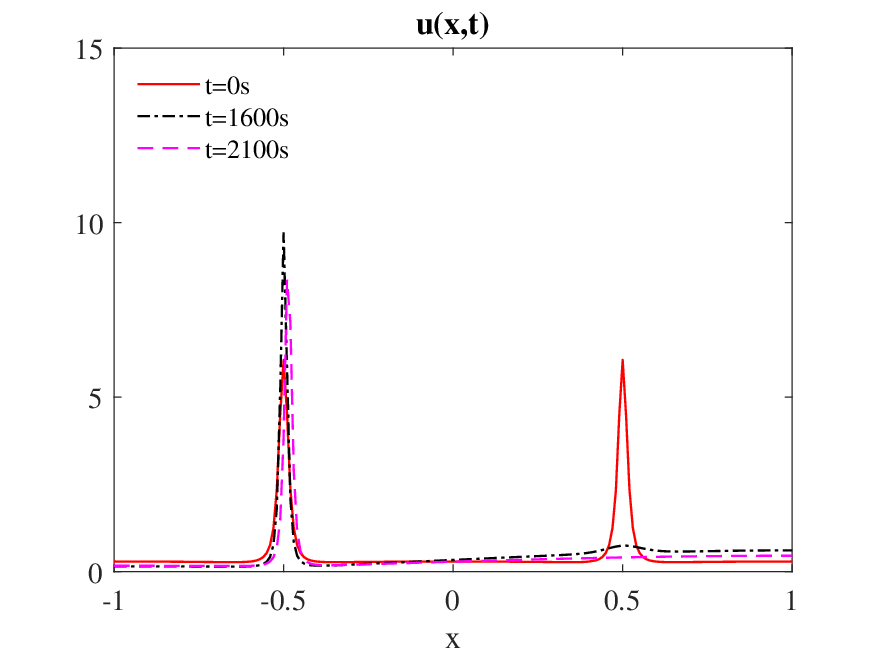}
     \caption{dynamics of $u$}
     \label{subfigcompetition1}
\end{subfigure}\hspace{-0.2in}
\begin{subfigure}[t]{0.33\textwidth}
  \includegraphics[width=1.0\linewidth,height=7.0cm]{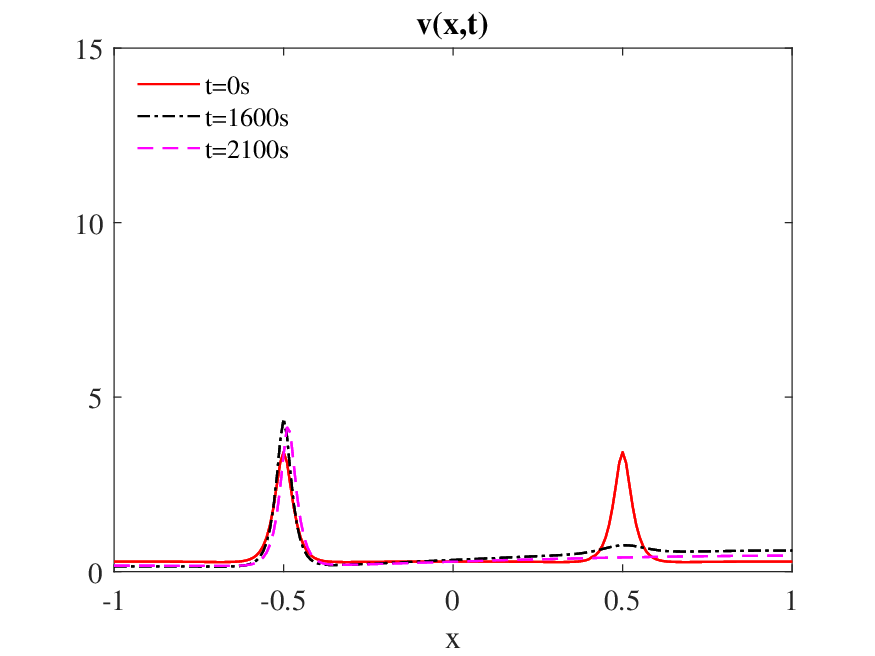}
    \caption{dynamics of $v$}
    \label{subfigcompetition2}
\end{subfigure}
\hspace{-0.2in}
\begin{subfigure}[t]{0.33\textwidth}
  \includegraphics[width=1.0\linewidth,height=7.0cm]{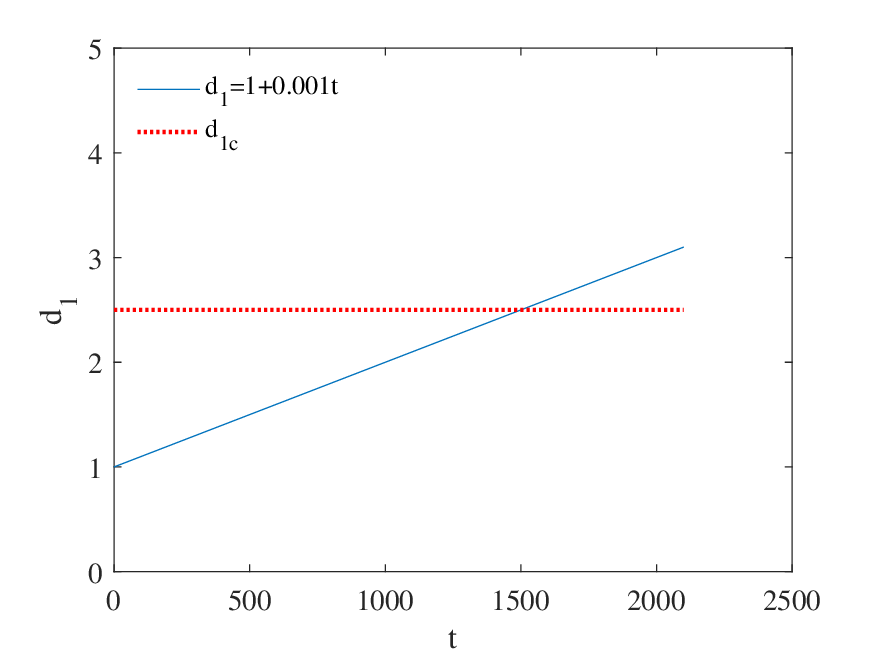}
    \caption{$d_1$ versus $t$}
    \label{subfigcompetition3}
\end{subfigure}
\caption{\fontsize{11pt}{9pt}\selectfont { {\em Full PDE simulations
      of (\ref{timedependent}) using FlexPDE7 \cite{flex2021}
      illustrating a competition instability of a two-spike
      steady-state when $d_1$ is increased slowly in time $t$.  Left
      and Middle: snapshots of $(u,v)$ at three times, showing the
      collapse of a spike, with $d_1=\chi=1$, $\bar u=2$, $d_2=0.0004$
      and $\mu=1;$ Right: the cellular diffusion $d_1$ versus time.}}
  In subfigure (c), the dotted line $d_{1c2}$ represents the stability
  threshold of large eigenvalues computed numerically and the solid
  line is the slow increasing ramp for $d_1$ versus $t$.  Observe that
  $d_{1c2}\approx 2.5$ agrees rather well with the analytical results
  in (\ref{d1234num_a}) and (\ref{d1234jnum}).}
\label{figurecompetition}
\vspace{-0.0in}
\end{figure}

\subsubsection{Invertibility of the Jacobian Matrix for $s_j$}\label{sec:jac_non}

We now provide an alternative approach to estimate the competition
instability threshold when $\tau=0$.  We will show that this threshold
closely approximates a bifurcation point associated with the
linearization of the coupled nonlinear algebraic system
(\ref{quasisj}) that was derived in our analysis of quasi steady-state
patterns.

We begin by writing (\ref{quasisj}) in the vector form
$\boldsymbol{F}(s_1,\ldots,s_N)=0$ with
$\boldsymbol{F}=(F_1,\ldots,F_N)^T$. By differentiating $F_i$ with
respect to $s_j$ we obtain, in terms of the Kronecker symbol
$\delta_{ij}$, that
\begin{align}\label{diffF}
  \frac{\partial F_i}{\partial s_j} (s_1,\ldots,s_N)=\delta_{ij}-
  2\epsilon\bar\chi v_{\max j}^2v_{\max j}^{\prime}G(x_i;x_j) \,,
\end{align}
where from (\ref{diffvmaxj_f}) of Appendix \ref{appendix:vderiv} we have
that
\begin{align}\label{diff_vmaxj}
  v_{\max j}^{\prime}:=\frac{dv_{\max j}}{d s_j} \sim -\frac{\zeta_{\max j}}
  {\bar{\chi} s_j}
  \,, \qquad \zeta_{\max j}:= \left(1 - \frac{2}{\bar{\chi} v_{\max j}}\right)^{-1}
  \,.
\end{align}

We now evaluate the Jacobian matrix
$\mathcal J=\Big(\frac{\partial F_i}{\partial s_j}\Big)_{N\times N}$
at the equilibrium solution where $x_j=x_j^0$, $s_j=s_0$,
$v_{\max j}= v_{\max 0}$, and
$\zeta_{\max j}=\zeta_0= \left(1 - {2/(\bar{\chi} v_{\max
      0})}\right)^{-1}$ for $j=1,\ldots,N$ with $x_j^0$ and $s_0$
defined by \eqref{locations} and (\ref{ag}). We seek to determine the
largest value of $d_1$ in the admissible set ${\mathcal T}_{e}$ of
(\ref{d1:admiss}) where the Jacobian matrix
is not invertible. Upon substituting (\ref{diff_vmaxj}) into
(\ref{diffF}), and evaluating the resulting expression at the
equilibrium solution, where we use
$s_0={2\epsilon \bar{\chi} a_g v_{\max 0}^3/3}$ from (\ref{ag}), we
obtain that
\begin{align}\label{jac:balance}
  \frac{\partial F_i}{\partial s_j}(s_1,\ldots,s_N)
  \Big|_{s_1=\cdots=s_N=s_0}\sim \delta_{ij}-\left(
  \frac{3}{ 2-\bar{\chi} v_{\max 0}}
  \right) \frac{G\big(x_i^0;x_j^0\big)}{a_g} \,.
\end{align}
In this way, the Jacobian matrix $\mathcal J$ at the equilibrium
solution is given for $\epsilon\to 0$ by
\begin{equation}\label{jacobian3}
   {\mathcal J} \sim I - \left(
  \frac{3}{ 2-\bar{\chi} v_{\max 0}}
\right) \frac{{\mathcal G}}{a_g}\,.
\end{equation}
Here ${\mathcal G}$ is the Green's matrix
$\Big(G\big(x_i^0;x_j^0\big)\Big)_{N\times N}$ for $\tau=0$, which is
evaluated at the equilibrium spike locations.

When $d_1\in {\mathcal T}_e$, it follows from (\ref{jacobian3}) that
the eigenvalues $\lambda_{\mathcal J,j}$ of the Jacobian
${\mathcal J}$ are related to the eigenvalues $\sigma_j$ of the
Green's matrix ${\mathcal G}$, obtained by setting $\tau=0$ in
(\ref{app:sigma_j}) of Appendix \ref{appendixA}, by
\begin{align}\label{lambdamathcalj}
  \lambda_{\mathcal{J},j}=1-\frac{3}{(2-\bar\chi v_{\max 0})}
  \frac{\sigma_{j}}{\sigma_{1}}\,,
\end{align}
where we used $\sigma_1=a_g$. The Jacobian matrix is singular when
$\lambda_{\mathcal J, j}=0$ in (\ref{lambdamathcalj}), which yields the
condition
\begin{equation}\label{njac:jthres}
\frac{\sigma_j}{\sigma_1} \sim   \frac{2}{3} - \frac{\bar\chi v_{\max 0}}{3} \,.
\end{equation}
The largest value of $d_1$ where the Jacobian is singular is obtained by
setting $j=N$. By using (\ref{jac:thresh}) this yields
\begin{equation}\label{njac:thresh}
  \frac{\sigma_N}{\sigma_1} =
  \frac{ \cos\left({2\theta/N}\right)-1 }
  {\cos\left({2\theta/N}\right)+\cos\left({\pi/N}\right)} \sim
  \frac{2}{3} - \frac{\bar\chi v_{\max 0}}{3} \,,
\end{equation}
where $\theta=\sqrt{\mu\bar{u}/d_1}$. Upon solving this expression for
$d_1$, we obtain the following critical threshold for $d_1$:
\begin{align}\label{d1d1cj}
  d_{1cN}^{\star}:=\frac{4\mu \bar u}{N^2
  \Big[\arccos\Big(\frac{1-a_1\cos\left({\pi/N}\right)}
  {a_1+1}\Big)\Big]^2}\,, \quad N=1,2,\ldots, \quad \mbox{where}
  \quad a_1:= \frac{\bar\chi v_{\max 0}}{3} -\frac{2}{3} \,.
\end{align}

We remark that the leading order term for $a_1$ given in
(\ref{d1d1cj}) is ${\bar\chi v_{\max 0}/3}$, which agrees precisely 
with the leading term of $a$ defined in (\ref{d1d1c}), as derived by
analyzing the zero-eigenvalue crossing condition of the NLEP. This
observation partially confirms our asymptotic results given in
Proposition \ref{prop33}. For the parameter values $d_2=0.0004,$
$\bar u=2,$ $\mu=1$ and $\bar\chi=1$, we use (\ref{d1d1cj}) to
calculate $d_{1c1}^{\star}=\infty$ and
\begin{align}\label{d1234jnum}
  d_{1c2}^{\star} \approx  2.91\,\,\, (N=2);
  \qquad d_{1c3}^{\star} \approx 0.97\,\,\, (N=3); \qquad
  d_{1c4}^{\star} \approx 0.54\,\,\,  (N=4)\,.
\end{align}

\subsection{Hopf Bifurcations: $\tau\not=0$}  \label{nlep:hopf_new}

In this subsection we focus on the possibility of an oscillatory
instability in the amplitude of a single steady-state spike for
(\ref{timedependent}) on the range $d_1\in {\mathcal T}_e$ where
$\tau\not=0$. In particular, for the linearization of a one-spike
steady-state solution we will show that there can be a Hopf
bifurcation leading to an oscillatory instability in the spike
amplitude. More specifically, by analyzing (\ref{NLEPbarzvar}) we will
compute the threshold $\tau=\tau_c>0$ such that the principal
eigenvalue of (\ref{NLEPbarzvar}) has the form $\lambda_0=i\lambda_H$
where $i:=\sqrt{-1}$ and $\lambda_H>0$ is real.
 
 As shown in (\ref{nlepequiv}) of Appendix \ref{appensec3}, if we
 define $w=\sqrt{\bar U_0}$ we can transform (\ref{NLEPbarzvar}) to
\begin{align}\label{secthopf-nlepequiv}
  \Psi_{0\bar z\bar z}+w^2\Psi_0-\kappa \frac{\int_{-\infty}^\infty
  w^2\Psi_0\, d\bar z}{\int_{-\infty}^\infty w^4\, d\bar z}w^2=
  \Lambda \Psi_0\,; \qquad \kappa:=\frac{\alpha(4-\Lambda)}{2+\alpha}\,,
  \qquad \Lambda:=4\frac{(\lambda_0+1)}
  {\bar\chi^2v_{\max 0}^2}\,.
\end{align}
By using the results in Appendix \ref{appendixA} for $\alpha$, we
obtain that the NLEP multiplier $\kappa$ is
\begin{align*}
  \kappa=4\left(1-\frac{\Lambda}{4}\right)\left(3-\sqrt{1-\frac{\tau\lambda_0}
  {\mu\bar u}}\frac{\tan\Big(\theta\sqrt{1-\frac{\tau\lambda_0}{\mu \bar u}}
  \Big)}{\tan\theta}\right)^{-1}\,, \qquad \mbox{where} \qquad
  \theta = \sqrt{\frac{\mu\bar u}{d_1}} \,,
\end{align*}
and where we have taken the principal branch of
$\sqrt{1-\frac{\tau\lambda_0}{\mu\bar u}}$.  Next, we transform
\eqref{secthopf-nlepequiv} to an algebraic equation in terms of
hypergeometric functions.  By using (\ref{thresholdeqappen}) of
Appendix \ref{appensec3}, we choose $\delta_1={\sqrt{\Lambda}/2}$ to
get
\begin{align}\label{secthopf-algebraic1}
  \frac{4}{\kappa} =
  &(1-\delta_1^2)^{-1} {}_{4}F_3(1,\frac{1}{2},2,2;2-\delta_1,2+\delta_1,
    \frac{5}{2};1)\nonumber\\
  &\quad + \frac{A}{3}\left(\frac{3}{2}\right)^{1+\delta_1}
    \frac{\Gamma(1+\delta_1)\Gamma
    \big(\frac{1}{2}\big)}{\Gamma(\frac{3}{2}+\delta_1)}{}_{3}
    F_2(1+\delta_1,\delta_1-\frac{1}{2},1+\delta_1;2\delta_1+1,\frac{3}{2}
    +\delta_1;1)\,,
\end{align}
where $\sqrt{\Lambda}$ is taken as the principal branch.  In terms of
$\tau=\tau_{c}$ and $\lambda_0=i\lambda_H$,
(\ref{secthopf-algebraic1}) is a single complex algebraic equation
that can be separated into real and imaginary parts to obtain a
coupled algebraic system for $\tau_c$ and $\lambda_H$.

The results obtained by solving this system numerically are shown in
Figure \ref{hopfthresholdfig}, where we set $\mu=1$, $\bar u=2$,
$\bar\chi=1$, and $d_2=0.0004$. Figure \ref{hopfthresholdfig_a} shows
that the spike will develop amplitude oscillations when $\tau$
increases passes through $\tau_c$.  The threshold $\tau_c$ is seen to
be a decreasing function of the cellular diffusivity $d_1$.  Figure
\ref{hopfthresholdfig_b} shows numerically that the transversality
condition of the Hopf bifurcation is satisfied, as unstable
eigenvalues enter $\mbox{Re}(\lambda_0)>0$ when $\tau$ increases above
$\tau_c$.

\begin{figure}[h!]  \centering
\begin{subfigure}[t]{0.5\textwidth}
    \includegraphics[width=1.0\linewidth,height=7.0cm]{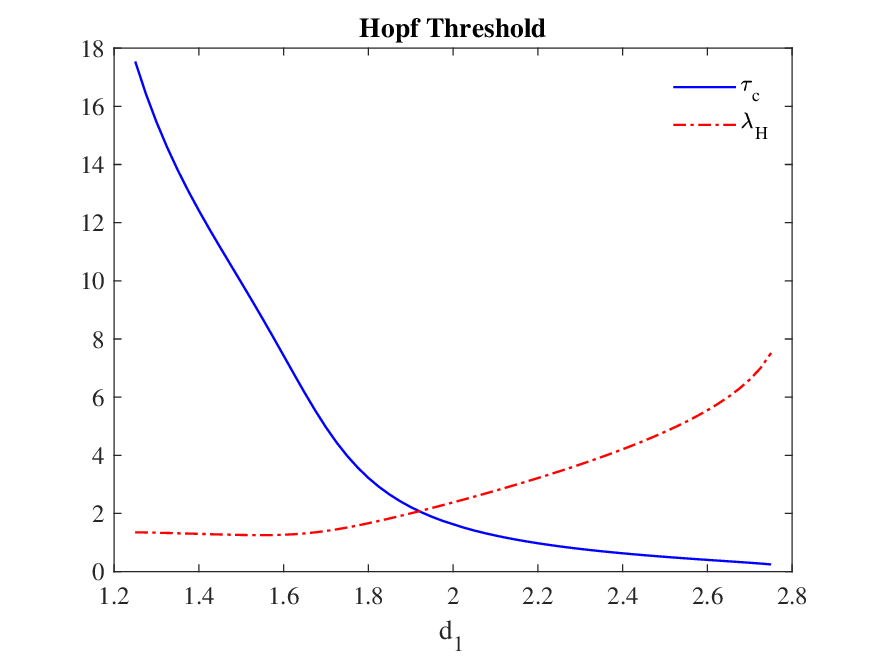}
    \caption{Hopf threshold $(\tau_c,\lambda_H)$}
      \label{hopfthresholdfig_a}
\end{subfigure}\hspace{-0.25in}
\begin{subfigure}[t]{0.5\textwidth}
  \includegraphics[width=1.0\linewidth,height=7.0cm]{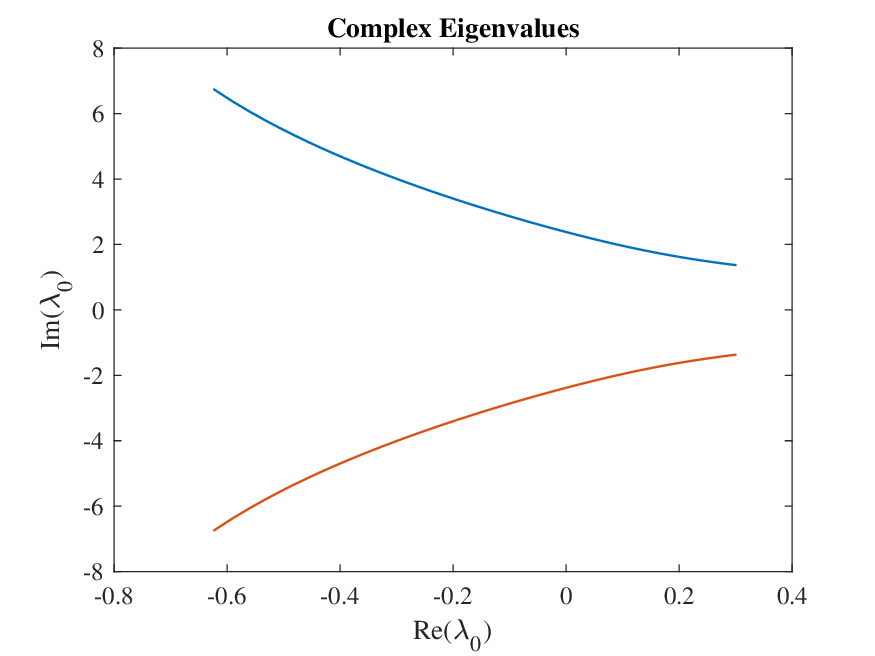}
  \caption{Complex Eigenvalues for \eqref{secthopf-nlepequiv}}
  \label{hopfthresholdfig_b}
\end{subfigure}
\\ \vspace{-0.15in}
\caption{\fontsize{11pt}{9pt}\selectfont { {\em The Hopf bifurcation
      threshold $(\tau_c,\lambda_H)$ (left panel) and the path of the
      complex spectra of (\ref{secthopf-nlepequiv}) as $\tau$ is
      increased above $\tau_c$ for the linearization of a single
      steady-state spike, as obtained by solving
      (\ref{secthopf-algebraic1}) numerically.}  In (a), the solid
    blue line represents the Hopf threshold $\tau_c(d_1)$ and the
    dotted red line denotes the critical eigenvalues $i\lambda_H$.
    The left panel (a) shows that the Hopf threshold $\tau_c$
    decreases as the cellular diffusivity $d_1$ increases.  The right
    panel (b) shows the path of the complex spectra for $d_1=2$ as
    $\tau$ increases. We observe that for $\tau>\tau_c$ unstable
    eigenvalues enter $\mbox{Re}(\lambda_0)>0$.}}
\label{hopfthresholdfig}
\vspace{-0.0in}
\end{figure}
 
\section{Analysis of the Small Eigenvalues}\label{sec4}

In \S \ref{sec3} we analyzed the linear stability of an $N$-spike
steady-state solution with respect to the large eigenvalues of the
linearization. In this section, for $d_1\in {\mathcal T}_e$, we will
formulate a matrix problem for the small eigenvalues of order
${\mathcal O}(\epsilon^3v_{\max 0}^3)$ in the linearization, and we
will calculate an explicit asymptotic formula for them.

To begin the analysis, we differentiate (\ref{ss1}) for $v$ to obtain
\begin{align}\label{smallep42}
  L_{\epsilon} v_{ex}=-u_{ex}\,, \qquad \mbox{where} \quad
  L_{\epsilon} \psi:=\epsilon^2\psi_{xx}-\psi \,.
\end{align}

Our first goal is to obtain an approximate expression for
(\ref{smallep42}) in terms of the inner coordinate near a spike.
Focusing on the $j^{\mbox{th}}$ spike, we find from Proposition \ref{prop1}
that the composite expansion of the quasi-equilibrium solution $u_q$
can be written near the $j^{\mbox{th}}$ spike as
\begin{align*}
  u_q\sim s_j(x)e^{\bar\chi (V_{0}(y)-s_0)}\,,\quad y=\epsilon^{-1}(x-x_j)\,,\quad
  j=1,\ldots,N\,.
\end{align*}
Here $V_0(y)$ is the inner solution near the $j^{\mbox{th}}$ spike and 
\begin{align}\label{smallepsj-of-x}
  s_j(x):=\frac{2\bar\chi}{3 }\epsilon \sum_{k=1}^N v_{\max k}^3(x)\, G(x;x_k)\,,
  \quad j=1,\ldots,N\,.
\end{align}
Since $s_0=o(1)$, we find $e^{-\bar\chi s_0}\sim 1$, so that
\begin{align*}
  u_q\sim s_j(x)e^{\bar\chi V_{0}(y)}\,, \quad y=\epsilon^{-1}(x-x_j)\,,
  \quad j=1,\ldots,N\,.
\end{align*}
We differentiate $u_q$ with respect to $x$ to get for the
$j^{\mbox{th}}$ spike that
\begin{align}\label{uq-of-x}
  u_{qx}\sim s_{jx}(x)e^{\bar\chi V_0}+\epsilon^{-1}\bar\chi s_j(x) e^{\bar\chi V_0}
  V_{0}^{\prime}\,,
\end{align}
and by differentiating (\ref{smallepsj-of-x}) we obtain that
\begin{align}\label{sjx}
  s_{jx}(x)=2\bar\chi\epsilon\sum_{k=1}^Nv_{\max k}^2(x)
  \left[\partial_{x}v_{\max k }(x)\right]\,G(x;x_k)+
  \frac{2\bar\chi}{3}\epsilon\sum_{k=1}^N v_{\max k}^3(x)\,G_x(x;x_k)\,.
\end{align}
Noting from (\ref{diffvmaxi}) of Appendix \ref{appendix:vderiv} that
we can approximate
\begin{align*}
  \partial_x v_{\max k}(x)\sim -\frac{\zeta_{\max k}}{\bar{\chi} s_k}
  \partial_x s_k(x) \,, \qquad \zeta_{\max k} = \left( 1 - \frac{2}{\bar\chi
  v_{\max k}} \right)^{-1} \,, \quad k=1,\ldots,N\,,
\end{align*}
we obtain that (\ref{sjx}) becomes
\begin{align*}
  s_{jx}(x)\sim -2\epsilon\sum_{k=1}^Nv_{\max k}^2(x)\zeta_{\max k}
  \frac{s_{kx}}{s_k}G(x;x_k)+
  \frac{2\bar\chi}{3}\epsilon\sum_{k=1}^N v_{\max k}^3(x)G_x(x;x_k)\,.
\end{align*}
At the steady-state, for which $x_j=x_j^0$, for $j=1,\ldots,N$, we
have $v_{\max k}(x)=v_{\max}^0(x)$, $s_k(x)=s^0(x)$, and
$\zeta_{0}=\zeta_{\max k}$.  Therefore, for the $j^{\mbox{th}}$ spike evaluated
at the steady-state we have
\begin{align}\label{small:sxx}
  s_{x}^0(x)\vert_{x_i=x_j^0}\sim
  \overbrace{-\frac{2s_x^0(x)}{s^0(x)}\epsilon[\zeta_0 (v^0_{\max})^2(x)]
  \sum_{k=1}^N G\big(x_j^{0};x_k^0\big)}^{I_1}+
  \overbrace{\frac{2\bar\chi}{3}\epsilon [(v^0_{\max })^3(x)]\sum_{k=1}^N
  G_x\big(x_{j}^{0};x_k^0\big)}^{I_2}\,.
\end{align}
Next, since $y=\epsilon^{-1}(x-x_j^0)$, we find from
$S^0(y):=s^0(x_j^{0}+\epsilon y)$ and
$V_{\max }^0(y):=v_{\max}^0(x_{j}^{0}+\epsilon y)$, where $S^0$
(resp. $V_{\max }^0$) is $s^0$ (resp. $v_{\max }^0$), that in terms of
the $y$-variable
\begin{align*}
  \partial_y S^0(y)\vert_{x_j=x_j^0}\sim \overbrace{-\frac{2\partial_y S^0(y)}
  {S^0(y)}\epsilon[\zeta_0 (V^0_{\max})^2(y)]
  \sum_{k=1}^N G\big(\epsilon y+x_j^0;x_k^0\big)}^{I_1}+
  \overbrace{\frac{2\bar\chi}{3}\epsilon^2 [(V^0_{\max })^3(y)]\sum_{k=1}^N
  G_x\big(\epsilon y+x_j^0;x_k^0\big)}^{I_2} \,,
\end{align*}
where for $G$ we have $G_x(x;x_k^0)=G_{x}(\epsilon y+x_j^0;x_k^0)\,.$
Then, at $x=x_j^0$, for $j=1,\ldots,N$, we have
$v^0_{\max}(x_j^0)=v_{\max 0}$ and $s^0(x_j^0)=s_0$. In this way,
(\ref{small:sxx}) becomes
\begin{align*}
    s_x^0(x_j^0)=\overbrace{-\frac{2s_x^0(x_j^0)}{s_0}\epsilon \zeta_0 v_{\max 0}^2
  \sum_{k=1}^N G\big(x_j^0;x_k^0\big)}^{II_1}+
  \overbrace{\frac{2\bar\chi}{3}\epsilon v_{\max 0}^3\sum_{k=1}^N
  G_x\big(x_j^0;x_k^0\big)}^{II_2}\,,
\end{align*}
where we identify that $II_2=u_{ox}(x_j^0)$ with $u_o(x)$ being the
outer solution constructed previously. In the $y$-variable, we find as
$|y|\rightarrow \infty$, that
\begin{align}\label{lep:sy_der}
  \partial_y S^0\rightarrow \overbrace{-\frac{2\partial_y S^0}{s_0}\epsilon
 \zeta_0 v_{\max 0}^2 \sum_{k=1}^N G\big(x_j^0;x_k^0\big)}^{II_1}+
  \overbrace{\frac{2\bar\chi}{3}\epsilon^2 v_{\max 0}^3\sum_{k=1}^N
  G_x\big(x_j^0;x_k^0\big)}^{II_2}\,.
\end{align}

According to (\ref{uq-of-x}), we set $x_j=x_j^0$ to conclude for the
$j^{\mbox{th}}$ spike region $x\in(x_j^0-\epsilon,x_j^0+\epsilon)$ that $u_e$
satisfies
\begin{align*}
  u_{ex}\sim s_{x}^0(x)e^{\bar\chi V_0}+\epsilon^{-1}\bar\chi s^0(x)
  e^{\bar\chi V_0}V_{0}^{\prime}\,,
\end{align*}
where $V_{0}^{\prime}=\partial_y V_0$. Finally, we use $y=\epsilon^{-1}(x-x_j^0)$
and transform $u_{ex}$ to the $y$-variable to get
\begin{align*}
  \partial_y U\sim \partial_y S^0(y) e^{\bar\chi V_0}+\bar\chi S^0(y)
  e^{\bar\chi V_0}\partial_y V_0\,,
\end{align*}
where $U_j(y)=u_e(x_j^{0}+\epsilon y)$.  It follows from (\ref{smallep42}) that
for $x$ near $x_j$
\begin{align}\label{leps:V_j}
  L_{\epsilon} V_{j}^{\prime}\sim -\partial_y S^0(y) e^{\bar\chi V_0}-\bar\chi S^0(y)
  e^{\bar\chi V_0}\partial_y V_0\,.
\end{align}

Next, we investigate the linearized eigenvalue problem (\ref{lep}).
To obtain the $j^{\mbox{th}}$ inner solution, we expand
\begin{align}\label{phij}
  \Phi_j(y)= c_j\Phi_{0j}+\epsilon^2 c_j\Phi_{1j}+\ldots\,, \qquad
  \Psi_j(y)=c_jV^{\prime}_j +\epsilon^2 c_j\Psi_{1j}+\ldots\,,
\end{align}
where $y=\epsilon^{-1}(x-x_j)$. Similarly as in \S \ref{sec3}, we
substitute (\ref{phij}) into (\ref{lep}) to get
\begin{align}\label{phi0j}
\Phi_{0j}=\bar\chi U_jV_j^{\prime}\,.
\end{align}
Moreover, by using the fact that $\lambda=o(1),$ we conclude from
(\ref{leps:V_j}) and (\ref{phi0j}) that the $\bar\chi U_jV^{\prime}_j$ term
in the $\psi$-equation (\ref{lep_b}) is cancelled but the term
$-\partial_y S^0(y) e^{\bar\chi V_0}$ remains. To eliminate this term, we need to
formulate the matching condition between the inner and outer
solutions.

Defining the outer solution by $\phi_o$, we now derive the appropriate
jump conditions across the $j^{\mbox{th}}$ spike for $\phi_0$. To begin with,
we observe that
$\partial_{y}S(y) e^{\bar\chi V_0}\sim \partial_y S(y)$ for $|y|$
large. Moreover, since $II_1$ defined in (\ref{lep:sy_der}) is
expressed in terms of the Green's function $G$, we have that $\phi_o$
satisfies the following jump condition across $x_j$:
\begin{align}\label{jumpconditionew}
  \Big[\frac{d_1}{\mu}\phi_{ox}\Big]_j=-\frac{2\epsilon \zeta_0 v_{\max 0}^2}{s_0}
  \langle \phi_o\rangle_j \,, \quad j=1,\ldots,N\,,
\end{align}
where $\big\langle f\big\rangle_j$ and $\big[f\big]_j$ are defined as
$\big\langle
f\big\rangle_j:=\big[f\big(x_j^+\big)+f\big(x_j^-\big)\big]/2$ and
$\big[f\big]_j:=f\big(x_j^+\big)-f\big(x_j^-\big)$, respectively. The
coefficient in (\ref{jumpconditionew}) can be simplified by
eliminating $s_0$ by using (\ref{algebraic2}). In addition, we find
as $\epsilon\rightarrow 0$ that
\begin{align}\label{U0CJphi0j}
  2U_0c_j\Phi_{0j}\sim \frac{2\bar\chi c_j}{3}\epsilon^2  v_{\max j}^3\,
  \delta^{\prime}(x-x_j)\,.
\end{align}
Upon defining $\phi_o:=\epsilon^2\bar \phi_o$, and dropping the
overbar notation, we combine (\ref{jumpconditionew}) and
(\ref{U0CJphi0j}) to obtain the following leading order outer problem
for $\phi_o$ with jump conditions across the $j^{\mbox{th}}$ spike:
\begin{align}\label{smallouterproblem413}
  \frac{d_1}{\mu} \phi_{oxx}+\bar u \phi_o\sim \frac{2\bar\chi }{3}v_{\max 0}^3
  \sum_{j=1}^N c_j \delta^{\prime}(x-x_j)- \frac{3\zeta_0 }{\bar\chi a_gv_{\max 0}}
  \sum_{j=1}^N\langle\phi_o\rangle_j  \delta(x-x_j)\,.
\end{align}

Our next aim is to establish the solvability condition that provides
the matrix eigenvalue problem for the small eigenvalues.  To do so, we
substitute (\ref{phij}) into (\ref{lep_b}) and
multiply it by $V_j^{\prime}$. Upon integrating the resulting
expression over $-1<x<1$, we drop some asymptotically negligible terms
to get
\begin{align}\label{dominantbefore}
  \sum_{i=1}^N \left(c_jL_{\epsilon}V_i^{\prime},V_j^{\prime}\right)+
  \epsilon^2 \sum_{i=1}^N\left(c_iL_{\epsilon}\Psi_{1i},V_j^{\prime}\right)
  +\sum_{i=1}^N\left( c_i\Phi_{0i},V_j^{\prime}\right)+\epsilon^2
  \left(\bar \phi_o, V_j^{\prime}\right)+\epsilon^2\sum_{i=1}^N
  \left(c_{i}\Phi_{1i},V_j^{\prime}\right)\sim \lambda \sum_{i=1}^N
  \left(c_iV_i^{\prime},V_j^{\prime}\right)\,,
\end{align}
for each $j=1,\ldots,N$.  Here the inner product $(f,g)$ is defined as
$\left(f,g\right):=\int_{-1}^1f g\, dx$.  Since $V_j$ decays
exponentially as $\vert y\vert\rightarrow \infty,$ we collect the
dominant terms to simplify (\ref{dominantbefore}) as
\begin{align}\label{dominantafter}
  c_j\left(V_j^{\prime},L_{\epsilon}V_j^{\prime}+\Phi_{0j}\right)+
  \epsilon^2 c_j \left(V_j^{\prime},L_{\epsilon}\Psi_{1j}+\Phi_{1j}\right)
  +\epsilon^2\left(\phi_o, V_j^{\prime}\right)\sim \lambda c_j
  \left(V_j^{\prime},V_j^{\prime}\right)\,, \quad j=1,\ldots,N\,.
\end{align}
Noting that $L_{\epsilon}$ is self-adjoint, we integrate
by parts on the second term of (\ref{dominantafter}). Expressing the
integrals in terms of $y=\epsilon^{-1}(x-x_j)$ we get in terms of
$u_{o}=u_o(x_j+\epsilon y)$ and $\phi_o=\phi_o(x_j+\epsilon y)$ that
\begin{align}\label{dominantintegral}
  -\epsilon^2 c_j\int_{-\infty}^\infty V_j^{\prime} u_{ox} \, dy
  +\epsilon^3\int_{-\infty}^\infty \phi_{o} V_j^{\prime} \, dy
  \sim \lambda c_j\epsilon \int_{-\infty}^\infty \left(V_j^{\prime}\right)^2\,dy\,,
  \quad j=1,\ldots,N \,.
\end{align}

Next, we analyze the left-hand side of (\ref{dominantintegral}) by
expanding  $u_o$ and $\phi_o$ in one-sided Taylor series. In this way,
the left-hand side of (\ref{dominantintegral}) becomes
\begin{align}\label{solvabilitybefore}
  -\epsilon^2 c_j\int_{-\infty}^\infty V_j^{\prime} u_{ox} \, dy
    +\epsilon^3\int_{-\infty}^\infty \phi_{o} V_j^{\prime} \, dy =
  \epsilon^4\langle \phi_{ox}\rangle_j \int_{-\infty}^\infty y\, V_j^{\prime}
     \, dy-\epsilon^3\langle u_{oxx}\rangle c_j\int_{-\infty}^\infty yV_j^{\prime}\,
     dy\,.
\end{align}
By using
$\langle u_{oxx}\rangle_j=-\frac{\bar u\mu}{d_1}\langle u_o\rangle_j$,
we further simplify (\ref{solvabilitybefore}) as
\begin{align}\label{lambdaformulabefore}
  -\epsilon^2 c_j\int_{-\infty}^\infty V_j^{\prime} u_{ox} \, dy
    +\epsilon^3\int_{-\infty}^\infty \phi_{o} V_j^{\prime}\, dy =
 \epsilon^4\langle \phi_{ox}\rangle_j \int_{-\infty}^\infty yV_j^{\prime}\,
     dy+\epsilon^3\frac{s_0c_j\bar u\mu}{d_1} \int_{-\infty}^\infty yV_j^{\prime}
     \, dy\,.
\end{align}
After rewriting the outer problem (\ref{smallouterproblem413}) in
terms of jump conditions, we combine (\ref{dominantintegral}) and
(\ref{lambdaformulabefore}) to obtain the following characterization for the
small eigenvalues:

\begin{proposition}\label{prop41lambda}
  For $d_1\in {\mathcal T}_e$, the eigenvalues $\lambda$ of
  (\ref{lep}) of order ${\mathcal O}(\epsilon^3v_{\max 0}^2)$ satisfy
\begin{align}\label{lambdaformula}
  \lambda c_j\int_{-\infty}^\infty \left(V_j^{\prime}\right)^2\,dy
  \sim \epsilon^3\left(\langle\phi_{ox}\rangle_j+
  \frac{s_0 c_j\bar u\mu}{\epsilon d_1}\right)
  \int_{-\infty}^\infty yV_j^{\prime}\, dy\,, \qquad j=1,\ldots,N\,,
\end{align}
where $s_0={\mathcal O}(\epsilon v_{\max 0}^3)$, and where
$\langle\phi_{ox}\rangle_j$ is determined by the 
solution to the BVP
\begin{align}\label{bvp1}
  \frac{d_1}{\mu} \phi_{oxx}+\bar u \phi_o=0\,, \quad -1<x<1\,, \quad
  x\neq x_j^{0}\,, \,\, j=1,\ldots,N\,; \qquad  \phi_{ox}(\pm 1)=0\,,
\end{align}
which satisfies the following jump conditions across each spike:
\begin{align}\label{bvp1:jump}
  \left[\frac{d_1}{\mu} \phi_o\right]_j=\frac{2\bar\chi c_j}{3}v_{\max 0}^3
\epsilon \,, \qquad \left[\frac{d_1}{\mu}\phi_{ox}\right]_j=-
  \frac{3\zeta_0}{\bar\chi a_gv_{\max 0}}\langle \phi_{o}\rangle_j\,,
  \qquad \zeta_0 := \left(1 - \frac{2}{\bar{\chi} v_{\max 0}}\right)^{-1}\,.
\end{align}
\end{proposition}

\subsection{Formulation of the Matrix Problem}

We will now solve (\ref{lambdaformula}) for $d_1\in {\mathcal T}_e$ so
as to derive a matrix eigenvalue problem for the small eigenvalues. To
do so, we let $m_k$, for $k=1,\ldots,N$, be constants to be found and
we write the solution to (\ref{bvp1}) in the form
\begin{align}\label{phiodecompose}
  \phi_o=\frac{2\bar\chi}{3}v_{\max 0}^3\sum_{k=1}^N c_kg(x;x_k)+
  \sum_{k=1}^N m_kG(x;x_k)\,.
\end{align}
Here, for $d_1\in {\mathcal T}_e$, the Green's function $G$ satisfies
(\ref{greenequation}), while the dipole Green's function $g$ satisfies
\begin{align}\label{green:dipole_small}
  \frac{d_1}{\mu}g_{xx}+\bar ug=\delta^{\prime}(x-x_j)\,, \quad -1<x<1\,;
  \quad g_x(\pm 1;x_j)=0\,;
  \qquad  \left[\frac{d_1}{\mu} g\right]_j=1 \,, \quad
      \left[g^{\prime}\right]_j=0\,.
\end{align}

Upon defining $\boldsymbol{m}:=(m_1,\ldots,m_N)^T$, we use the jump
condition in (\ref{bvp1:jump}) to obtain from (\ref{phiodecompose})
that $\boldsymbol{m}$ satisfies
\begin{align}\label{mequation}
  \boldsymbol{m}=-\frac{3\zeta_0}{\bar\chi a_gv_{\max0}}
  \Bigg(\mathcal G\boldsymbol{m} +\frac{2\bar\chi}{3}v_{\max 0}^3
  \mathcal P_g\boldsymbol{c}\Bigg)\,,
\end{align}
where ${\mathcal G}$ is the Green's matrix and where $\mathcal P_g$
and $\boldsymbol{c}$ are defined as
\begin{align}\label{matrixPg}
\mathcal P_g:=\left( \begin{array}{ccc}
  \langle g(x_1;x_1)\rangle_1 & \cdots &  g(x_1;x_N)\\
  \vdots & \ddots & \ddots\\
   g(x_N;x_1) &\cdots & \langle g(x_N;x_N)\rangle_N\\
    \end{array}
  \right)\,, \qquad \boldsymbol{c}:=\left( \begin{array}{c}
  c_1\\
  \vdots  \\
c_N\\
    \end{array}
  \right)\,.
  \end{align}
Upon solving (\ref{mequation}) for $\boldsymbol{m}$ we get
\begin{align}\label{mexpression4226}
  \boldsymbol{m}=-\frac{2v_{\max 0}^2\zeta_0}{a_g}\bigg(I+\frac{3\zeta_0}
  {\bar\chi a_gv_{\max 0}}\mathcal G\bigg)^{-1}\mathcal P_g\boldsymbol{c}\,.
\end{align}

Next, we use (\ref{phiodecompose}) to calculate $\langle\phi_{ox}\rangle_j$,
for $j=1,\ldots,N$, in the form
\begin{align}\label{smallphiouterxsec4}
  \langle \boldsymbol{\phi_{ox}}\rangle=\frac{2\bar\chi}{3}v_{\max 0}^3
  \mathcal G_g\boldsymbol{c}+\mathcal P\boldsymbol{m}\,,
\end{align}
where  $\langle\boldsymbol{\phi_{ox}}\rangle:=
\left(\langle\phi_{ox}\rangle_1,\ldots,\langle\phi_{ox}\rangle_N\right)^T$
and  $\boldsymbol{m}$ is given by (\ref{mexpression4226}). 
Here $\mathcal P$ and $\mathcal G_g$ are defined by
\begin{align}\label{matrixGg}
\mathcal P:=\left( \begin{array}{ccc}
  \langle G_x(x_1;x_1)\rangle_1 & \cdots &  G_x(x_1;x_N)\\
  \vdots & \ddots & \ddots\\
   G_x(x_N;x_1) &\cdots & \langle G_x(x_N;x_N)\rangle_N\\
    \end{array}
  \right)\,,
\qquad
\mathcal G_g:=\left( \begin{array}{ccc}
 g_x(x_1;x_1) & \cdots &  g_x(x_1;x_N)\\
  \vdots & \ddots & \ddots\\
   g_x(x_1;x_N) &\cdots &  g_x(x_N;x_N)\\
    \end{array}
    \right)\,.
\end{align}
By substituting  (\ref{mexpression4226}) and (\ref{smallphiouterxsec4}) into
(\ref{lambdaformula}) of Proposition \ref{prop41lambda}, we obtain that
\begin{align}\label{433insection4}
  \lambda \boldsymbol{c}\sim -\epsilon^3\beta_0\mathcal M\boldsymbol{c}\,,
  \qquad \mbox{where} \quad   \beta_0:=-
  \frac{\int_0^\infty yV_{0}^{\prime}\, dy}{\int_0^\infty
  \left(V^{\prime}_{0}\right)^2 \, dy}>0 \,.
\end{align}
Here $V_0$ is the common leading order core solution, and ${\mathcal M}$ is
defined for $d_1\in {\mathcal T}_e$ by
\begin{align}\label{jacobianbutinsmallep}
  \mathcal M:=\frac{2\bar\chi}{3}v_{\max 0}^3\mathcal G_g-
  \frac{2v_{\max 0}^2\zeta_0}{a_g}\mathcal P\left(I+\frac{3\zeta_0}
  {\bar\chi a_gv_{\max 0}}
  \mathcal G\right)^{-1}\mathcal P_g+\frac{s_0 \bar u\mu}{\epsilon d_1} I\,.
\end{align}

This result shows that $\lambda$ and $\boldsymbol{c}$ are related to
eigenpairs of the matrix $\mathcal M.$ As a result, the
analysis of the linear stability properties of the small eigenvalues
in (\ref{lep}) when $d_1\in {\mathcal T}_e$ is reduced to the problem
of analyzing the eigenvalues of the matrix $\mathcal M$ and
determining conditions on the parameters for which
${\mbox Re}(\lambda)<0$.

An important relationship between the existence of ${\mathcal M}$ and
the invertibility of the Jacobian associated with the nonlinear
algebraic system of quasi-equilibria, as studied in \S
\ref{sec:jac_non}, is summarized as follows:

\begin{remark}\label{remark:non_invert} Recalling that $a_g=\sigma_1$, the
  inverse
  $\left(I+\frac{3\zeta_0}{\bar\chi a_gv_{\max 0}} \mathcal
    G\right)^{-1}$ appearing in ${\mathcal M}$ of
  (\ref{jacobianbutinsmallep}) does not exist when
  \begin{equation}\label{small:non_invert}
    \frac{\sigma_j}{\sigma_1}= -\frac{\bar{\chi}v_{\max 0}}{3\zeta_0}=
    - \frac{\bar{\chi} v_{\max 0}}{3} \left(1 - \frac{2}{\bar{\chi} v_{\max 0}}
      \right) = \frac{2}{3} - \frac{\bar{\chi}v_{\max 0}}{3} \,,
 \end{equation}
 where $\sigma_j$ for $j=1,\ldots,N$ are the eigenvalues of
 ${\mathcal G}$ when $\tau=0$ and $d_1\in {\mathcal T}_e$. As a
 result, the non-existence of the small eigenvalues coincides, by
 using (\ref{njac:jthres}), with the non-invertibility of the Jacobian
  matrix of the linearization of the quasi-equilibrium solution around
  the steady-state. By setting $j=N$ in (\ref{small:non_invert}), we obtain
  $d_1=d_{1cN}^{\star}$, as given in (\ref{d1d1cj}), which approximates the
  competition instability threshold for an $N$-spike steady state solution
  when $\tau=0$ (see Proposition \ref{prop33}).
\end{remark}

To analyze ${\mathcal M}$, we must calculate the matrix spectrum of the
dipole Green's matrix $\mathcal G_g$ given in \eqref{matrixGg} when
$d_1\in {\mathcal T}_e$.  As shown in Appendix \ref{appensecB}, when
$d_1\in {\mathcal T}_e$ the inverse matrix of $\mathcal G_g$ is
readily identified as being proportional to the inverse of a
$N\times N$ symmetric tridiagonal matrix, labeled by ${\mathcal D}_g$,
and defined in (\ref{mathcalDg}) as
\begin{equation}\label{slep:dg}
  {\mathcal G}_g = \frac{\mu\theta}{d_1} {\mathcal D}_g^{-1} \,.
\end{equation}
The matrix spectrum of ${\mathcal D}_g$ for $d_1\in {\mathcal T}_e$,
is readily calculated as in \cite{iron2001stability}, and is
summarized as follows:

\begin{proposition}\label{prop41}
  The eigenvalues $\xi_j$ and the normalized eigenvectors
  $\boldsymbol{\nu_j}=(\nu_{1,j},\ldots,\nu_{N,j})^T$ of $\mathcal D_g$ are
\begin{subequations}\label{prop411}
\begin{align}
       \xi_1&=2\cot\left(\frac{2\theta}{N}\right)+2
       \csc\left(\frac{2\theta}{N}\right)=2\cot\left(\frac{\theta}{N}
             \right) \,,  \quad
{\boldsymbol{\nu}}_{1}=\frac{1}{\sqrt{N}}\left(1,-1,\ldots,1,\ldots,
             (-1)^{N+1}\right)^T\,,\label{prop411_a}\\
  \xi_j&=2\cot\left(\frac{2\theta}{N}\right) -
  2 \csc\left(\frac{2\theta}{N}\right)\cos\left(\frac{\pi (j-1)}{N}\right)
  \,, \quad  \nu_{l,j}=
   \sqrt{\frac{2}{N}}\sin\left(\frac{\pi (j-1)}{N}(l-\frac{1}{2})\right)\,,
  \,\,\, j=2,\ldots,N\,, \label{prop411_b}
\end{align}
\end{subequations}
for $l=1,\ldots,N$, where $\theta= \sqrt{\mu\bar{\mu}/d_1}$.  When
$d_1\in {\mathcal T}_e$, i.e.~$\theta <{\pi N/2}$, we have the
ordering $\xi_2<\ldots<\xi_N<\xi_1$.
\end{proposition}

By using the key Proposition \ref{prop41}, in Appendix \ref{appensecC}
we show how to diagonalize ${\mathcal M}$ and compute its
spectrum. This leads to the following explicit asymptotic
result for the small eigenvalues, valid as $\epsilon\to 0$:

\begin{proposition}\label{prop:small} For
  $d_1\in {\mathcal T}_e$ and $d_{1}<d_{1cN}^{\star}$, the small
  eigenvalues $\lambda_j$ satisfying (\ref{433insection4}) are given
  explicitly for $\epsilon\to 0$ by
\begin{align}\label{prop:small_eig}
  \lambda_j\sim -\frac{2\epsilon^3\beta_0}{3} \bar\chi v_{\max 0}^3
  \left( \frac{\mu\theta}{d_1\xi_j} -\frac{3\zeta_0}{\bar\chi a_gv_{\max 0}}
  \frac{\omega_j}{\xi_j}+\frac{\bar u\mu}{d_1} a_g\right)\,, \quad
  j=1,\ldots,N\,,
\end{align}
where $\xi_j$ are the matrix eigenvalues in (\ref{prop411}) and
$\zeta_0=\left(1-{2/(\bar{\chi}v_{\max 0})}\right)^{-1}$.
 Here $a_g$ and $\omega_j$, as
defined in (\ref{ag}) and (\ref{appc:sigma_eig}), respectively,
are given by
 \begin{align}\label{prop:small_omegaj}
   a_g = \frac{1}{2} \sqrt{\frac{\mu}{d_1\bar{\mu}}}
  \cot\left(\frac{\theta}{N}\right) \,; \qquad \omega_1=0 \,, \quad
   \omega_j=\frac{\mu^2}{d_1^2}\csc^2\left(\frac{2\theta}{N}\right)
   \frac{\sin^2\left(\frac{(j-1)}{N}\pi\right)}{\left(-\xi_j+
   \frac{3\zeta_0}{\bar\chi a_gv_{\max 0}}\sqrt{\frac{\mu}{d_1\bar u}}\right)}\,,
   \quad j=2,\ldots,N\,,
 \end{align}
 where $\theta=\sqrt{\mu\bar{\mu}/d_1}$.  The associated eigenvectors
 $\boldsymbol{c}$ are simply the eigenvectors of ${\mathcal G}_g$ as
 given in (\ref{prop411}).
\end{proposition}

As shown below in \S \ref{sec:balance_dae}, the stability threshold of
an $N$-spike steady-state for the small eigenvalues can also be obtained
by first deriving a DAE system for slow spike dynamics and then linearizing
this DAE system about the equilibrium spike locations.

\subsection{Stability Thresholds for the Small Eigenvalues}\label{sec:small_thresh}

In this subsection, we examine the explicit formulae
(\ref{prop:small_eig}) for the small eigenvalues on the range
$d_1\in {\mathcal T}_e$ but with $d_1<d_{1cN}^{\star}$ as given in
(\ref{d1d1cj}). This latter inequality is needed to ensure that the
steady-state is linearly stable with respect to the large eigenvalues
when $\tau=0$. To this end, we write (\ref{prop:small_eig}) in the
more convenient form
\begin{equation}\label{ss:lambda_j}
  \lambda_j=-\frac{2\epsilon^3\beta_0}{3} \bar\chi v_{\max 0}^3h_j \,,
  \qquad \mbox{where} \quad h_j:= \frac{1}{\xi_j} \left(
  \frac{\mu\theta}{d_1} -
  \frac{3\zeta_0\omega_j}{\bar\chi a_gv_{\max 0}}\right)
  +\frac{\bar u\mu}{d_1} a_g\,, \quad j=1,\ldots,N\,,
\end{equation}
where $\omega_j$ and $a_g$ are defined in (\ref{prop:small_omegaj}).
If on the range $d_1\in {\mathcal T}_e$, but with $d_1<d_{1cN}^{\star}$, we
have $h_j>0$ for each $j=1,\ldots,N$, we conclude from (\ref{ss:lambda_j}) that
the $N$-spike steady-state solution is linearly stable with respect to
both the small and large eigenvalues when $\tau=0$.  Alternatively, if
for some $j\in \lbrace{1,\ldots,N\rbrace}$, we have $h_j<0$ on some
range of $d_1\in {\mathcal T}_e$, but with $d_1<d_{1cN}^{\star}$, it
follows that the $N$-spike steady-state is unstable to the small
eigenvalues on this range but is linearly stable to the large
eigenvalues when $\tau=0$.

For any $N\geq 1$, we first establish the sign of $h_1$ in
(\ref{ss:lambda_j}) when $d_1\in {\mathcal T}_e$. By using
$\theta=\sqrt{\mu \bar{u}/d_1}$ together with
(\ref{prop:small_omegaj}) and (\ref{prop411_a}) for $a_g$ and $\xi_1$,
respectively, we use the fact that $\omega_1=0$ in (\ref{ss:lambda_j})
to obtain 
\begin{equation}\label{splot:h1}
  h_1= \frac{\theta^3}{2\bar u} \left[\tan\left(\frac{\theta}{N}\right)+
    \cot\left(\frac{\theta}{N}\right)
  \right] = \frac{\theta^3}{\bar u} \csc\left(\frac{2\theta}{N}\right)\,.
\end{equation}
When $d_1\in \mathcal{T}_e$, we have $d_1>d_{1pN}$ and so we require
that $\theta<{N\pi/2}$. As $d_1\to d_{1pN}$ from above, or
equivalently as $\theta\to {N\pi/2}$ from below, $h_1$ has a vertical
asymptote with $h_1\to +\infty$. However, for $\theta<{N\pi/2}$, we
observe from (\ref{splot:h1}) that $h_1>0$, and so this mode is always
stable for the small eigenvalues.  This leads to the following result:

\begin{proposition}\label{prop:one-spike} For $d_1>d_{1p1}={4\mu\bar{u}/\pi^2}$,
  and in the limit $\epsilon\to 0$, a one-spike steady-state solution
  for (\ref{timedependent1}) is always linearly stable with respect to
  the small eigenvalue.
\end{proposition}

To examine the other mode functions $h_{j}$ for $j=2,\ldots,N$, it is
convenient to write $h_{j}$ in (\ref{ss:lambda_j}) in terms of
$\theta=\sqrt{\bar{u}\mu/d_1}$ rather than $d_1$. To do so, we substitute
(\ref{prop:small_omegaj}) into (\ref{ss:lambda_j}), and observe that
$a_g= \theta (2\bar{u})^{-1} \cot\left({\theta/N}\right)$ and
$v_{\max 0}\bar{\chi} (6\zeta_0)^{-1}={\left(v_{\max 0}\bar{\chi}-2\right)/6}$
upon recalling that $\zeta_0=\left(1-{2/(\bar{\chi}v_{\max 0})}\right)^{-1}$.
In this way, we obtain after some algebra that $h_j$ can be
written explicitly in terms of $\theta$ as
\begin{subequations}\label{small:hj_theta}
\begin{equation}\label{small:hj_theta_1}
  h_{j} = \frac{\theta^3}{2\bar{u}} \cot\left(\frac{\theta}{N}\right)
  \hat{h}_j \,, \qquad
  \hat{h}_j := \frac{1}{\hat{\xi}_j}\left(
    2+ \hat{\xi}_j  - 2
    \frac{\csc^2\left({2\theta/N}\right) \sin^2\left({\pi (j-1)/N}\right)}
    {1 - a_1 {\hat{\xi}_j/2}} \right) \,,
\end{equation}
where we have defined $\hat{\xi}_j$ and re-introduced $a_1$ (see
(\ref{d1d1cj})) as
\begin{equation}\label{small:hj_theta_2}
  a_1:= \frac{1}{3} \left(v_{\max 0} \bar{\chi} -2 \right)\,, \qquad
  \hat{\xi}_j := \xi_j \cot\left(\frac{\theta}{N}\right) \,.
\end{equation}
\end{subequations}

Next, we determine the algebraic sign, the asymptotes, and the
continuity properties of $h_j$ for the modes $j=2,\ldots,N$. By using
(\ref{prop411_b}) of Proposition \ref{prop411} for $\xi_j$ for
$j=2,\ldots,N$, we readily determine the following two equivalent
identities for $\hat{\xi}_j$, as defined in
(\ref{small:hj_theta_2}):
\begin{equation}\label{xi:add}
  \hat{\xi}_{j} =\csc^{2}\left(\frac{\theta}{N}\right)
   \left[\cos\left(\frac{2\theta}{N}\right)-
     \cos\left(\frac{\pi (j-1)}{N}\right)\right]  =
   -2 + 2 \sin^2\left(\frac{\pi (j-1)}{2N}\right)
   \csc^{2}\left(\frac{\theta}{N}\right)\,, \quad j=2,\ldots,N\,.
\end{equation}
For $d_1\in {\mathcal T}_e$ we have that $\theta<\theta_N:={N\pi/2}$
but with $\theta\neq \theta_m:={m\pi/2}$ for $m=1,\ldots,N-1$.  For
any $m=1,\ldots,N-1$, when $d_1\to d_{1Tm}$, or equivalently when
$\theta\to \theta_m$, we conclude from the first identity in
(\ref{xi:add}) that $\hat{\xi}_{m+1}$ vanishes.  As a result, we
observe from (\ref{small:hj_theta_1}) that $\hat{h}_{m+1}$ has an
apparent singularity as $\theta\to\theta_m$, which will require the
evaluation of a singular ${0/0}$ limit. However, by a further
analytical simplification of $\hat{h}_j$, as summarized below in Lemma
\ref{small:explicit}, we can show that this singularity at
$\theta=\theta_m$ is removable.

\begin{lemma}\label{small:explicit} On the range $\theta<\theta_N:={N\pi/2}$,
  we have for $j=2,\ldots,N$ that
  $\lambda_j=-2\epsilon^3\beta_0\bar\chi v_{\max 0}^3{h_j/3}$, where $h_j$
  is given explicitly by
  \begin{equation}\label{hj:explicit}
     h_j = \frac{\theta^3}{\bar{u}}\csc\left(\frac{2\theta}{N}
    \right) \sin^{2}\left(\frac{\pi(j-1)}{2N}\right)
    \frac{ \left[1-a_1 - (1+a_1)\cos\left({2\theta/N}\right)\right]}
    { \left[1 + a_1 \cos\left({\pi(j-1)/N}\right)
        - (1+a_1)\cos\left({2\theta/N}\right)\right]} \,.
  \end{equation}
  It follows that $\lambda_j\to -\infty$ as $\theta\to \theta_N^{-}$.
  In addition, for all $j=2,\ldots,N$, we have $\lambda_j<0$ on the range
  $\theta_{sN}<\theta<\theta_{N}$, where the simultaneous
  zero-crossing threshold $\theta_{sN}$ satisfies
  \begin{equation}\label{hj:zero}
    \theta_{sN}:=\frac{N}{2} \arccos\left( \frac{1-a_1}{a_1+1}\right) \,,
    \quad \mbox{where} \quad a_{1}= \frac{1}{3}\left(\bar{\chi} v_{\max 0}-2
    \right)\,.
  \end{equation}
  Finally, $\lambda_j$ is continuous and satisfies $\lambda_j>0$ for
  all $j=2,\ldots,N$ on the range $\theta_{cN}<\theta<\theta_{sN}$,
  where
\begin{equation}\label{hj:asymp}
  \theta_{cN}:=\frac{N}{2} \arccos\left( \frac{1-a_1\cos\left({\pi/N}\right)
      }{a_1+1}\right) \,.
\end{equation}
This threshold $\theta_{cN}$ is the value of $\theta$ for which
$\hat{h}_N$ has a vertical asymptote. As $\theta$ is decreased below
$\theta_{N}$, it is the mode $j=N$ that first has a vertical
asymptote.  Written in terms of $d_1$, this vertical asymptote is
equivalent to the approximation $d_{1cN}^{\star}$, given in
(\ref{d1d1cj}), for the competition instability threshold associated
with the large eigenvalues when $\tau=0$.
\end{lemma}

\begin{proof}
  We first derive (\ref{hj:explicit}). In the proof it is convenient
  to label $\varphi:={\theta/N}$ and $b:={\pi(j-1)/(2N)}$, so on
  $0<\theta<\theta_N$, and for $j=2,\ldots,N$, we have $0<\varphi<{\pi/2}$
  and $0<b<{\pi/2}$. In terms of $\varphi$ and $b$, (\ref{xi:add}) becomes
  \begin{equation}\label{xi:add_new}
    \hat{\xi}_j=\frac{\left[\cos(2\varphi)-\cos(2b)\right]}{\sin^{2}(\varphi)}
    =-2+ \frac{2\sin^{2}(b)}{\sin^{2}(\varphi)} \,; \qquad
    1-\hat{\xi}_j \frac{a_1}{2}= \frac{1}{\sin^2(\varphi)}\left[\sin^2(\varphi) 
      -\frac{a_1}{2}\left(\cos(2\varphi)-\cos(2b)\right) \right]\,.
   \end{equation}
   By substituting (\ref{xi:add_new}) into (\ref{small:hj_theta_1}), we
   obtain that
  \begin{align*}
    \hat{h}_j &= \frac{2\sin^{2}(\varphi) \sin^{2}(b)}
                {\left[\cos(2\varphi)-\cos(2b)\right]}
   \frac{ \left[ 2-\hat{\xi}_j a_1 - 2 {\sin^{2}(\varphi)\sin^2(2b)/
   \left(\sin^2(2\varphi)\sin^2(b)\right)} \right]}{\left[2\sin^2(\varphi) -
                a_1 \left(\cos(2\varphi)-\cos(2b)\right) \right]} \,.
  \end{align*}
  Next, we use ${\sin(2\omega)/\sin{\omega}}=2\cos{\omega}$ to simplify
  the trigonometric ratio in the numerator of $\hat{h}_j$ to get
  \begin{equation*}
    \hat{h}_j = \frac{2\sin^{2}(\varphi) \sin^{2}(b)}{\cos^{2}(\varphi)
                \left[\cos(2\varphi)-\cos(2b)\right]}
              \frac{ \left[ 2\left(\cos^2(\varphi)-\cos^2(b)\right) -
                  \hat{\xi}_j a_1 \cos^2(\varphi)        \right]}
              {\left[2\sin^2(\varphi) - a_1 \left(\cos(2\varphi)-\cos(2b)\right)
                \right]} \,.
  \end{equation*}
  Upon using $\cos^2(\varphi)-\cos^2(b)={[\cos(2\varphi)-\cos(2b)]/2}$
  together with the first identity for $\hat{\xi}_j$ in
  (\ref{xi:add_new}), we can cancel the common factor
  $\cos(2\varphi)-\cos(2b)$ from the numerator and denominator of
  $\hat{h}_j$, which leaves
  \begin{equation}\label{xj:hj_1}
    \hat{h}_j = \frac{2\sin^{2}(b)\left[\sin^2(\varphi) - a_1\cos^2(\varphi)
      \right]}
    {\cos^2(\varphi)\left[2\sin^2(\varphi) - a_1
        \left(\cos(2\varphi)-\cos(2b)\right) \right]}
  =   \frac{\sin^{2}(b)}{\cos^2(\varphi)}
    \frac{\left[1-a_1-(a_1+1)\cos(2\varphi)\right]}{
              \left[1+ a_1 \cos(2b) - (a_1+1) \cos(2\varphi)\right]} \,.
\end{equation}
Finally, we substitute (\ref{xj:hj_1}) into (\ref{small:hj_theta_1}) and
use $\sin^{2}(b){\cot(\varphi)/\cos^{2}(\varphi)}=2{\sin^2(b)/\sin(2\varphi)}$.
Upon recalling
the definition of $\varphi$ and $b$ we readily obtain the explicit result
(\ref{hj:explicit}).

Next, we let $\theta\to\theta_{N}^{-}$ for which $\cos(2\theta/N)\to-1$
and $\csc(2\theta/N)\to +\infty$. It readily follows from
(\ref{hj:explicit}) that for each $j=2,\ldots,N$, we have
$h_j\to+\infty$ as $\theta\to\theta_{N}^{-}$ when $a_1>0$.  Since
$v_{\max 0}\gg 1$ when $\epsilon \ll 1$, $a_1>0$ must hold. This implies
that $\lambda_j\to -\infty$ as $\theta\to\theta_N^{-}$.  Finally, the
zero-eigenvalue crossing threshold (\ref{hj:zero}) and the value of
$\theta$ for the mode $j=N$ that yields the first vertical asymptote
(\ref{hj:asymp}) as $\theta$ is decreased, are both readily identified
from the numerator and denominators in (\ref{hj:explicit}),
respectively.
\end{proof}

\begin{figure}[h!]
\centering
\begin{subfigure}[t]{0.5\textwidth}
  \includegraphics[width=\linewidth,height=5.0cm]{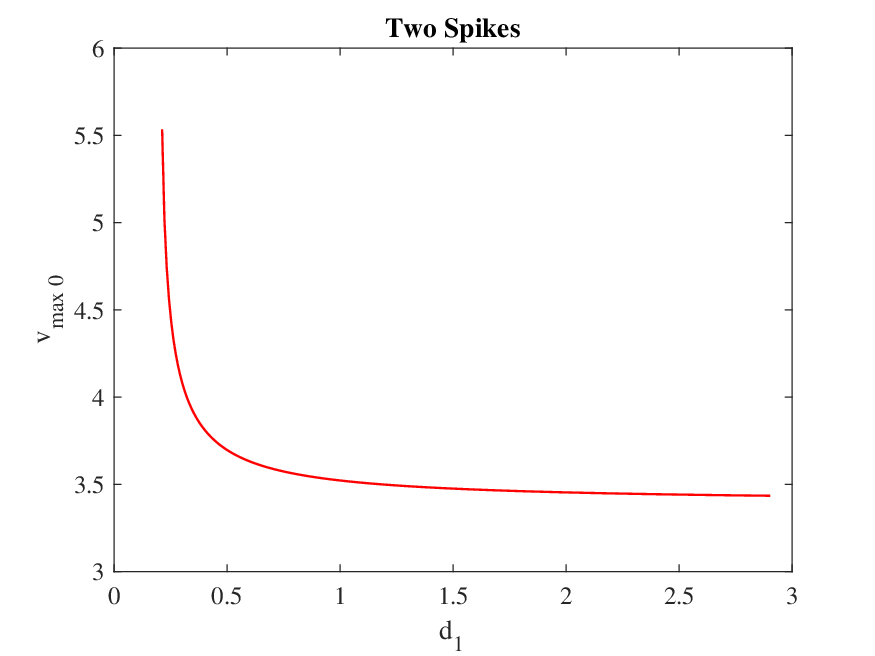}
     \caption{$v_{\max 0}$ versus $d_1$}
     \label{twospikemode2subfig1new}
\end{subfigure}\hspace{-0.25in}
\begin{subfigure}[t]{0.5\textwidth}
  \includegraphics[width=\linewidth,height=5.0cm]{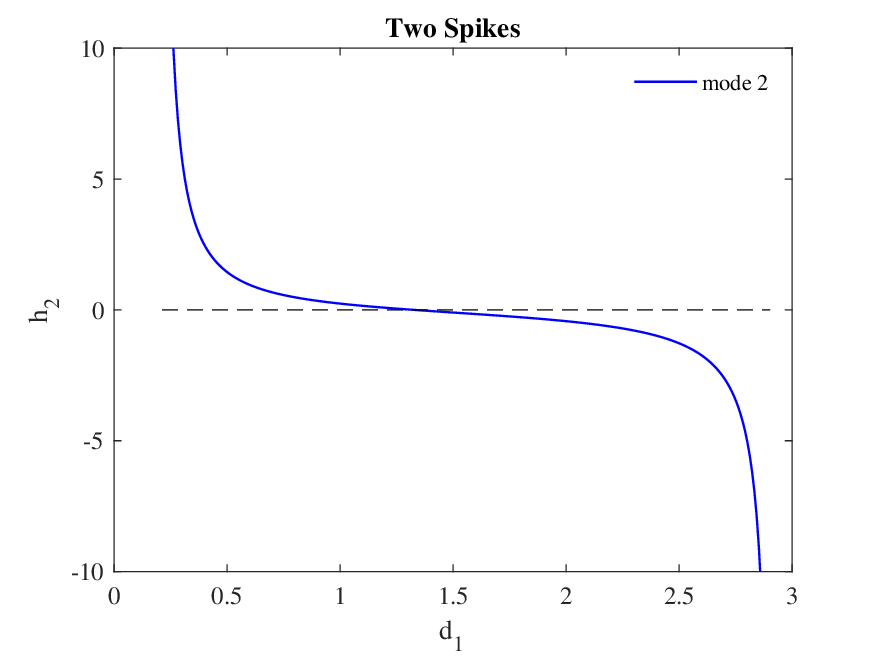}
    \caption{$h_2$ versus $d_1$}
    \label{twospikesmode2subfig2new}
\end{subfigure}
\caption{\fontsize{11pt}{9pt}\selectfont { {\em Numerically computed
      results for $v_{\max 0}$ and $h_2$ on the range
      $d_{1p2}<d_1<d_{1c2}^{\star}$ when $N=2$, $\bar\chi=1$,
      $\bar u=2$, $d_2=0.0004=\epsilon^2$, and $\mu=1$.}}  Here
      $d_{1p2}\approx 0.2$ and $d_{1c2}^{\star}\approx 2.91$. Left:
      $v_{\max 0}$ is monotone decreasing in $d_1$. Right:
      $h_2$ slowly decreases and crosses zero at
      $d_1\approx d_{1s2}\approx 1.61$.}
\label{figuretwospikemode2new}
\vspace{-0.0in}
\end{figure}

To illustrate the implication of Lemma \ref{small:explicit} for $N=2$
as $d_1$ is varied, in the left and right panels of Figure
\ref{figuretwospikemode2new} we plot $v_{\max 0}$ and $h_2$ versus
$d_1$ as computed from (\ref{singlevmax}) and (\ref{hj:explicit}),
respectively.  We conclude from Figure \ref{figuretwospikemode2new}
that the two-spike steady-state is unstable with respect to the small
eigenvalue with mode $m=2$ when
$1.61\approx d_{1s2}<d_1<d_{1c2}^{\star}\approx 2.91$, but is linearly
stable on the range $0.20\approx d_{1p2}<d_1<d_{1s2}\approx
1.61$. Similar results are shown in Figure
\ref{figurethreespikemode23new} for $N=3$ for the same parameter set.
We conclude that a three-spike steady-state is unstable with respect
to the small eigenvalue modes $j=2$ and $j=3$ on the range
$0.70 \approx d_{1s3}<d_1<d_{1c3}^{\star}\approx 0.97$. On the range
$0.09 \approx d_{1p3}<d_1<d_{1s3}\approx 0.70$, the three-spike
steady-state is linearly stable for all the small eigenvalues.

\begin{figure}[h!]
\centering
\begin{subfigure}[t]{0.33\textwidth}
    \includegraphics[width=1.0\linewidth,height=6.5cm]{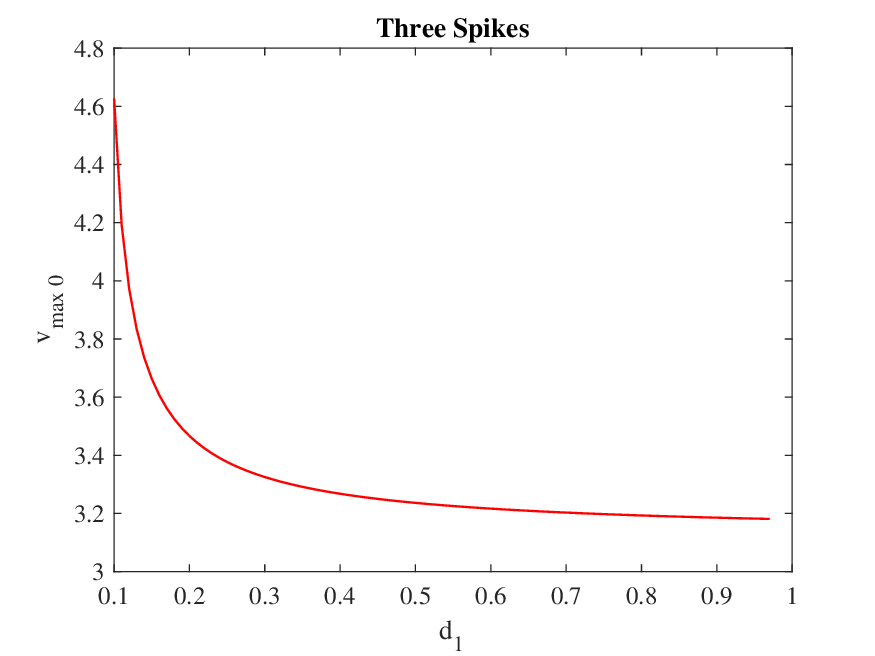}
     \caption{$v_{\max 0}$ versus $d_1$}
     \label{threespikemode2subfig1new}
\end{subfigure}\hspace{-0.25in}
\begin{subfigure}[t]{0.33\textwidth}
  \includegraphics[width=1.0\linewidth,height=6.5cm]{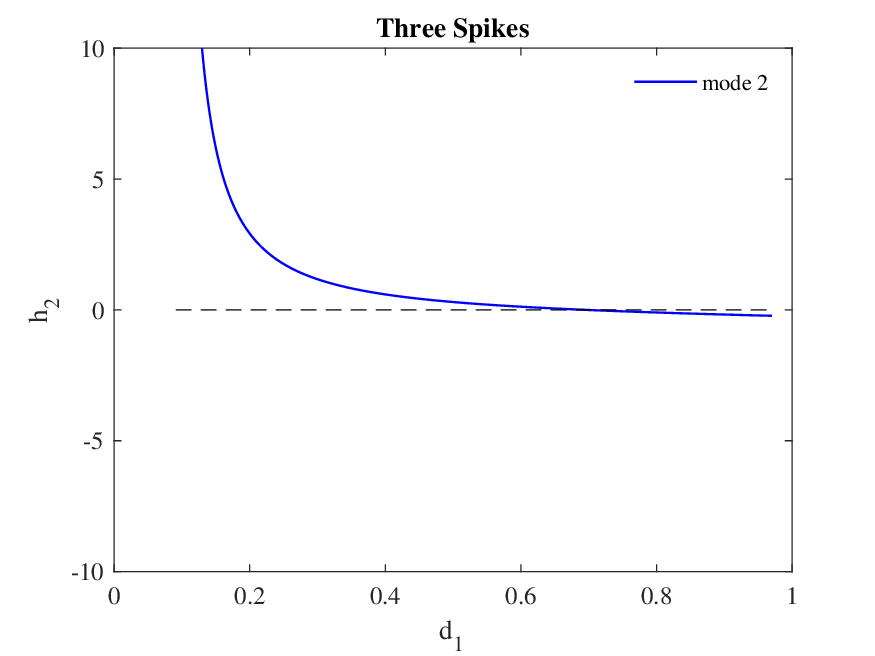}
    \caption{$h_2$ versus $d_1$}
    \label{threespikesmode2subfig2new}
\end{subfigure}
\hspace{-0.25in}
\begin{subfigure}[t]{0.33\textwidth}
  \includegraphics[width=1.0\linewidth,height=6.5cm]{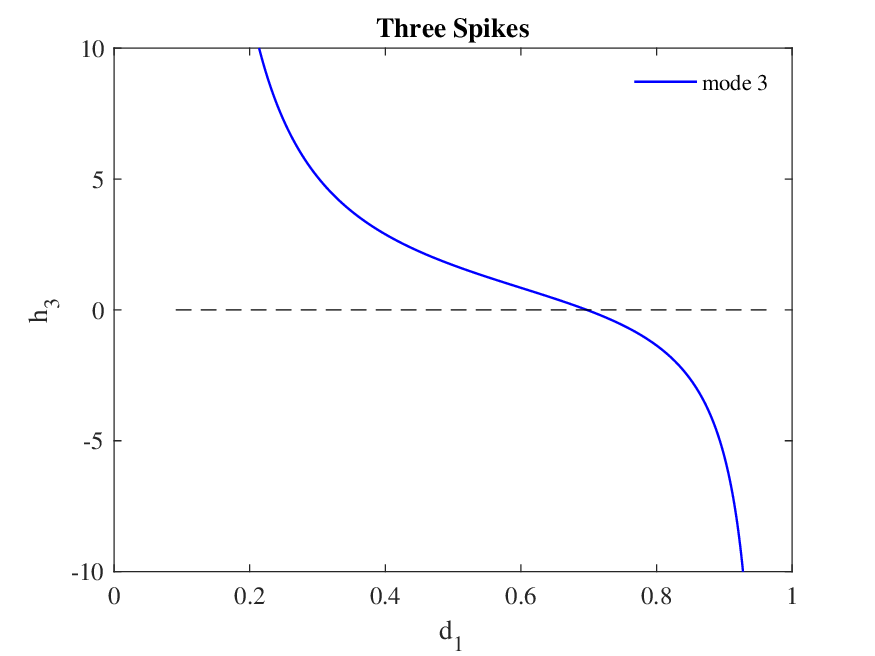}
    \caption{$h_3$ versus $d_1$}
    \label{threespikesmode2subfig3new}
\end{subfigure}
\caption{\fontsize{11pt}{9pt}\selectfont { {\em
Numerically computed
      results for $v_{\max 0}$, $h_2$, and $h_3$  on the range
      $d_{1p3}<d_1<d_{1c3}^{\star}$ when $N=3$, $\bar\chi=1$,
      $\bar u=2$, $d_2=0.0004=\epsilon^2$, and $\mu=1$.}}  Here
  $d_{1p3}\approx 0.09$ and $d_{1c3}^{\star}\approx 0.97$. Left:
  $v_{\max 0}$ is monotone decreasing in $d_1$. Middle and Right:
  $h_2$ and $h_3$ slowly decrease as
  $d_1$ increases and the simultaneous zero crossing occurs at
  $d_{1s3}\approx 0.70$.}
\label{figurethreespikemode23new}
\vspace{-0.0in}
\end{figure}

In summary, in terms of $d_1$, Lemma \ref{small:explicit} shows that
an $N$-spike steady-state solution loses translation stability to
$N-1$ possible modes when $d_{1}$ increases above a critical threshold
$d_{1sN}$. In this way, we obtain our main linear stability result for
$N$-spike steady-state solutions of (\ref{timedependent}).

\begin{proposition}\label{main:stab} For $\tau=0$ and
  $\epsilon\to 0$, an $N$-spike steady-state solution of (\ref{timedependent})
  is linearly stable to both the large and small eigenvalues of the
  linearization when
  \begin{equation}\label{main:p1}
    d_{1pN} < d_{1} < d_{1sN} \,, \quad \mbox{where} \quad
    d_{1pN}=\frac{4\mu\bar{u}}{N^2\pi^2} \,, \quad d_{1sN}:=\frac{4\mu\bar{u}}
    {N^2 \left(\arccos\left(\frac{1-a_1}{1+a_1}\right)\right)^2} \,.
  \end{equation}
  Here $a_1$ is defined in (\ref{hj:zero}). The steady-state is
  unstable to $N-1$ modes of instability for the small eigenvalues,
  but is linearly stable with respect to the large eigenvalues
  when $d_{1sN}<d_1<d_{1cN}^{\star}$. Finally, when
  $d_{1}>d_{1cN}^{\star}$, the steady-state is unstable with respect
  to both the large and small eigenvalues.
 \end{proposition}

 In Appendix \ref{app:asymmetric} we show that the simultaneous
 zero-eigenvalue crossing threshold $\theta_{sN}$ for the small eigenvalues
 occurs precisely at the critical threshold where asymmetric steady-state
 solutions bifurcate from the symmetric steady-state solution branches
 constructed in \S \ref{sec2}.

\begin{figure}[h!]
\centering
\begin{subfigure}[t]{0.45\textwidth}
    \includegraphics[width=\linewidth,height=7.0cm]{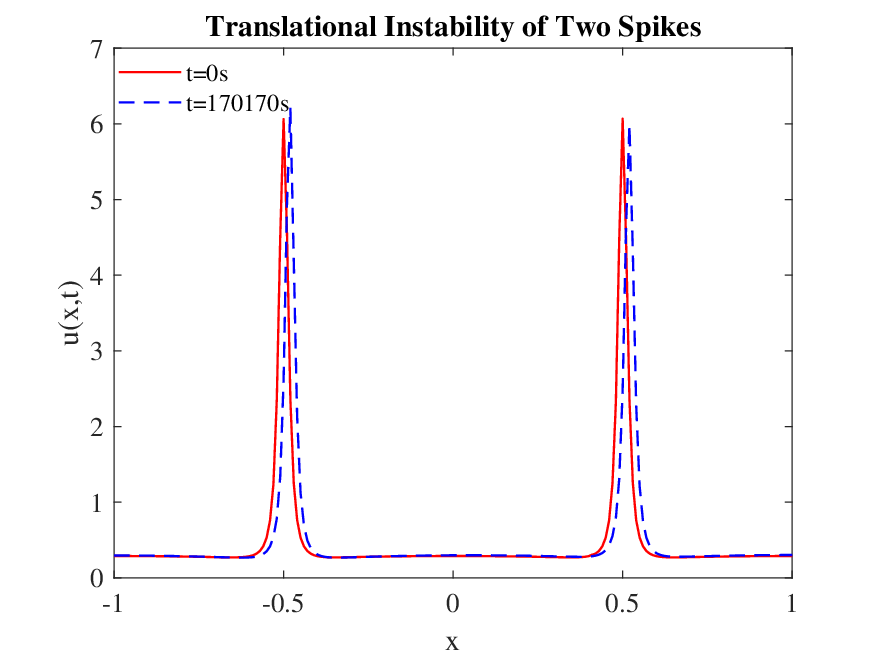}
    \caption{dynamics of $u$, slow drift}
  \end{subfigure}\hspace{-0.25in}
\begin{subfigure}[t]{0.45\textwidth}
  \includegraphics[width=\linewidth,height=7.0cm]{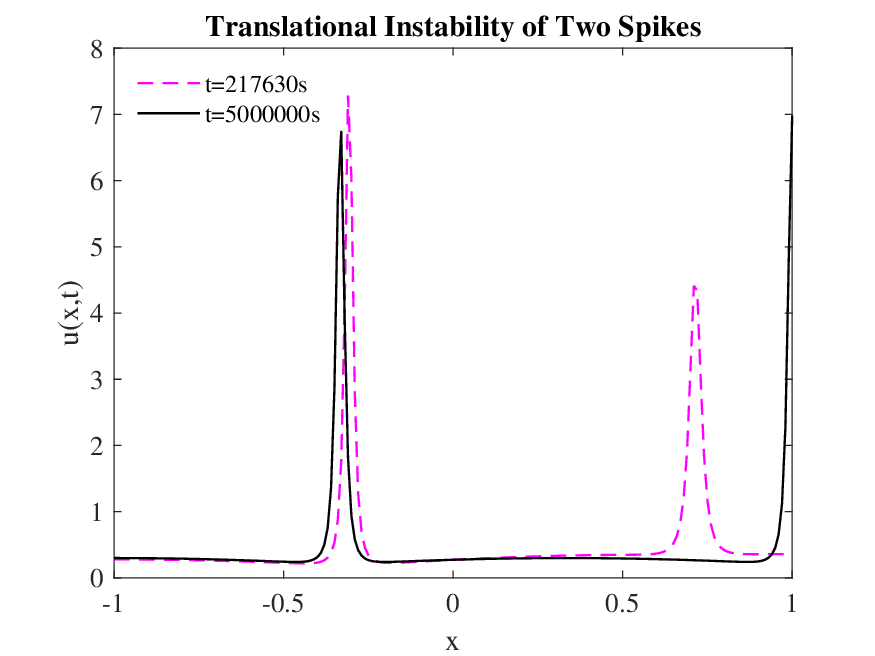}
  \caption{dynamics of $u$, boundary collapse}
\end{subfigure}
\caption{\fontsize{11pt}{9pt}\selectfont { {\em Full PDE simulations
      of (\ref{timedependent}) using FlexPDE7 \cite{flex2021} illustrating
      a translational instabilitiy of a two-spike steady-state when
      $d_1=1.6$, and the long-time behavior for $\bar\chi=1$, $\bar u=2$,
      $d_2=0.0004$ and $\mu=1$.}} Left: snapshots of $u$ at two times
  showing the initial slow motion of a two-spike quasi-equilibrium.
  Right: long time dynamics leads to a final steady-state with an
  interior and a boundary spike.}
\label{dynamicstranslationalinstability}
\end{figure}

In Figure \ref{dynamicstranslationalinstability} we show FlexPDE7
simulations of (\ref{timedependent}) that illustrates a translation
instability for a two-spike pattern when $d_1$ is on the range
$d_{1s}<d_1<d_{1c2}^{\star}$ for the parameter set in the caption of Figure
\ref{figuretwospikemode2new}. For these values, the interior two-spike
steady-state is unstable to the mode $j=2$ small eigenvalue.  The
resulting long-time dynamics leads to a final steady-state that has an
interior and a boundary spike.

Finally, for an otherwise identical parameter set, in Figure
\ref{figurenucleation} we show FlexPDE7 numerical results for
(\ref{timedependent}) for an initial two-spike quasi-steady state
solution as the cellular diffusivity $d_1$ is slowly decreased in time
below the threshold $d_{1p2}$ for which the base-state is unstable to
a Turing instability. This figure illustrates that the instability of
the base state leads to the nucleation of boundary spikes at each
endpoint together with the creation of a new interior spike. The
analysis of this spike nucleation behavior is beyond the scope of this
paper.

\begin{figure}[h!]
\centering
\begin{subfigure}[t]{0.33\textwidth}
    \includegraphics[width=1.0\linewidth,height=7.0cm]{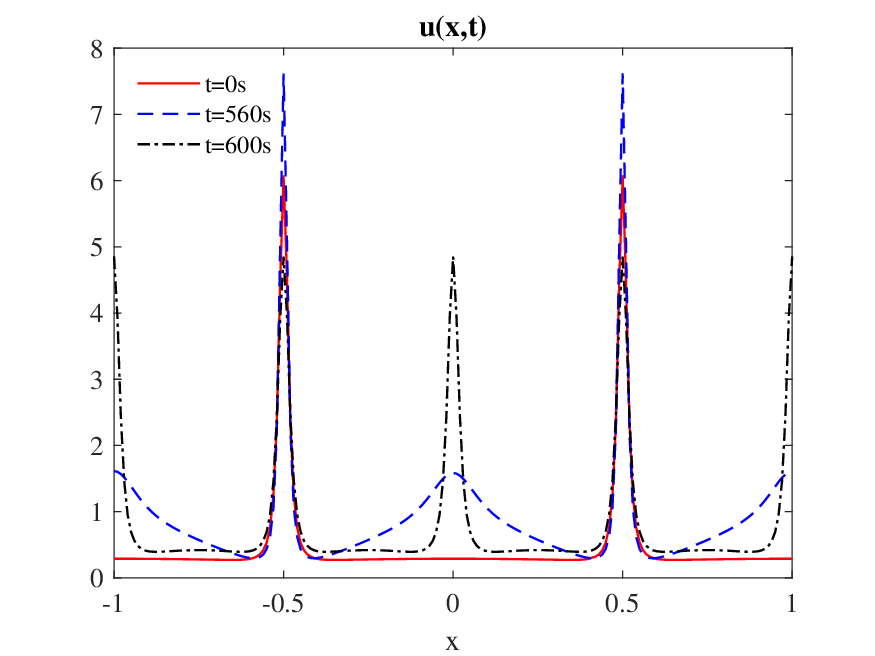}
     \caption{dynamics of $u$}\label{u1:nuc}
\end{subfigure}\hspace{-0.2in}
\begin{subfigure}[t]{0.33\textwidth}
  \includegraphics[width=1.0\linewidth,height=7.0cm]{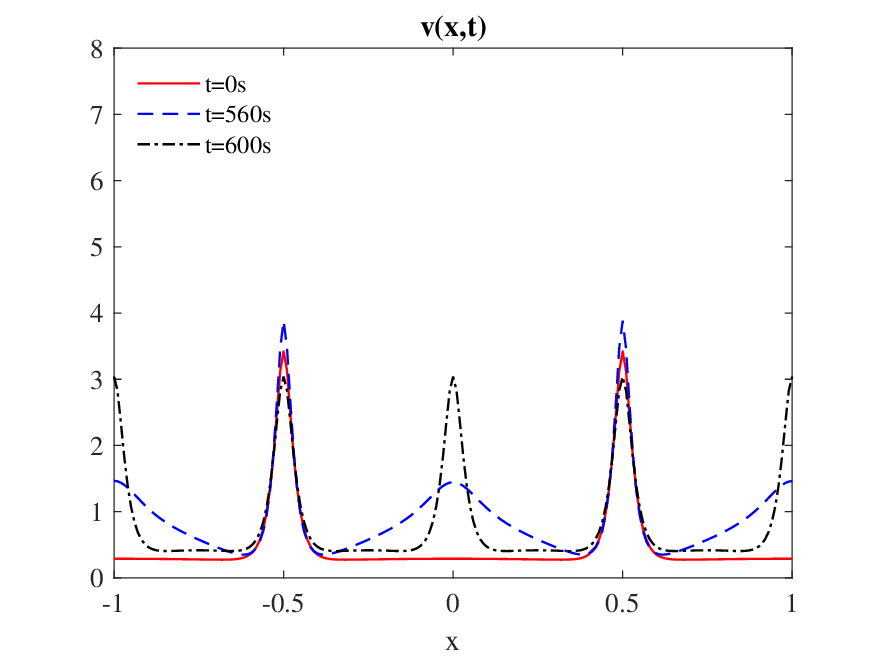}
    \caption{dynamics of $v$}\label{v1:nuc}
\end{subfigure}
\hspace{-0.2in}
\begin{subfigure}[t]{0.33\textwidth}
  \includegraphics[width=1.0\linewidth,height=7.0cm]{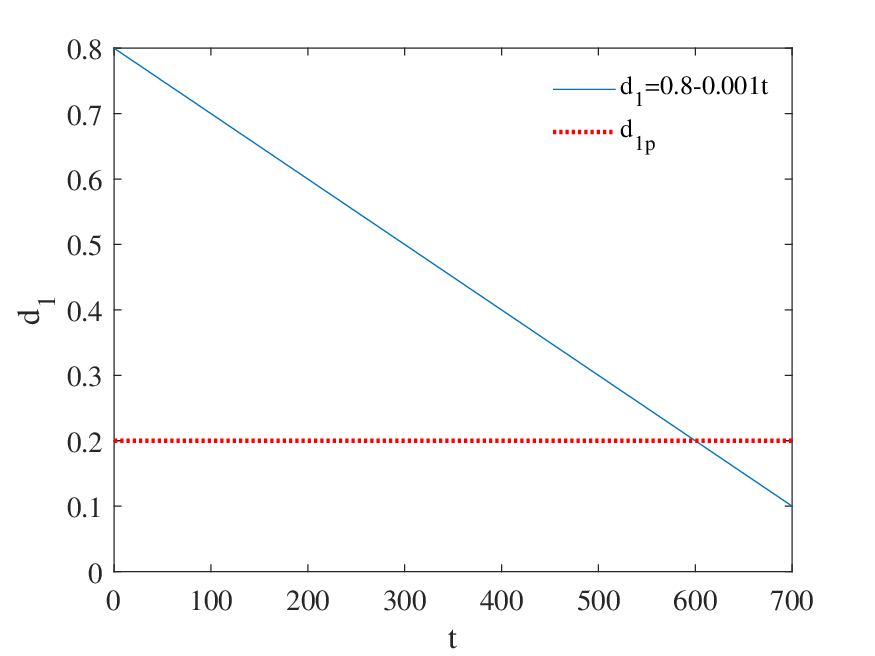}
    \caption{$d_1$ versus $t$}\label{d1:nuc}
\end{subfigure}
\caption{\fontsize{11pt}{9pt}\selectfont { {\em Full PDE simulations
      of (\ref{timedependent}) using FlexPDE7 \cite{flex2021}
      illustrating nucleation behavior for an initial two-spike quasi
      steady-state when $d_1$ is decreased slowly in time.}}  Left and
  Middle: snapshots of $(u,v)$ at three times, showing the spike
  nucleation behavior, with $\bar\chi=1$, $\bar u=2$, $d_2=0.0004$ and
  $\mu=1;$ Right: the diffusivity $d_1$ versus $t$.  As $d_1$
  decreases below $d_{1p2}\approx 0.20$, a new spike is nucleated
  between the two initial spikes and two new boundary spikes are
  created.}
\label{figurenucleation}
\vspace{-0.0in}
\end{figure}

\section{Slow Dynamics of $N$-Spike Quasi-Equilibria}\label{sec5}
Next, we analyze the slow dynamics of an $N$-spike quasi-equilibrium
pattern for (\ref{timedependent}), denoted by $(u_q,v_q)$. Over a long
time-scale, this analysis will characterize how the
spike locations tend to their steady-state values. Similar slow motion
spike dynamics have been derived for other RD systems such as the GM
and Gray-Scott models (\cite{iw2}, \cite{chenwan2011stability},
\cite{dkp}). However, there have been no previous such analyses for
chemotaxis-type RD systems that exhibit slow spike dynamics over
algebraically long time-scales of order ${\mathcal O}(\epsilon^{-p})$
for some $p>0$. In our analysis, we will implicitly assume that the
quasi-equilibrium pattern is linearly stable on ${\mathcal O}(1)$
time-scales, and that the base-state between spikes is linearly stable
in the sense that (\ref{d1:qe}) holds.

Recall that the spatial profile of the $N$-spike quasi-equilibrium
pattern is characterized as in Proposition \ref{prop1}. In this
result, we will now allow the spike locations to depend slowly on time
in that $x_{j}=x_{j}(T)$ where $T=\epsilon^3 t$ is the long time-scale
and with $d_2=\epsilon^2$ in (\ref{timedependent_b}). In the
$j^{\mbox{th}}$ inner region, we introduce the local variables
\begin{equation*}
  y=\epsilon^{-1}[x-x_j(T)]\,, \qquad U(y,T)=
  u\Big(x_j+\epsilon y,\epsilon^{-3} T\Big)\, \qquad
  V(y,T)=v\left(x_j+\epsilon y, \epsilon^{-3} T\right)\,,
\end{equation*}
and we expand the inner solution to (\ref{timedependent}) as
\begin{align}\label{quasiexpansion}
  U(y,T)=U_{0j}[y,x_j(T)]+\epsilon^2 U_{1j}(y,T)+\ldots\,, \qquad
  V(y,T)=V_{0j}[y,x_j(T)]+\epsilon^2 V_{1j}(y,T)+\ldots \,.
\end{align}
Upon substituting \eqref{quasiexpansion} into (\ref{timedependent}),
we obtain from the leading order problem that $(U_{0j},V_{0j})$ satisfy the
core problem (\ref{leading}). Moreover, we obtain from collecting
${\mathcal O}(\epsilon^2)$ terms that, on $-\infty<y<\infty$, $U_{1j}$
and $V_{1j}$ satisfy
\begin{align}\label{residual}
         U_{1j}^{\prime\prime}-\bar\chi\big(U_{0j}V_{1j}^{\prime}\big)^{\prime}-\bar\chi
  \big(U_{1j}V_{0j}^{\prime}\big)^{\prime}+\frac{\mu}{d_1}U_{0j}(\bar u-U_{0j})=0
  \,, \qquad V_{1j}^{\prime\prime}-V_{1j}+U_{1j}=-V_{0j}^{\prime}
         {\dot{x}_j}(T)\,,
\end{align}
where ${\dot{x}_j}(T):=\frac{d}{dT}x_j(T)$ and
${\bar{\chi}}={\chi/d_1}$.  The imposition of a solvability condition
for (\ref{residual}) will yield ${\dot{x}}_j$.

To this end, we decompose $U_{1 j}$ and $V_{1 j}$ into even and odd
parts with respect to $y$ in the form
\begin{align}\label{soldecomposition}
U_{1 j}=U_{1 jE}+U_{1jO}\,, \qquad V_{1}=V_{1 j E}+V_{1jO}\,,
\end{align}
where $U_{1j E}$ (resp. $V_{1jE}$) and $U_{1jO}$ (resp. $V_{1jO}$)
satisfy homogeneous Neumann and Dirichlet boundary conditions at
$y=0,$ respectively.  From substituting (\ref{soldecomposition}) in
(\ref{residual}), we obtain two problems, each defined on
$-\infty<y<\infty$:
\begin{subequations}\label{residualE}
\begin{align}
& U_{1j E}^{\prime\prime}-\bar\chi \big(U_{0j}V^{\prime}_{1jE}\big)^{\prime}-\bar
  \chi (U_{1jE}V_{0j}^{\prime})^{\prime}+\frac{\mu}{d_1}U_{0j}(\bar u-U_{0j})=0\,;
  \qquad U^{\prime}_{1jE}(0)=0 \,, \\
& V_{1jE}^{\prime\prime}-V_{1jE}+U_{1jE}=0\,; \qquad V^{\prime}_{1jE}(0)=0\,,
\end{align}
\end{subequations}
and
\begin{subequations}\label{residualO}
\begin{align}
   &      U_{1jO}^{\prime\prime}-\bar\chi\big(U_{0j}V^{\prime}_{1jO}\big)^{\prime}-
  \bar\chi(U_{1jO}V_{0j}^{\prime})^{\prime}=0\,; \qquad
U_{1jO}(0)=0\,, \\
    &     V_{1jO}^{\prime\prime}-V_{1jO}+U_{1jO}=- V_{0j}^{\prime}
  {\dot{x}_j}(T)\,; \qquad
  V_{1jO}(0)=0\,.
\end{align}
\end{subequations}
Upon defining the functions $g_{1jE}$ and $g_{1jO}$ by
\begin{align}\label{g1def}
  g_{1jE}=\frac{U_{1jE}}{U_{0j}}-\bar\chi V_{1jE}\,, \qquad g_{1jO}=
  \frac{U_{1jO}}{U_{0j}}-\bar\chi V_{1jO}\,,
\end{align}
we can more conveniently rewrite (\ref{residualE}) and (\ref{residualO}) on
$-\infty<y<\infty$ as
\begin{align}\label{divergencg1E}
  \big(U_{0j}g_{1jE}^{\prime}\big)^{\prime}+\frac{\mu}{d_1}U_{0j}
  (\bar u-U_{0j})=0\,,
\quad V^{\prime}_{jE}(0)=0\,; \qquad
  V_{1jE}^{\prime\prime}-V_{1jE}+U_{1jE}=0\,,\quad g_{1jE}^{\prime}(0)=0\,.
\end{align}
and
\begin{align}\label{divergencg1O}
\big(U_{0j}g_{1jO}^{\prime}\big)^{\prime}=0\,, \quad V_{1jO}(0)=0\,; \qquad
          V_{1jO}^{\prime\prime}-V_{1jO}+U_{1jO}=-V_{0j}^{\prime}
  {\dot{x}_j}(T)\,, \quad g_{1jO}(0)=0\,.
\end{align}
By solving the $g$-equation in (\ref{divergencg1E}), we have for
$y\in(0,\infty)$ that
\begin{align}\label{ge1}
  g_{1jE}=\frac{\mu}{d_1} \int_0^y\frac{1}{U_{0j}(\rho)}
  \Big(\int_0^{\rho}U_{0j}(\xi)
  (\bar u-U_{0j}(\xi))\,d\xi\Big)\,d\rho+g_{1jE}(0)\,,
\end{align}
where $g_{1jE}(0)$ is an unknown constant.  In this way, the $V$-equation in
(\ref{divergencg1E}) becomes
\begin{align}\label{sum1}
         V_{1jE}^{\prime\prime}-V_{1jE}+U_{0j}g_{1jE}+ \bar\chi U_{0j}V_{1jE}=0\,,
         \quad -\infty<y< \infty\,; \qquad V_{1jE}^{\prime}(0)=0\,.
\end{align}
Similarly, we can solve the $g$-equation in (\ref{divergencg1O}) to get
\begin{align}\label{g1o}
g_{1jO}=\bar C_j\int_0^y \frac{1}{U_{0j}}\, d\xi\,,
\end{align}
where the constant $\bar C_j >0$ is undetermined.  Then, $V_{1jO}$ in
(\ref{divergencg1O}) satisfies
\begin{align}\label{sum2}
         V_{1jO}^{\prime\prime}-V_{1jO}+U_{0j}g_{1jO}+ \bar\chi U_{0j}V_{1jO}=-
         V_{0j}^{\prime}{\dot{x}_j}(T)\,, \quad
         -\infty<y<\infty\,; \qquad V_{1jO}(0)=0\,.
\end{align}
By adding (\ref{sum1}) and (\ref{sum2}), we obtain that the problem
for $V_1$ can be written in terms of an operator ${\mathcal L}$ as
\begin{align}\label{v1eq}
  {\mathcal L} V_{1j}+  U_{0j}g_{1jO}+U_{0j}g_{1jE}=-
  V_{0j}^{\prime}{\dot{x}}_{j}\,, \qquad \mbox{where} \quad
  \mathcal LV_{1j}:=V_{1j}^{\prime\prime}-V_{1j}+
  \bar\chi U_{0j}V_{1j}\,.
\end{align}
Here $g_{1jE}$ and $g_{1jO}$ are given by (\ref{ge1}) and (\ref{g1o}),
respectively.

To derive our solvability condition for (\ref{v1eq}), we multiply
(\ref{v1eq}) by $V_{0j}^{\prime}$ and integrate the resulting
expression over $(-\rho,\rho)$ with $\rho$ large. This yields that
\begin{align}\label{515insec5slowdynmaics}
  &\lim_{\rho\rightarrow+\infty}\bigg(\int_{-\rho}^{\rho} V_{0j}^{\prime}
    \mathcal LV_{1j} \,
    dy+\int_{-\rho}^\rho U_{0j} g_{1jO}V_{0j}^{\prime}\, dy+\int_{-\rho}^\rho
    U_{0j}g_{1jE}V_{0j}^{\prime} \, dy\bigg)=- \dot{x}_{j}
    \lim_{\rho\rightarrow+\infty}\int_{-\rho}^\rho\left(V_{0j}^{\prime}\right)^2
    \,dy\,.
\end{align}
To simplify \eqref{515insec5slowdynmaics}, we invoke Green's second
identity in the form
\begin{align}\label{solvability}
  &\lim_{\rho\rightarrow+\infty}\bigg(\int_{-\rho}^{\rho} V_{0j}^{\prime} \mathcal
    LV_{1j}\,
    dy-\int_{-\rho}^{\rho} V_{1j}\, \mathcal L V_{0j}^{\prime} \, dy\bigg)=
    \lim_{\rho\rightarrow +\infty }V_{1j}^{\prime}V_{0j}^{\prime}\Big{\vert}_{-\rho}^{\rho}-
    \lim_{\rho\rightarrow+\infty}V_{0j}^{\prime\prime} V_{1j}\Big{\vert}_{-\rho}^{\rho}\,.
\end{align}
Since $V_{0j}^{\prime}$ and $V_{0j}^{\prime\prime}$ are exponentially small as
$\vert y\vert\to\infty$, while $\mathcal LV_{0j}^{\prime}=0$, we conclude
from (\ref{solvability}) that
\begin{align}\label{condition1}
  \lim_{\rho\rightarrow+\infty}\int_{-\rho}^{\rho} V_{0j}^{\prime} {\mathcal L} V_{1j} \,
  dy  =0\,.
\end{align}
Moreover, since $U_{0j}$ and $g_{1jE}$ are even, while
$V_{0j}^{\prime}$ is odd, we get that
$\int_{-\rho}^\rho U_{0j}g_{1jE}V_{0j}^{\prime} \, dy=0$.  By using
this result together with (\ref{condition1}) in
(\ref{515insec5slowdynmaics}) and letting $\rho\to\infty$, we obtain from
(\ref{g1o}) the solvability condition
\begin{align}\label{solvability1}
  {\dot{x}}_j = \bar C_j \, \beta_j \,, \qquad \mbox{where} \quad
  \beta_j := -\frac{\int_0^\infty U_{0j}V_{0j}^{\prime}
  \left(\int_0^y\frac{1}{U_{0j}}\,
  d\xi\right) \, dy}{\int_0^\infty \left(V^{\prime}_{0j}\right)^2 \,dy}\,.
\end{align}
This expression determines the speed ${\dot{x}}_j$ of the spike in terms
of the, as yet, undetermined constant $\bar C_j$.

Our final step in the analysis is to formulate a matching condition
between the inner and outer solutions so as to determine $\bar
C_j$.  To do so, we find from (\ref{g1def}) and (\ref{g1o}), and
together with the relation $U_{0j}^{\prime}=\bar{\chi}U_{0j}
V_{0j}^{\prime}$ from the core problem (\ref{leading}), that the odd
part of the inner solution in the $j^{\mbox{th}}$ region satisfies
\begin{align}\label{U0g1C1}
  U_{0j}g_{1jO}^{\prime}=U_{1jO}^{\prime}-\bar\chi\big(U_{0j} V_{1jO}^{\prime}\big)-
  \bar\chi\left(U_{1jO}V_{0j}^{\prime}\right)=\bar  C_j\,.
\end{align}
Owing to the exponential decay of $V_{0j}^{\prime}$ and $U_{0j}^{\prime}$ as
$y\to\pm \infty$, as shown in \S \ref{sec2}, we obtain that the far-field
behavior of the derivative of the odd part of the inner solution must
satisfy
\begin{align}\label{innerfarfieldbehaviorinsec5}
  U_{1jO}^{\prime} \sim V^{\prime}_{1jO}\sim \bar C_j\,, \quad \mbox{as} \quad
  y\rightarrow \pm\infty\,.
\end{align}

In the outer region we have, as a result of the slow time dependence,
that the outer solution satisfies $u_o\sim w_o$, where $w_o$ was given
in (\ref{woouter}) of \S \ref{sec2} in our analysis of the
quasi-equilibrium pattern. From (\ref{woouter}), we have that 
\begin{equation*}
  u_o\sim w_o\sim \frac{2\bar\chi}{3 }\epsilon\sum_{k=1}^N v_{\max k}^3G(x;x_k)
  = \frac{2\bar\chi}{3 }\epsilon\sum\limits_{k\not=j} v_{\max k}^3 G(x;x_k)+
  \frac{2\bar\chi}{3 }\epsilon v_{\max j}^3 G(x;x_j) \,,
\end{equation*}
where $G(x;x_k)$ is the Helmholtz Green's function of (\ref{greenfunction}).

To proceed, it is convenient to decompose $G(x;x_k)$ globally on
$-1<x<1$ as
\begin{align}\label{g:decomp}
  G(x;x_k) = K(|x-x_k|) + R(x;x_k)\,, \quad \mbox{where} \quad
  K:=\frac{\mu}{2d_1}|x-x_k|\,.
\end{align}
Here $K$ is the singular part of $G$, while $R$ is the smooth regular
part.  By expanding $G(x;x_k)$ as $x\rightarrow x_j$, we get
\begin{align}\label{Gexpansion}
G(x;x_k)\sim\left\{\begin{array}{ll}
 G(x_j;x_k)+G_x(x_j;x_k)(x-x_j)\,,&j\not=k\,,\\
 K(\vert x-x_k\vert)+R(x_k;x_k)+R_x(x_k;x_k)(x-x_k)\,,&j=k\,.
\end{array}
\right.
\end{align}
Upon using (\ref{Gexpansion}), we obtain that the limiting
behavior of $u_{ox}\sim w_{ox}$ as we approach the $j^{\mbox{th}}$ spike is
\begin{align}\label{wxqouter}
 u_{ox}\sim  w_{ox}\sim  \frac{2\bar\chi}{3 }\epsilon\sum\limits_{k\not=j} v_{\max k}^3
  G_x(x_j;x_k)+ \frac{2\bar\chi}{3 }\epsilon v_{\max j}^3 R_x(x_j;x_j) +
  \frac{2\bar\chi}{3 }\epsilon v_{\max j}^3 \, K_{x}^{\pm}(x_j;x_j)\,, \quad
  \mbox{as} \quad x\rightarrow x_j^{\pm}\,,
\end{align}
where $K_x^{\pm}:=\pm\frac{\mu}{2d_1}$. To find ${\bar C}_j$, we
use the matching condition that $\epsilon^{2}U_{1j0}^{\prime}$ as
$y\to \pm \infty$ must agree with $\epsilon u_{ox}$ as $x\to x_j$,
when we include only the first two terms on the right-hand side of
(\ref{wxqouter}). This determines ${\bar C}_j$ as
\begin{align}\label{c1j}
  {\bar C}_j=\frac{2\bar\chi}{3 }\sum\limits_{k\not=j}^N v_{\max k}^3
  G_x(x_j;x_k)+\frac{2\bar\chi}{3 } v_{\max j}^3 R_x(x_j;x_j)\,.
\end{align}
Since $U_{1jE}^{\prime}$ is an odd function, the last
term in (\ref{wxqouter}) must match with the far-field behavior of
$\epsilon^2 U_{1jE}^{\prime}$. However, since this explicit matching
requirement does not affect our solvability condition, it is not
performed here.

Upon substituting (\ref{c1j}) into (\ref{solvability1}), we obtain a
coupled nonlinear ODE system for the spike locations in the
quasi-equilibrium pattern. In our ODE system, $v_{\max j}$ and
$\beta_j$ must be calculated by using the nonlinear algebraic system
(\ref{algebraicreal}) for $C_j$, $s_j$ and $v_{\max j}$.  This leads
to a differential algebraic system (DAE) of ODE's characterizing slow
spike dynamics for (\ref{timedependent}), which we summarize in the
following formal proposition.

\begin{proposition}\label{propositon51insec5}
  For (\ref{timedependent}) with $d_2=\epsilon^2\ll 1$ and where
  $d_1\in {\mathcal T}_e$, as defined in (\ref{d1:admiss}), assume
  that the quasi-equilibrium pattern is linearly stable with respect
  to the large eigenvalues and that (\ref{d1:qe}) holds. Then, the
  slow dynamics of a collection $x_1,\ldots, x_N$ of spikes satisfies
  the DAE system:
\begin{align}\label{DAE1}
  \frac{dx_{j}}{dt}\sim \frac{2\bar\chi}{3 }\epsilon^3 \beta_j
  \left(\sum\limits_{k\not=j}^N v_{\max k}^3 G_x(x_j;x_k)+
  v_{\max j}^3 R_x(x_j;x_j)\right)\,, \qquad \left\{\begin{array}{ll}
C_ke^{\bar \chi s_k}-s_k=0\,,\\
              -\frac{1}{2} v_{\max k}^2+\frac{1}{2}s^2_k+
       \frac{C_k}{\bar\chi}e^{\bar\chi v_{\max k}}-\frac{s_k}{\bar\chi}=0\,,\\
s_k=\frac{2\bar\chi}{3 }a_gv_{\max k}^3\epsilon \,,
\end{array}
\right.
\end{align}
where $j=1,\ldots,N.$ Here $\beta_j$ is defined in
(\ref{solvability1}) with the asymptotics (\ref{betajvalue}). The
Green's functions $G(x;x_k)$ and its regular part $R_x$ can be found
explicitly from (\ref{greenfunction}).  In particular, the locations
$x_j^0$, for $j=1,\ldots,N$, of the $N$-spike true steady-state
solution, are the equilibrium point of the slow dynamics and satisfy
\begin{align}\label{ss:balance}
  \sum\limits_{k\not=j}^N G_x\big(x_j^0;x_k^0\big)+
  R_x\big(x_j^0;x_j^0\big)=0\,, \quad j=1,\ldots,N \,.
\end{align}
\end{proposition}

Proposition \ref{propositon51insec5} characterizes the
slow dynamics of an $N$-spike quasi-equilibrium solution on the long
${\mathcal O}\left(\epsilon^{-3}\right)$ time-scale. We remark that
this time-scale is longer than the
${\mathcal O}\left(\epsilon^{-2}\right)$ time-scale of slow spike
dynamics for the GM and Gray-Scott models (\cite{iw2},
\cite{chenwan2011stability}, \cite{dkp}), where there are no
chemotactic effects.

In Appendix \ref{appendixFfinal}, we show that $\beta_j$, as given in
(\ref{solvability1}), can be calculated asymptotically by retaining only the
contribution from the sub-inner solution. In particular, in
Appendix \ref{appendixFfinal} we provide the leading order estimate
\begin{align}\label{betajvalue}
\beta_j\sim \frac{2}{v_{\max j}}\,, \qquad \mbox{for} \,\,\, v_{\max j}\gg 1\,.
\end{align}
Moreover, in Appendix \ref{appendixFfinal} we show at the steady-state spike
locations that $\beta_j=\beta_0$ $\forall j$, with $\beta_0$ given in
(\ref{433insection4}).

To illustrate our results, we now compare the dynamics computed from
the DAE system (\ref{DAE1}) with corresponding numerical results
computed from the full PDE system (\ref{timedependent}) using FLEXPDE7
\cite{flex2021}. In our comparison, we computed the integrals defining
$\beta_j$ numerically from (\ref{solvability1}). The results for a one- and
two-spike dynamics are shown in Figure \ref{figuredyna} for the
parameter values in the figure caption.  In Figure \ref{subfig1},
where we chose the initial condition $x_1(0)=-0.1$, the asymptotic and
numerical spike trajectories are favorably compared for a one-spike
quasi-equilibrium pattern. In Figure \ref{subfig2} a similar favorable
comparison is shown for the case of two-spike dynamics starting from
the initial condition $x_1(0)=-0.6$ and $x_2(0)=0.6$.

\begin{figure}[h!]
\centering
\begin{subfigure}[t]{0.5\textwidth}
    \includegraphics[width=1.0\linewidth,height=7.0cm]{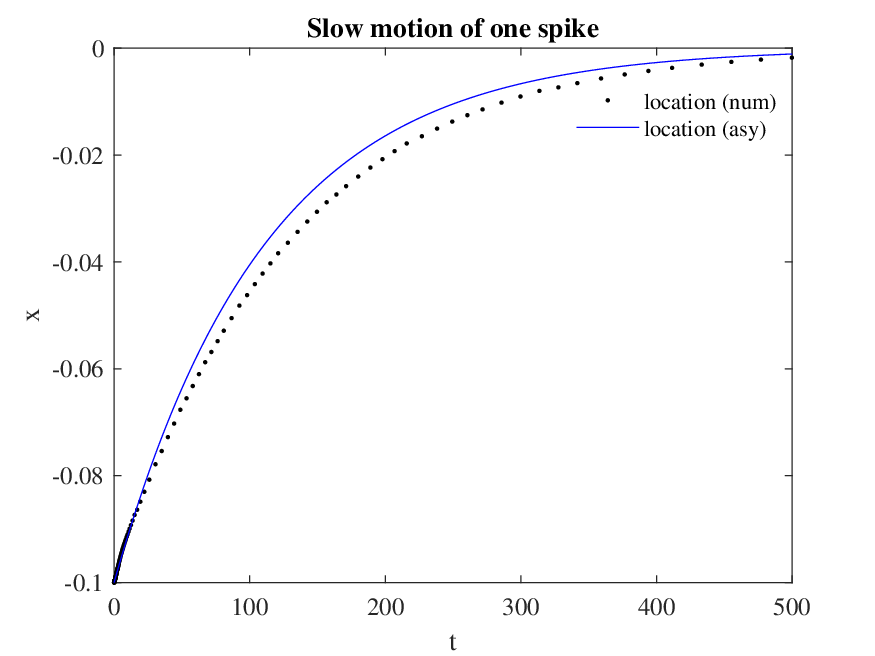}
     \caption{one-spike slow dynamics}
     \label{subfig1}
\end{subfigure}\hspace{-0.25in}
\begin{subfigure}[t]{0.5\textwidth}
  \includegraphics[width=1.0\linewidth,height=7.0cm]{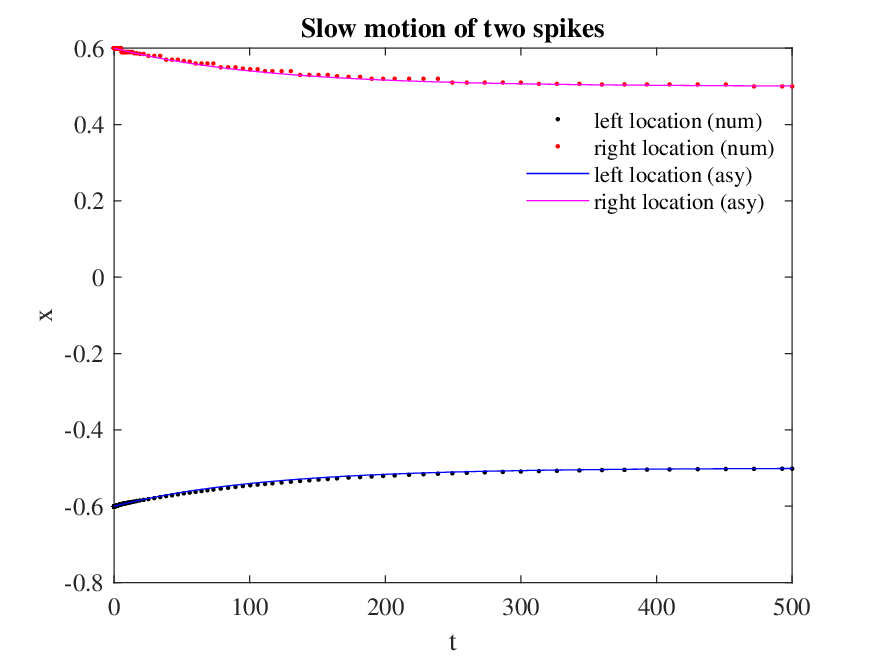}
    \caption{two-spike slow dynamics}
    \label{subfig2}
\end{subfigure}
\caption{\fontsize{11pt}{9pt}\selectfont { {\em Slow dynamics of one-
      and two-spike quasi-equilibria for (\ref{timedependent}) with
      different parameter values.  Left: $d_1=\chi=1$, $\bar u=2$,
      $d_2=0.02=\epsilon^2$ and $\mu=0.25;$ Right: $d_1=\chi=1,$
      $\bar u=2$, $d_2=0.005=\epsilon^2$, $\mu=1$.}  The solid curves
    are the results from the asymptotic DAE system (\ref{DAE1}).  The
    dotted curves are the results obtained from the full numerical PDE
    simulation of (\ref{timedependent}) using \cite{flex2021}. In the
    numerical results, the center of the spike is chosen as the
    maximum of $u$ on the computational grid. Observe the slow
    dynamics towards the equilibrium spike locations.}}
\label{figuredyna}
\vspace{-0.0in}
\end{figure}

\subsection{Computation of Jacobian Matrix for Balancing Conditions}\label{sec:balance_dae}

In this subsection, and as remarked in \S \ref{sec4}, we show that
when $d_1\in {\mathcal T}_{e}$ the matrix ${\mathcal M}$ in
(\ref{jacobianbutinsmallep}) arises from the linearization of the DAE
dynamics (\ref{DAE1}) in Proposition \ref{propositon51insec5} about
the steady-state spike locations. Our approach below is inspired by a
related analysis for the GM model in \cite{wei2007existence}.

To this end, we use the Green's function in
(\ref{greenequation}) together with its decomposition in
(\ref{g:decomp}) to define
\begin{align}\label{computation533insec51new}
\partial_{x_j}G(x_j;x_k) :=\left\{\begin{array}{ll}
\frac{\partial R}{\partial x}(x;x_j)\vert_{x=x_j}\,,&j=k\,,\\
\frac{\partial G}{\partial x}(x;x_k)\vert_{x=x_j}\,,&j\not=k\,,
\end{array}
\right. \qquad
\partial_{x_j}\partial_{x_k}G(x_j;x_k)=\left\{\begin{array}{ll}
 \frac{\partial}{\partial x}|_{x=x_j}\frac{\partial }{\partial y}|_{y=x_k}R(x;y)\,,
                                                &j=k\,,\\
\partial_{x_j}\partial_{x_k} G(x_j;x_k)\,,&j\not=k\,.
\end{array}
\right.
\end{align}
Here $R$ is the smooth regular part of $G$ as defined in
(\ref{g:decomp}).  Next, we denote the $N\times N$ matrices
$\nabla \mathcal G$, $(\nabla \mathcal G)^T$, and
$\nabla^2 \mathcal G$ evaluated at the steady-state spike locations by
\begin{align}\label{daes:g_mats}
  \nabla \mathcal G:=(\partial_{x_j} G(x_j^0;x_k^0))_{N\times N}\,, \qquad
  (\nabla \mathcal G)^T:=(\partial_{x_k}G(x_j^0;x_k^0))_{N\times N}\,, \qquad
  \nabla^2 \mathcal G:=(\partial_{x_j}\partial_{x_k} G(x_j^0;x_k^0))_{N\times N}\,.
\end{align}
The relationship between these matrices and the matrices
${\mathcal P}$, ${\mathcal P}_g$, and ${\mathcal G}_g$, as defined in
(\ref{matrixGg}), (\ref{matrixPg}), and (\ref{matrixGg}),
respectively, that were used in our analysis of the small eigenvalues
in \S \ref{sec4}, is clarified in Appendix \ref{appendixG}.

In our analysis, it is convenient to write the DAE system (\ref{DAE1})
in the form
\begin{align}\label{balancjacobi2}
  \frac{dx_j}{dt} = \frac{2\bar{\chi}}{3} \epsilon^3 \beta_j {\mathcal F}_j
  \,, \qquad \mbox{where} \qquad
  \mathcal F_j: = \sum\limits_{k=1}^N v_{\max k}^3 \partial_{x_j}
  G(x_j;x_k)\,, \qquad j=1,\ldots,N\,.
\end{align}
Our goal below is to compute the Jacobian matrix ${\mathcal J}:=
\Big(\frac{\partial \mathcal F_j}{\partial x_i}\Big)_{N\times N}$  that arises
by linearizing the DAE system (\ref{DAE1}) around the steady-state spike
locations. More specifically, if we introduce the perturbation
\begin{equation*}
  x_j = x_{j}^{0} + c_j e^{\lambda t} \,, \qquad j=1,\ldots,N\,,
\end{equation*}
into (\ref{balancjacobi2}), the linearization of the DAE system yields
the matrix eigenvalue problem
\begin{equation}\label{dae:linearize}
  \lambda \boldsymbol{c} = - \epsilon^3 \beta_0 \tilde{\mathcal M}
  \boldsymbol{c} \,,
  \qquad  \tilde{\mathcal M} := - \frac{2\bar\chi}{3} {\mathcal J} \,,
  \qquad \mbox{where} \qquad
  {\mathcal J}:= \left(\frac{\partial \mathcal F_j}{\partial x_i}
  \right)_{N\times N}\,,
\end{equation}
where $\boldsymbol{c}:=(c_1,\ldots,c_N)^T$.  From an explicit
calculation of $\mathcal J$ given below, we will show via
(\ref{dae:linearize}) that $\tilde{\mathcal M}$ is identical to the
matrix ${\mathcal M}$ as given in (\ref{jacobianbutinsmallep}), which
was derived in our analysis of the small eigenvalues.

To calculate the Jacobian, we first differentiate $\mathcal F_j$ in
(\ref{balancjacobi2}) with respect to $x_i$ to obtain
\begin{align}\label{daes:fjxi}
  \frac{\partial \mathcal F_j}{\partial x_i}  = \sum\limits_{k=1}^N 3v_{\max k}^2
  \left(\partial_{x_i} v_{\max k} \right) \partial_{x_j} G(x_j;x_k) +
  \sum\limits_{k=1}^N v_{\max k}^3 \partial_{x_i} \left[
  \partial_{x_j} G(x_j;x_k)\right]  \,, \qquad j=1,\ldots,N\,.
\end{align}
By using the nonlinear algebraic system in (\ref{DAE1}), in
(\ref{diffvmaxi}) of Appendix \ref{appendix:vderiv} we calculate
$\partial_{x_i} v_{\max k}$, so as to obtain
\begin{align}\label{Fj;xi}
  \frac{\partial \mathcal F_j}{\partial x_i}  \sim -3 \sum\limits_{k=1}^N
  \frac{v_{\max k}^2 \zeta_{\max k}}{\bar{\chi} s_k} \left(\partial_{x_i} s_{k}
\right) \partial_{x_j} G(x_j;x_k) +
  \sum\limits_{k=1}^N v_{\max k}^3 \partial_{x_i} \left[
  \partial_{x_j} G(x_j;x_k)\right]  \,, \qquad j=1,\ldots,N\,,
\end{align}
where
$\zeta_{\max k}=\left(1 - {2/(\bar{\chi} v_{\max k})}\right)^{-1}$.
To determine $\partial_{x_i} s_{j}$, as needed in (\ref{Fj;xi}), we
differentiate (\ref{quasisj}) in $x_i$ to get
\begin{align}\label{xisj}
  \partial_{x_i}s_j=\frac{2\bar\chi\epsilon}{3}\sum_{k=1}^N
  \bigg[3v_{\max k}^2 \left(\partial_{x_i}v_{\max k} \right) G(x_j;x_k)+v_{\max k}^3
  \partial_{x_i} G(x_j;x_k)\bigg]\,.
\end{align}
By using (\ref{diffvmaxi}) of Appendix \ref{appendix:vderiv}
to estimate $\partial_{x_i} v_{\max k}$, we obtain for $\epsilon\to 0$ that
\begin{align}\label{dae3:vk_xi}
  \partial_{x_i}s_j \sim -2\epsilon\sum_{k=1}^N
   \frac{v_{\max k}^2\zeta_{\max k}}{s_k}
  \left(\partial_{x_i}s_k \right)\, G(x_j;x_k)+ \frac{2\bar{\chi}\epsilon}{3}
  \sum_{k=1}^{N} v_{\max k}^3\partial_{x_i} G(x_j;x_k) \,, \quad j=1,\ldots,N\,.
\end{align}
Then, by calculating the second term in (\ref{dae3:vk_xi}), we get
\begin{align}\label{dae:vk_xi}
  \partial_{x_i}s_j\sim -
  2\epsilon\sum_{k=1}^N
   \frac{v_{\max k}^2\zeta_{\max k}}{s_k}
 \left( \partial_{x_i}s_k \right) G(x_j;x_k) +
  \left\{\begin{array}{ll}
           \frac{2\bar{\chi}\epsilon}{3} v_{\max i}^3 \partial_{x_i}G(x_j;x_i)\,,
           \quad &i\neq j\,,\\
             \frac{2\bar{\chi}\epsilon}{3} v_{\max i}^3 \partial_{x_i}G(x_j;x_i) +
             \frac{2\bar{\chi}\epsilon}{3} \sum_{k=1}^{N} v_{\max k}^3
             \partial_{x_i} G(x_j;x_k)\,,\quad  &i=j\,.
    \end{array}
\right.
\end{align}

Next, we evaluate (\ref{dae:vk_xi}) at the equilibrium solution where
$x_j=x_j^{0}$, for which $s_k=s_0$, $v_{\max k}=v_{\max 0}$, and $\zeta_{\max k}=
\zeta_0$, where $s_0={2\bar{\chi} a_g v_{\max 0}^3\epsilon/3}$. Moreover, at
the steady-state, we use the equilibrium condition (\ref{ss:balance}) to
eliminate the last sum in (\ref{dae:vk_xi}) that
holds when $i=j$. In this way, (\ref{dae:vk_xi}) reduces to
\begin{align}\label{dae2:vk_xi}
  \partial_{x_i}s_j\sim -
  \frac{3\zeta_0}{\bar{\chi} a_g v_{\max 0}} \sum_{k=1}^N
  \left( \partial_{x_i}s_k \right) G(x_j^{0};x_k^{0}) +
  \frac{s_0}{a_g} \partial_{x_i}G(x_j^{0};x_i^{0})\,, \quad i,j = 1,\ldots,N
  \,,
\end{align}
when evaluated at the steady-state. By introducing
$\boldsymbol{s}:=(s_1,\ldots,s_N)^T$ and
$\nabla:=(\partial_{x_{1}},\ldots,\partial_{x_N})$, 
 (\ref{dae2:vk_xi}) can be written in
matrix form as
\begin{align}\label{daes:smat} 
  \nabla \boldsymbol{s}\sim \frac{s_0}{a_g}\left(I+
  \frac{3\zeta_0}{\bar{\chi} a_g v_{\max 0}} \mathcal G\right)^{-1}
  \left(\nabla \mathcal G\right)^T\,, \qquad
  \mbox{where} \quad \nabla \boldsymbol{s}:=(\nabla s_1,\ldots,\nabla
s_N)^T\vert_{x_j=x_j^0,~j=1,\ldots,N}\,.
\end{align}
Here $\mathcal G:=(G(x_j^0;x_k^0))_{N\times N}$ is the Green's matrix
at the steady-state and $\left(\nabla \mathcal G\right)^T$ is defined
in (\ref{daes:g_mats}).  By using (\ref{daes:smat}), the first
term on the right-hand side of (\ref{Fj;xi}), when evaluated at
the steady-state, is the matrix product
\begin{align}\label{daes:term1_done}
-\frac{3}{\bar{\chi} s_0} v_{\max 0}^2 \zeta_0 \left(\sum\limits_{k=1}^N
  \partial_{x_j} G(x_j^0;x_k^0) \left(\partial_{x_i} s_{k}\right)\right)_{N\times
  N} \sim  - \frac{3}{\bar{\chi} s_0} v_{\max 0}^2 \zeta_0 \,\,
  \nabla \mathcal G \,  \nabla \boldsymbol{s} =
  -\frac{3}{\bar{\chi} a_g} v_{\max 0}^2 \zeta_0\,
  \nabla \mathcal G \,  \left(I+
  \frac{3\zeta_0}{\bar{\chi} a_g v_{\max 0}} \mathcal G\right)^{-1}
   \left(\nabla \mathcal G\right)^T\,.
\end{align}

Next, we focus on the second sum in (\ref{Fj;xi}), which is equivalent to
\begin{align}\label{daes:term2_a}
  \sum\limits_{k=1}^N v_{\max k}^3 \partial_{x_i}[\partial_{x_j}G(x_j;x_k)]=
  \left\{\begin{array}{ll}
        v_{\max i}^3 \, \partial_{x_i} \partial_{x_j} G(x_j;x_i)\,,
           \quad &i\neq j\,,\\
           \sum_{k\ne j}^{N} v_{\max k}^3 \partial^2_{x_j}G(x_j;x_k)+
           v_{\max j}^3 \, \frac{\partial}{\partial y}\vert_{y=x_j}
          \frac{\partial}{\partial x}\vert_{x=x_j} R(x;y) \,, \quad
            & i=j \,.
    \end{array}
\right.
\end{align}
From the BVP (\ref{greenequation}) satisfied by $G(x;x_k)$, we conclude that
$\partial_{x_j}^2 G(x_j;x_k) = - \frac{\bar{u} \mu}{d_1} G(x_j;x_k)$ for
$j\neq k$, so that 
\begin{align}\label{daes:term2_b}
  \sum\limits_{k=1}^N v_{\max k}^3 \partial_{x_i}[\partial_{x_j}G(x_j;x_k)]=
  \left\{\begin{array}{ll}
        v_{\max i}^3 \partial_{x_i} \partial_{x_j} G(x_j;x_i)\,,
           \quad &i\neq j\,,\\
           -\frac{\bar{u}\mu}{d_1}\sum_{k\ne j}^{N}
           v_{\max k}^3 G(x_j;x_k)+
           v_{\max j}^3 \, \frac{\partial}{\partial y}\vert_{y=x_j}
          \frac{\partial}{\partial x}\vert_{x=x_j} R(x;y)
       \,, \quad & i=j \,.
    \end{array}
\right.
\end{align}
To evaluate the last term in (\ref{daes:term2_b}), we use the chain rule
on the regular part $R(x;y)$ to get
\begin{equation}\label{daes:term2_more}
    \frac{\partial}{\partial y}\vert_{y=x_j}
    \frac{\partial}{\partial x}\vert_{x=x_j} R(x;y) =
    - R_{xx}(x_j;x_j) + \partial_{x_j}
    \left(\frac{\partial R}{\partial x}(x;x_j)\vert_{x=x_j} \right) \,.
\end{equation}
By using (\ref{computation533insec51new}) to identify the second term on
the right-hand side of (\ref{daes:term2_more}), and by calculating
$R_{xx}(x_j;x_j)$ from (\ref{appendixG:rxx}) of
Appendix \ref{appendixG}, we conclude that
\begin{equation}\label{daes:term2_nasty}
      \frac{\partial}{\partial y}\vert_{y=x_j}
    \frac{\partial}{\partial x}\vert_{x=x_j} R(x;y) =
    - \frac{\bar u\mu}{d_1} G(x_j;x_j) + \partial^2_{x_j}
     G(x_j;x_j) \,.
\end{equation}
Upon substituting (\ref{daes:term2_nasty}) into (\ref{daes:term2_b}) we
obtain
\begin{align}\label{daes:term2_c}
  \sum\limits_{k=1}^N v_{\max k}^3 \partial_{x_i}[\partial_{x_j}G(x_j;x_k)]=
  \left\{\begin{array}{ll}
        v_{\max i}^3 \partial_{x_i} \partial_{x_j} G(x_j;x_i)\,,
           \quad &i\neq j\,,\\
           -\frac{\bar{u}\mu}{d_1}\sum_{k=1}^{N}
           v_{\max k}^3 G(x_j;x_k)+  v_{\max j}^3 \,
       \partial^2_{x_j}  G(x_j;x_j) \,, \quad & i=j \,.
    \end{array}
\right.
\end{align}

Finally, we evaluate (\ref{daes:term2_c}) at the steady-state solution
where $x_j=x_j^{0}$ and $v_{\max j}=v_{\max 0}$ for $j=1,\ldots,N$, and where
we recall that $a_g=\sum_{k=1}^{N} G(x_j^{0};x_k^{0})$. Upon writing the
resulting expression in matrix form, we get
\begin{align}\label{daes:term2_final}
  \left(  \sum\limits_{k=1}^N v_{\max 0}^3 \partial_{x_i}[\partial_{x_j}G(x_j;x_k)]
  \right)_{N\times N} =-\frac{\bar{u}\mu}{d_1} v_{\max 0}^3 a_g I +
  v_{\max 0}^3 \nabla^2 \mathcal G\,,
\end{align}
where $\nabla^2 \mathcal G$ was defined in (\ref{daes:g_mats}). By
substituting (\ref{daes:term2_final}) and (\ref{daes:term1_done}) in
(\ref{daes:fjxi}), we calculate the Jacobian as
\begin{align}\label{daes:jac_fin}
  {\mathcal J} :=
  \bigg(\frac{\partial \mathcal F_i}{\partial x_j}\bigg)_{N\times N} =
  - \frac{3 v_{\max 0}^2\zeta_0}{\bar{\chi} a_g} \nabla \mathcal G
  \left( I + \frac{3\zeta_0}{\bar{\chi} a_g v_{\max 0}} {\mathcal G}
  \right)^{-1} \left( \nabla \mathcal G\right)^T +
  v_{\max 0}^3 \nabla^2 \mathcal G - \frac{{\bar u}\mu a_g}{d_1} v_{\max 0}^3
  I \,,
\end{align}
where $\nabla\mathcal G$, $(\nabla \mathcal G)^T$ and
$\nabla^2\mathcal G$ were given in (\ref{daes:g_mats}). Then, by
evaluating the matrix $\tilde{\mathcal M}$ in (\ref{dae:linearize})
that arises in the linearization of the DAE system around the
steady-state spike locations, we get
\begin{align}\label{daes:tilde_mat}
  \tilde{\mathcal M}=
  \frac{2 v_{\max 0}^2\zeta_0}{a_g} \nabla \mathcal G
  \left( I + \frac{3\zeta_0}{\bar{\chi} a_g v_{\max 0}} {\mathcal G}
  \right)^{-1} \left( \nabla \mathcal G\right)^T -
  \frac{2\bar\chi}{3} v_{\max 0}^3 \nabla^2 \mathcal G +
  \frac{2{\bar u}\mu \bar{\chi} a_g}{3d_1} v_{\max 0}^3
  I \,.
\end{align}
Finally, upon using $s_0={2\bar{\chi} a_g v_{\max 0}^3\epsilon/3}$ to
simplify the coefficient of the identity matrix in
(\ref{daes:tilde_mat}), and by using the key relations
\begin{equation}\label{daes:mat-conv}
  \nabla {\mathcal G} = {\mathcal P} \,, \qquad
  \left( \nabla {\mathcal G} \right)^T= -{\mathcal P}_g \,, \qquad
  \nabla^2 {\mathcal G} = -{\mathcal G}_g \,, \qquad
\end{equation}
as derived in Appendix \ref{appendixG}, we conclude upon comparing
(\ref{daes:tilde_mat}) and (\ref{jacobianbutinsmallep}) that
$\tilde{\mathcal M}=\mathcal M$.

In summary, our analysis establishes that the small eigenvalues
associated with the linearization of the steady-state solution are
precisely the same eigenvalues that are associated with linearizing
the DAE system of slow spike dynamics about the steady-state spike
locations.

\section{Discussion}\label{sec:discussion}

In this concluding section, we first discuss how our analysis of 1D
spike patterns in the KS model (\ref{timedependent}) in the limit
$d_2\ll 1$ shares some common features with a related analysis of
localized spike patterns for the GM model
(cf.~\cite{iron2001stability}, \cite{iw2}, \cite{wwhopf},
\cite{wei2007existence}). We also mention a few open problems that
warrant further investigation.

\subsection{Comparison with the GM System}\label{secextra}

We first make some remarks on an interesting connection between the
analysis of spike patterns for the KS model (\ref{timedependent}) in
the limit $d_2=\epsilon^2\ll 1$ and that for the GM model
\begin{equation}\label{GM}
  A_t = d_a A_{xx} - A + {A^{p}/H^q} \,, \qquad
  \tau H_t = D H_{xx} - \mu H + {A^r/H^s} \,,
\end{equation}
in the limit $d_a\ll 1$ of small activator diffusivity. In this
context, $A$ and $H$ are the activator and inhibitor fields,
respectively. Moreover, $\tau>0$, $D>0$, and $\mu>0$ are constants and
the GM exponents $(p,q,r,s)$ satisfy the usual conditions $p>1$,
$q>0$, $r>0$, $s\geq 0$, and ${(p-1)/q}<{r/(s+1)}$.

In \cite{iron2001stability}, steady-state 1D spike patterns in which
$A$ is spatially localized with spike-width ${\mathcal O}(\sqrt{d_a})$
were constructed for (\ref{GM}) in the limit $d_a\ll 1$ using the
method of matched asymptotic expansions. For this class of solutions,
the spike profile is characterized by the homoclinic solution of
$w_{yy}-w+w^p=0$, where $y=d_{a}^{-1/2}(x-x_j)$. In the outer region
between spikes where $A\approx 0$, the interaction between
steady-state spikes is mediated by the inhibitor diffusion field with
the term $A^{r}/H^s$ being approximated by Dirac masses concentrated
at the spike locations.  As a result, when $d_a\ll 1$, the activator
$A$ behaves like a linear combination of discrete spikes on the
domain, while $H$ is well-approximated by a superposition of
translates of the reduced-wave Green's function.

In comparison, we observe from our steady-state analysis of the KS
model (\ref{timedependent}) in the limit $d_2\ll 1$ given in \S
\ref{sec2} that the chemoattractant $v$ and the cellular population
density $u$ share a similar asymptotic structure to $A$ and $H$,
respectively, in (\ref{GM}). In our analysis, the spike profile for
$v$ is represented by a homoclinic solution (\ref{Hamiltonian}), while
the outer solution for $u$ is well-approximated by a superposition of
translates of the Helmholtz Green's functions
(\ref{woouter}). Moreover, in the limit $d_2\ll 1$, the background
constant $s_0\ll 1$ in (\ref{ag}) plays the same role as the locally
constant inhibitor field in the core of a spike for (\ref{GM}). 

With regards to the NLEP linear stability analysis, the approximating
NLEP (\ref{NLEPbarzvar}) that arises from our sub-inner layer analysis
in \S \ref{sec3} is rather similar in form to the NLEP for the GM
model that occurs for the exponents $p=3$ and $s=0$.  This connection
results from the explicit form given in (\ref{gm_anal}) for the
sub-inner solution. As a result, by adapting the NLEP linear stability
analysis given in \cite{iron2001stability}, \cite{wwhopf}, and
\cite{wei2007existence}, we are able to calculate parameter thresholds
corresponding to either zero-eigenvalue crossings as $d_1$ is varied,
or Hopf bifurcations as $\tau$ is increased.  In particular, although
in minimal KS models, without the logistic term, spike amplitude
temporal oscillations are not expected, our NLEP linear stability
analysis in the presence of the logistic source term has shown that a
sufficiently large reaction time $\tau>0$ in (\ref{lep}) can trigger
such spike oscillations for a one-spike steady-state. The mechanism
for these oscillations, being a sufficiently large diffusive
time-delay between the two components in (\ref{timedependent}), is
qualitatively the same as that studied in \cite{wwhopf} for the GM
model.

With regards to the analysis of the small eigenvalues, which
characterize possible translational instabilities of the spike
locations, the reduced multi-point BVP derived in Proposition
\ref{prop41lambda} is very similar in form to that derived for the GM
model in \S 4 of \cite{iron2001stability}. As a result, the detailed
framework for the GM matrix analysis in \S 4 of
\cite{iron2001stability} was employed for obtaining Proposition
\ref{prop:small} for the small eigenvalues, which lead to the explicit
result in Lemma \ref{small:explicit}.  As qualitatively similar to
that for the GM model (cf.~\cite{iron2001stability}), we showed for
our $N$-spike steady-sate solution that there are $N-1$ simultaneous
zero-crossings for the small eigenvalues that occur at the same
critical value of $d_1$. For the GM model, these simultaneous
crossings occur at a common value of $D$, and this threshold
provides the critical value of $D$ for which branches of
asymmetric spike equilibria, corresponding to spikes of different
height, bifurcate from the symmetric steady-state branch
(cf.~\cite{asymm}).  Finally, the slow DAE dynamics of spike
quasi-equilibria, as characterized by Proposition
\ref{propositon51insec5} in terms of gradients of the Helmholtz
Green's function, is rather similar to that derived for the GM model
in \cite{iw2}.

One novel feature of our analysis has been to use distinctly different
approaches to both calculate and verify linear stability thresholds
resulting from our detailed asymptotic analysis. In particular, in \S
\ref{sec:jac_non}, the non-invertibility of the Jacobian matrix that
resulted from the steady-state analysis for fixed spike locations
closely approximates the NLEP linear stability threshold when
$\tau=0$.  Moreover, the linearization of the steady-state of the DAE
slow spike dynamics (\ref{DAE1}) was found in \S \ref{sec:balance_dae}
to correspond identically to our asymptotic result in Proposition
\ref{propositon51insec5} for the small eigenvalues. Finally, our
zero-eigenvalue crossing condition for the small eigenvalues in
(\ref{hj:zero}) was shown in Appendix \ref{app:asymmetric} to
correspond to the bifurcation point where asymmetric equilibria
emerge from the symmetric steady-state solution branch.

Next, we discuss some key differences between our analysis of spike
patterns for the KS model (\ref{timedependent}) and that for the GM
model in \cite{iron2001stability}. Firstly, owing to the different
Green's functions mediating the spike interactions for
(\ref{timedependent}) and for (\ref{GM}), for the GM model there is no
analogue of the positivity and resonance conditions of
(\ref{d1:admiss}) discussed in Remark \ref{remark:pos}.  Secondly, the
competition and translational stability thresholds for symmetric spike
equilibria for the GM model (\ref{GM}) are given by  explicit
critical values for the inhibitor diffusivity $D$. In our analysis of
the KS model, these two thresholds are characterized by weakly
nonlinear algebraic equations in the cellular diffusivity $d_1$. This
distinction arises, in part, to the existence of an intricate
sub-inner structure of the spike profile for the KS model
(\ref{timedependent}) that has no counterpart in the GM model
(\ref{GM}).  Finally, for the KS model, the numerical results shown in
Figure \ref{figurenucleation} suggest that spike nucleation behavior
can occur from the midpoint of the background state between
neighboring spikes as $d_1$ is decreased below the positivity
threshold $d_{1pN}$ in (\ref{d1:admiss}). Such nucleation behavior
does not occur for the GM model (\ref{GM}).

\subsection{Further Directions}\label{sec:further}

In the limit of small diffusivity $d_2=\epsilon^2\ll 1$ for the
chemotactic concentration field, we have developed a hybrid
asymptotic-numerical approach to analyze the existence, linear
stability, and slow dynamics of 1D spike patterns for
(\ref{timedependent}). The study of pattern forming properties for
(\ref{timedependent}) when $d_2\ll 1$ is distinctly different than
that based on the usual approach of considering the large chemotactic
drift limit, i.e.~$\chi\gg 1$ in (\ref{timedependent}), as was done in
most previous analyses and numerical simulations of localized patterns
(cf.~\cite{KWX2022}, \cite{KWX20221}, \cite{jin2016pattern},
\cite{wang2016qualitative}). In the limit $d_2\ll 1$, we have shown
that the analysis of localized 1D spike patterns is rather closely
related to that for the GM model.

We now discuss a few open problems related to our study. From a
mathematical viewpoint, the analytical tractability of our quasi
steady-state and linear stability analysis has relied to a large
extent on the availability of certain explicit formulae for the spike
profile that exists in the sub-inner region of a spike. More
specifically, our explicit but approximate analysis is based on the
asymptotically large spike height $v_{\max}\gg 1$ limit. However,
since $v_{\max}={\mathcal O}(-\log\epsilon)$ is only rather large when
$\epsilon$ is extremely small, our asymptotic results for
steady-states and for the linear stability thresholds provide only a
moderately decent prediction of corresponding full numerical results when
$\epsilon=0.01$ is only fairly small. One theoretical open
challenge is to provide a rigorous steady-state and linear stability
analysis for multi-spike patterns that is based on the full inner problem
(\ref{Hamiltonian}) and the corresponding NLEP (\ref{NLEP}), which does
not exploit the large $v_{\max}$ limit.

One important open problem from the viewpoint of global bifurcation
theory is to numerically compute solution branches of localized 1D
steady-state spike patterns for (\ref{timedependent}) as $d_1$, $d_2$
and $\chi$ are varied. This would clarify how solution branches of
spike equilibria differ when either $d_2\ll 1$ or when $\chi\gg
1$. The observation of spike nucleation behavior as shown in
\cite{painter2011spatio,hillen2013merging} for certain parameter sets,
and hinted at in Figure \ref{figurenucleation} as $d_1$ is decreased
below the positivity threshold in (\ref{d1:admiss}) of Remark
\ref{remark:pos} should be investigated. For chemotaxis models of
urban crime, spike nucleation events for the emergence of hotspots
have been shown to occur near saddle-node bifurcation points of
branches of spike equilibria (cf.~\cite{tsecrime}). In contrast, for
(\ref{timedependent}) when $d_1={\mathcal O}(1)$,
$d_2={\mathcal O}(1)$ and $\chi={\mathcal O}(1)$ they appear to arise
from Turing bifurcations of the base state
(cf.~\cite{painter2011spatio,hillen2013merging}). Two other possible
extensions of our 1D analysis are to analyze the existence and linear
stability of asymmetric spike equilibria for (\ref{timedependent}) and
to analyze steady-state patterns for variants of (\ref{timedependent})
that incorporate other cellular population growth models and possible
nonlinear mechanisms that couple the cellular density to the
chemoattractant concentration.

It would also be worthwhile to extend our 1D analysis to analyze the
existence, linear stability, and slow dynamics of localized patterns
for (\ref{timedependent1}) when $d_2\ll 1$ in a 2D bounded domain.
One such direction would be to analyze the linear stability properties
of a localized stripe in a 2D rectangular domain that results from a
trivial extension of the 1D spike in the transverse
direction. Numerical results in \cite{painter2011spatio} suggest that,
in marked contrast to the well-known instability behavior of
homoclinic stripes for the GM model (cf.~\cite{gmstripe}), a localized
stripe for a coupled chemotaxis system may be linearly stable to
breakup into spots. As a result, it would be interesting to
theoretically investigate the possibility of varicose or transverse
instabilities of such localized stripes.  A second interesting
direction is motivated by the numerical simulations reported in
\cite{jin2016pattern} that suggest that localized 2D spot patterns for
(\ref{timedependent1}) should exist in the singular limit $d_2\ll
1$. Given the rather close correspondence between the analysis of
localized patterns for (\ref{timedependent1}) in the limit $d_2\ll 1$ 
and the GM model in the limit of a small activator diffusivity, the
framework for a 2D steady-state and linear stability analysis of
(\ref{timedependent1}) for spot patterns would likely rely somewhat on
the approach developed for the 2D GM model, as summarized in
\cite{WeiWinterBook}.

\section*{Acknowledgments}

We thank Professor T. Kolokolnikov for useful discussions and many
critical suggestions.  The research of M.J. Ward and J. Wei is
partially supported by NSERC of Canada.

\begin{appendix}
\renewcommand{\theequation}{\Alph{section}.\arabic{equation}}
\setcounter{equation}{0}

\section{Solvability of the Outer Problem: Turing Instability of the Base State}\label{app:Turing}

In this appendix, we relate the solvability of the outer problem
(\ref{outerproblem}) to Turing bifurcation points in the 
parameter $d_1$ for the spatially uniform base state $u=v=0$ of
(\ref{timedependent}). This analysis will motivate Remark
\ref{remark:pos}.

On an interval of length $L$, with homogeneous Neumann conditions for
$u$ and $v$, we linearize (\ref{timedependent}) around $u=v=0$ by
setting $u=e^{\lambda t + ikx}\Phi$ and
$v=e^{\lambda t + ikx}N$, where $k={m \pi/L}$ with
$m=1,2,\ldots$. We readily obtain that
\begin{align}\label{app:turing_matrix}
 \left( \begin{array}{cc}
          -d_1 k^2 + \mu \bar{u} - \lambda & 0\\
           1  &  -\epsilon^2 k^2 -1 - \lambda\\
    \end{array}
  \right)  \left( \begin{array}{c}
          \Phi\\
          N \\
    \end{array}
  \right) = {\mathbf 0} \,,
\end{align}
which has a nontrivial solution if and only if $\lambda=-1-\epsilon^2 k^2$
or $\lambda=-d_1k^2+\mu \bar{u}$. As such, with $k={m \pi/L}$, there
is a zero-eigenvalue crossing associated with the spatially uniform state
$u=v=0$ at the critical values
\begin{equation}\label{app:tur_d1m}
    d_1 = \frac{\mu \bar{u} L^2}{m^2 \pi^2} \,, \quad m=1,2,\ldots\,.
\end{equation}
This base-state is linearly stable on a domain of length $L$ when
$d_1>{\mu \bar{u} L^2/\pi^2}$. Setting $L=2$, consistent with
(\ref{timedependent}), we conclude that (\ref{app:tur_d1m}) coincides
precisely with the ``resonant'' values of $d_1$ in
(\ref{prop:w0h}) for the outer problem.

However, in our construction of $N$-spike steady-state patterns for
(\ref{timedependent}), the spatially uniform base state approximates
the outer solution $w_o$ only on intervals of length ${2/N}$. Upon
setting $L={2/N}$ in (\ref{app:tur_d1m}), this observation suggests
that the outer solution for an $N$-spike steady-state should be
linearly stable when $d_1> {4\mu \bar{u} /(N^2 \pi^2)}$. This latter
threshold also has the alternative interpretation that it is the
smallest value of $d_1$ for which the outer solution $w_o$ is always
positive in $|x|<1$. In particular, for an $N$-spike steady-state, it
is easy to verify that this positivity condition for $w_o$ holds when
$d_1> d_{1pN}:= {\bar u \, \mu /\lambda_1}$, where
$\lambda_1:={N^2\pi^2/4}$ is the first non-zero eigenvalue of the
negative Neumann Laplacian -${d^2/dx^2}$ on $(-1/N,1/N)$. We remark
that for quasi-equilibrium patterns with unequally spaced spikes, this
positivity threshold must be modified to (\ref{d1:qe}).

Next, we verify that the outer problem (\ref{outerproblem}) is
solvable for an $N$-spike steady-state pattern when $d_1=d_{1Tm}$,
where $d_{1Tm}$ is one of the ``resonant'' values in (\ref{prop:w0h}) with
$m=1,\ldots,N-1$. For the steady-state problem, where
$v_{\max k}=v_{\max 0}$ and where $x_k=x_{k}^{0}$, with $x_{k}^{0}$ as
given in (\ref{locations}), (\ref{outerproblem}) is solvable at
$d_1=d_{1Tm}$ if and only if
\begin{equation}\label{app:trig_sum}
  \int_{-1}^{1} w_{oh}  {\mathcal L}_0 w_o \, dx =
  \frac{2\bar\chi \epsilon}{3} v_{\max 0}^3 \sum_{k=1}^{N}
  \int_{-1}^{1}  w_{oh}(x)\, \delta(x-x_k^{0}) \, dx =
  \frac{2\bar\chi \epsilon}{3} v_{\max 0}^3 \sum_{k=1}^{N}
  \cos\left(\frac{m\pi}{2}\frac{(2k-1)}{N}\right) =0 \,.
\end{equation}
The trigonometric sum in (\ref{app:trig_sum}) can be readily evaluated
for $m=1,\ldots,N-1$ with the result
\begin{equation}
\sum_{k=1}^{N}
\cos\left(\frac{m\pi}{2}\frac{(2k-1)}{N}\right) =
\frac{\sin(m\pi)}{2\sin\left({m\pi/N}\right)}=0 \,.
\end{equation}
As a consequence, (\ref{outerproblem}) is solvable for an $N$-spike
steady-state even when $d_1={\mathcal T}_e$ (see (\ref{d1:admiss})).

Finally, we remark that when $d_1={4\mu \bar{u}/(m^2 \pi^2)}$, for
some $m=1,\ldots,N-1$, a solution (non-unique) to (\ref{outerproblem})
for an $N$-spike steady-state can be represented as
$u_o\sim w_o=\frac{2}{3}\bar\chi\epsilon v_{\max 0}^3\sum_{k=1}^N
G_{m}(x;x_k^{0})$. Here, with the operator ${\mathcal L}_0$ of
(\ref{outerproblem}), the modified Green's function $G_{m}(x;\xi)$
satisfies
\begin{equation}\label{app:gmod}
  {\mathcal L}_0 G_m = \delta(x-\xi) - w_{oh}(\xi) w_{oh}(x) \,, \quad
  |x|\leq 1 \,; \qquad G_{mx}(\pm 1;\xi)=0\,.
\end{equation}
Although $G_m$ can be found analytically, for simplicity we have
restricted our analysis only to when $d_1\in {\mathcal T}_e$.

\section{Calculation of $\mathcal G_{\lambda}$ and $\mathcal P$}\label{appendixA}

In this appendix, we show how to determine the matrix spectrum of
${\mathcal G}_{\lambda}$, as defined in (\ref{nlep:glambda}) of \S
\ref{sec3}. Moreover, we calculate ${\mathcal P}$, as defined in
(\ref{matrixGg}) of \S \ref{sec4}. To do so, we introduce an
auxiliary problem for $y=y(x)$, given by
\begin{align}\label{bvpy}
  \frac{d_1}{\mu}y^{\prime\prime}+ {\hat u} y=0\,, \quad -1<x<1\,; \quad
  y^{\prime}(\pm 1)=0\,; \quad [y]_j=0\,, \quad \Big[\frac{d_1}{\mu}y^{\prime}
\Big]_j=b_j\,,
\end{align}
for $j=1,\ldots,N$, where $[y]_j:=y(x_j^+)-y(x_j^-)$ and $x_j=x_j^{0}$
is given by (\ref{locations}). Here
$\hat{u}:=\bar{u}-{\tau\lambda_0/\mu}$. This problem is solvable when
$d_1\neq {4\mu \hat{u}/(m^2\pi^2)}$ for $m=1,2,\ldots$. When $\tau=0$,
(\ref{bvpy}) is always solvable when $d_1\in {\mathcal T}_e$.

With the exception of this restricted set for $d_1$, the solution to
(\ref{bvpy}) can be represented in terms of the Green's function
$G_{\lambda}(x;x_k)$, satisfying (\ref{lam:greenequation}), as
$y=\sum_{k=1}^N b_k G_{\lambda}(x;x_k)$. Upon defining
$\boldsymbol{y}:=\left(y_{1},\ldots,y_N\right)^T$,
$\boldsymbol{\langle y^{\prime}\rangle}:=\left(y^{\prime}_{1},\ldots,
  y^{\prime}_{N}\right)^T$ and
$\boldsymbol{b}:=\left(b_1,\ldots,b_N\right)^T$, where $y_j=y(x_j)$
and
$\langle y^{\prime}\rangle_j=
\big(y^{\prime}(x_j^{+})+y^{\prime}(x_j^{-})\big)/2$, we identify the
eigenvalue-dependent Green's matrix ${\mathcal G}_{\lambda}$ of
(\ref{nlep:glambda}) and ${\mathcal P}$ of (\ref{matrixGg}) as
\begin{align}\label{yprimeandomega}
  \boldsymbol{y}=\mathcal G_{\lambda}\boldsymbol{b}\,, \qquad
 \boldsymbol{ \langle y^{\prime}\rangle} =\mathcal P\boldsymbol{b}\,.
\end{align}

Next, we show how to represent ${\mathcal G}_{\lambda}$ and
${\mathcal P}$ in terms of tridiagonal matrices. By solving
(\ref{bvpy}) on each subinterval, and enforcing the continuity
conditions $[y]_j=0$ for $j=1,\ldots,N$, we get
\begin{align}\label{analyticy}
y=\left\{\begin{array}{ll}
y_1\frac{\cos[\theta_\lambda(1+x)]}{\cos[\theta_\lambda(1+x_1)]}\,,&-1<x<x_1\,,\\
  y_j\frac{\sin[\theta_\lambda(x_{j+1}-x)]}{\sin[\theta_\lambda(x_{j+1}-x_j)]}+
      y_{j+1}\frac{\sin[\theta_\lambda(x-x_j)]}{\sin[\theta_\lambda(x_{j+1}-x_j)]}\,,
                &x_j<x<x_{j+1}\,,\quad j=1,\ldots,N-1\,,\\
y_{N}\frac{\cos[\theta_\lambda(1-x)]}{\cos[\theta_\lambda(1-x_{N})]}\,,&x_{N}<x<1\,.
\end{array}
\right.
\end{align}
Then, upon satisfying the jump conditions in (\ref{bvpy}) we can write
$\boldsymbol{b}$ as 
\begin{equation}\label{theta_lambda}
  \boldsymbol{b} = \frac{d_1\theta_{\lambda}}{\mu} {\mathcal D} \boldsymbol{y} \,,
  \qquad \mbox{where} \quad  \theta_{\lambda}:=
  \sqrt{\frac{\mu }{d_1}\left(\bar{u}-\frac{\tau\lambda_0}
      {\mu}\right)} \,.
\end{equation}
Here, for $d_1\neq {4\mu \hat{u}/(m^2\pi^2)}$ with $m=1,2,\ldots$,
${\mathcal D}$ is the invertible tridiagonal matrix defined by
\begin{align}\label{inverseG}
\mathcal D=\left( \begin{array}{ccccccc}
  d & f&0&\cdots &0 &0& 0\\
  f & e&f&\cdots &0 &0& 0\\
   0 & f&e&\ddots &0 &0& 0\\
     \vdots & \vdots&\ddots&\ddots &\ddots &\vdots& \vdots\\
    0 & 0&0&\ddots &e &f& 0\\
  0 & 0&0&\cdots &f &e& f\\
   0 & 0&0&\cdots &0 &f& d\\
    \end{array}
  \right)\,.
  \end{align}
The matrix entries of ${\mathcal D}$, for which the identity $d=f+e$ holds,
are
\begin{align}\label{def}
    d\equiv\tan(\theta_\lambda/N)-\cot(2\theta_\lambda/N)\,,\quad
    e\equiv-2\cot(2\theta_\lambda/N)\,,\quad f\equiv\csc(2\theta_\lambda/N)\,.
\end{align}
By combining (\ref{theta_lambda}) with the first equation
in (\ref{yprimeandomega}), we conclude for
$d_1\neq {4\mu \hat{u}/(m^2\pi^2)}$ for $m=1,2,\ldots$ that
\begin{equation}\label{key:G_to_D}
  {\mathcal G}_{\lambda} = \sqrt{\frac{\mu}{d_1\hat{\mu}}} {\mathcal D}^{-1}\,,
  \quad \mbox{with} \quad \hat{u}=\bar{u} - \frac{\tau\lambda_0}{\mu} \,.
\end{equation}
When $\tau=0$, we remark that (\ref{key:G_to_D}) holds when
$d_1\in {\mathcal T}_e$.

Since ${\mathcal D}$ is a tridiagonal matrix with a constant row sum, its
eigenpairs $\kappa_j$ and ${\bf q}_j$ for $j=1,\ldots,N$ can be calculated
explicitly (see \cite{iron2001stability}), with the following result:

\begin{proposition}\label{prop31}
The eigenvalues $\kappa_j$ and the normalized eigenvectors of $\mathcal D$ are
\begin{align*}
    \left\{\begin{array}{ll}
             \kappa_1=e+2f\,;\quad \kappa_j=e+2f\cos\left(\pi(j-1)/N\right)\,,
             &j=2,\ldots,N,\\
 \boldsymbol{q}_1=\frac{1}{\sqrt{N}}(1,\ldots,1)^T\,,\quad
 q_{l,j}=\sqrt{\frac{2}{N}}\cos\left(\frac{\pi(j-1)}{N}(l-\frac{1}{2})\right)\,,
           \quad &j=2,\ldots,N\,,\,\, l=1,\ldots,N,
    \end{array}
    \right.
\end{align*}
where $\boldsymbol{q_j}=(q_{1,j},\ldots,q_{N,j})^T$ and $d$, $e$ and $f$ are
given by (\ref{def}). By using (\ref{key:G_to_D}), the eigenvalues $\sigma_j$
of ${\mathcal G}_{\lambda}$ when $d_1\neq {4\mu \hat{u}/(m^2\pi^2)}$ for
  $m=1,2,\ldots$, are
\begin{equation}\label{app:sigma_j}
  \sigma_j = \sqrt{\frac{\mu}{d_1\hat{u}}} \left[e + 2f
    \cos\left( \frac{\pi(j-1)}{N}\right) \right]^{-1}, \qquad j=1\,,\ldots,N
  \,.
\end{equation}
By setting $\lambda_0=0$, we use (\ref{key:G_to_D}) and
Proposition \ref{prop31} to calculate $a_g$, as defined in
(\ref{ag}). For $d_1\in {\mathcal T}_e$, we get
\begin{equation}\label{realag}
  a_g = \sum_{k=1}^N G\big(x_j^0;x_k^0\big)= \sqrt{\frac{\mu}{d_1\bar{u}}}
  {\mathcal D}^{-1} \sqrt{N} {\boldsymbol q}_1 =
  \sqrt{\frac{\mu}{d_1\bar {u}}}\frac{1}{(e+2f)}=
  \frac{1}{2} \sqrt{\frac{\mu}{d_1\bar{u}}}
  \cot\left(\frac{\theta}{N}\right) \,, \quad
  \theta= \sqrt{\frac{\mu\bar{u}}{d_1}}\,.
\end{equation}
\end{proposition}

To determine ${\mathcal P}$ when $\lambda_0=0$, we use
(\ref{analyticy}) to write $\boldsymbol{\langle y^{\prime}\rangle}$ in
terms of $\boldsymbol{y}$ as
$\boldsymbol{\langle y^{\prime}\rangle }=-\left({\theta/2}\right)
  \csc\Big({2\theta/N}\Big)\mathcal C^T\boldsymbol{y}$,
  where $\theta=\sqrt{\mu\bar{u}/d_1}$ and ${\mathcal C}$ is the tridiagonal
  matrix defined by
\begin{align}\label{app:matc}
\mathcal C:=\left( \begin{array}{ccccccc}
  1 & 1&0&\cdots &0 &0& 0\\
  -1 & 0&1&\cdots &0 &0& 0\\
   0 & -1&0&\ddots &0 &0& 0\\
     \vdots & \vdots&\ddots&\ddots &\ddots &\vdots& \vdots\\
    0 & 0&0&\ddots &0 &1& 0\\
  0 & 0&0&\cdots &-1 &0& 1\\
   0 & 0&0&\cdots &0 &-1& -1\\
    \end{array}
  \right).
  \end{align}
  By combining the second equation in (\ref{yprimeandomega}) with this
  result, we conclude for $d_1\in {\mathcal T}_e$ that
  \begin{align}\label{matrixP}
    \mathcal P=-\frac{\mu}{2d_1}\csc\left(\frac{2\theta}{N}\right)\mathcal
    C^T\mathcal D^{-1}\,.
  \end{align}

  \section{Proof of Theorem \ref{theorem1}}\label{appensec3}

For convenience, we drop the overbars in \eqref{NLEPbarzvar} to rewrite the
NLEP as
\begin{align}\label{NLEPbarzvartilde}
 \Psi_{0 z z}+ U_0\Psi_0-\alpha U_0\frac{\int_{-\infty}^\infty { U_0}^2\Psi_0 \, dz }
  {\int_{-\infty}^\infty { U_0}^2 \, dz}=\Lambda \Psi_0\,, \quad
  -\infty<z<+\infty\,; \qquad \Psi_0 \,\,\, \mbox{bounded as} \,\,\,
  |z|\rightarrow \infty \,.
\end{align}
Here $U_0=2\sech^2 z$, $\Lambda:=\delta^2(\lambda_0+1)$ with
$\delta:={2/(\bar\chi v_{\max 0})}$.  It is well-known
\cite{kolokolnikov2009existence} that the homoclinic solution to
$w_{zz}- w+{w}^3=0$ on $-\infty<z<\infty$ with $w(0)>0$,
$w^{\prime}(0)=0$ and $w\rightarrow 0$ as $|z|\to\infty$ is
$w=\sqrt{2}\sech(y)$.  Therefore, we have $U_0=w^2$ and the NLEP
(\ref{NLEPbarzvartilde}) becomes
\begin{align}\label{NLEPbarzvartilde1}
\Psi_{0zz}+  w^2\Psi_0- \alpha w^2\frac{\int_{-\infty}^\infty { w}^4\Psi_0 \, dz }
  {\int_{-\infty}^\infty { w}^4 \, dz}=\Lambda \Psi_0\,, \quad -\infty<z<+\infty
  \,; \qquad \Psi_0 \,\,\, \mbox{bounded as}\,\,\,  |z|\rightarrow \infty \,.
\end{align}

There is a standard approach \cite{Wei1999single} to study
(\ref{NLEPbarzvartilde1}).  Firstly, we focus on the following
local eigenvalue problem:
\begin{align}\label{local}
  \Psi_{0zz}+  w^2\Psi_0=\lambda\Psi_0\,, \quad -\infty<z<\infty\,;
  \qquad \Psi_0 \,\,\, \mbox{bounded as} \,\,\, |z|\to \infty\,.
\end{align}
As shown in \cite{kolokolnikov2009existence}, the principal eigenvalue
of (\ref{local}) is $\lambda=1$ and the corresponding eigenfunction is
$\Psi_0=w$.  Next, we transform (\ref{NLEPbarzvartilde1}) into a form
more amenable for analysis. To this end, we observe from the ODE
$w^{\prime\prime}-w+w^3=0$ that $w^2$ satisfies
\begin{align}\label{groundstatewsquare}
  (w^2)_{zz}- 4w^2+3w^4=0\,, \quad -\infty< z<+\infty\,; \qquad
  w\rightarrow 0 \quad \mbox{as} \quad |z|\to \infty\,.
\end{align}
Therefore, upon multiplying the $\Psi_0$-equation in
(\ref{NLEPbarzvartilde1}) by $w^2$ and integrating it over
$(-\infty,\infty)$ by parts, we get
\begin{align}\label{appendixc1}
  \int_{-\infty}^\infty (w^2)_{zz}\Psi_0\, dz+\int_{-\infty}^\infty w^4\Psi_0\, dz-
  \alpha \int_{-\infty}^\infty w^4\Psi_0\, dz=\Lambda\int_{-\infty}^\infty w^2\Psi_0\,
  dz\,.
\end{align}
Next, upon substituting (\ref{groundstatewsquare}) into
(\ref{appendixc1}), we obtain
\begin{align}\label{relationappendix3}
  (4-\Lambda)\int_{-\infty}^\infty w^2\Psi_0\, dz=(2+\alpha)\int_{-\infty}^\infty
  w^4\Psi_0\, dz\,.
\end{align}
Then, by using (\ref{relationappendix3}), we transform the NLEP in
(\ref{NLEPbarzvartilde1}) into the following form, as written in
(\ref{secthopf-nlepequiv}):
\begin{align}\label{nlepequiv}
  \Psi_{0zz}+w^2\Psi_0- \kappa \frac{\int_{-\infty}^\infty w^2\Psi_0\, dz}
  {\int_{-\infty}^\infty w^4 \, dz}\, w^2=\Lambda \Psi_0\,, \qquad
  \kappa :=\frac{\alpha(4-\Lambda)}{(2+\alpha)}\,,
\end{align}

Next, we test (\ref{nlepequiv}) against the conjugate $\Psi_0^*$ and
by integrating the resulting expression by parts we get
\begin{align}\label{appendixsec3eq1}
  \int_{-\infty}^\infty \vert \Psi_{0z}\vert^2\, dz-\int_{-\infty}^\infty
  w^2\vert\Psi_0\vert^2\, dz+\Lambda\int_{-\infty}^\infty \vert\Psi_0\vert^2\, dz
  =-\frac{\alpha (4-\Lambda)\big\vert\int_{-\infty}^\infty w^2\Psi_0\, dz
  \big\vert^2}{(2+\alpha)\int_{-\infty}^\infty w^4\, dz}\,.
\end{align}
We first claim that $\Lambda$ is real-valued when $\alpha$ is
real-valued.  To show this, the imaginary part of
(\ref{appendixsec3eq1}) yields
\begin{align}\label{appensec3eq2}
  \mbox{Im}(\Lambda)\int_{-\infty}^\infty \vert\Psi_0\vert^2\, dz
  =\frac{\alpha\, \mbox{Im}(\Lambda)}{2+\alpha}\frac{\big\vert\int_{-\infty}^\infty
  w^2\Psi_0\, dz\big\vert^2}{\int_{-\infty}^\infty w^4\, dz}\,.
\end{align}
Then, upon invoking the Cauchy-Schwartz inequality, we obtain
\begin{align}\label{appensec3eq3}
  \frac{\alpha}{2+\alpha}\frac{\big\vert\int_{-\infty}^\infty w^2\Psi_0\, dz
  \big\vert^2}{\int_{-\infty}^\infty \, w^4dz}\leq
  \frac{\alpha}{2+\alpha}{\int_{-\infty}^\infty\vert \Psi_0\vert^2\, dz}\,.
\end{align}
Upon substituting this inequality into (\ref{appensec3eq2}), we
conclude that $\mbox{Im}(\Lambda)=0$. This completes the proof of our
claim.  It immediately follows that (\ref{appendixsec3eq1}) is also
real-valued when $\alpha$ is real-valued.

The next step is to study the sign of $\Lambda$ in (\ref{nlepequiv}).
We claim that
\begin{align*}
  \int_{-\infty}^\infty \vert\Psi_{0z}\vert^2\, dz-
  \int_{-\infty}^\infty w^2\vert\Psi_0\vert^2\, dz\geq  -
  \frac{\big\vert\int_{-\infty}^\infty w^2\Psi_0\, dz\big\vert^2 }
  {\int_{-\infty}^\infty w^2\, dz}\,.
\end{align*}
Similarly as the proof of Lemma 5 in \cite{kolokolnikov2009existence},
this claim is established if we can equivalently show that the
real eigenvalues $\upsilon$ of the following NLEP are non-positive:
\begin{align}\label{appendix-E-eq1}
  \Delta\Psi_0+w^2\Psi_0-w^2\frac{\int_{-\infty}^\infty w^2\Psi_0 dz}
  {\int_{-\infty}^\infty w^2 dz}=\upsilon \Psi_0 \,.
\end{align}
We first observe that if $\Psi_0\equiv 1$, then $\upsilon=0.$ Next, we
observe that (\ref{appendix-E-eq1}) is equivalent to solving
\begin{align}
  (L_{0}-\upsilon)\Psi_0=w^2\,, \qquad \int_{-\infty}^\infty w^2\Psi_0\, dz=
  \int_{-\infty}^\infty w^2\, dz=4 \,.
\end{align}
As such, we define $\Xi$ as
\begin{align*}
\Xi(\upsilon):= \int_{-\infty}^\infty w^2(L_{0}-\upsilon)^{-1}w^2 \, dz-4 \,.
\end{align*}
Since the operator is self-adjoint and $L_0 (1)=w^2$,  we obtain that
$\Xi(0)=0.$  By differentiating in $\Xi$ we get
\begin{align*}
  \Xi^{\prime}(\upsilon)=\int_{-\infty}^\infty w^2(L_{0}-\upsilon)^{-2}w^2 \, dz
  =\int_{-\infty}^\infty  [(L_{0}-\upsilon)^{-1}w^2]^2\, dz>0\,.
\end{align*}
Noting that $L_0$ admits a single positive eigenvalue at $\upsilon=1,$
it follows that $\Xi$ has a single pole at $\upsilon=1$ and that there
are no other poles for $\upsilon>0.$ On the other hand, as
$\upsilon\rightarrow +\infty,$ we have
$$\Xi(\upsilon)\sim -\frac{1}{\upsilon}\int_{-\infty}^\infty w^4\, dz\rightarrow
0^{-}\,. $$ To summarize, $\Xi(\upsilon)$ has a vertical asymptote at
$\upsilon=1$; $\Xi(0) = 0$, $\Xi\rightarrow 0^{-}$ as
$\upsilon\rightarrow \infty$ and $\Xi$ is increasing for $\upsilon>0$.
It follows that $\Xi(\upsilon)\not=0$ for all $\upsilon>0$, which
proves our claim.

Next, from (\ref{appendixsec3eq1}), we conclude that when
$\Lambda\geq {4/(\bar\chi^{2}v_{\max 0}^{2})}$ we have
\begin{align}\label{appensec3eq4}
  -1+\Lambda\frac{\int_{-\infty}^\infty w^2\, dz}{\int_{-\infty}^\infty w^4\, dz}
  \leq -\frac{\alpha(4-\Lambda)}{2+\alpha}\frac{\int_{-\infty}^\infty w^2 \, dz}
  {\int_{-\infty}^\infty w^4\, dz}\,.
\end{align}
By using the identity
$4\int_{-\infty}^\infty w^2\, dz=3\int_{-\infty}^\infty w^4\,dz$,
(\ref{appensec3eq4}) implies that $\alpha\leq 1-{3\Lambda/4}$.
By observing that the condition
$\Lambda\geq {4/(\bar\chi^{2}v_{\max 0}^{2})}$ holds when
$\lambda_0<0$, we conclude that $\lambda_0<0$ when
\begin{align}\label{alphacritical}
\alpha\leq 1-3\bar\chi^{-2}v_{\max 0}^{-2}\,.
\end{align}

Similarly as in \cite{Wei1999single}, we find when $\alpha=1$,
$\Psi_0\equiv 1$ is an eigenfunction such that
(\ref{NLEPbarzvartilde1}) admits the zero eigenvalue.  If $\alpha>1,$
we claim there exists a positive real eigenvalue of
(\ref{NLEPbarzvartilde1}).  In fact, assume that some $\Lambda$
satisfies $\Lambda \geq 0$. Then, one obtains that
(\ref{NLEPbarzvartilde1}) can be written as the equivalent form
\begin{align*}
  \Psi_0=\alpha\frac{\int_{-\infty}^\infty w^4\Psi_0\, dz}{\int_{-\infty}^\infty w^4
  \,dz} (L_0-\Lambda)^{-1}w^2\,, \quad \mbox{where} \quad L_0\Psi_0=\Psi_{0zz}+
  w^2\Psi_0\,,
\end{align*}
and where $\alpha$ satisfies
$\int_{-\infty}^\infty w^4\, dz=\alpha \int_{-\infty}^\infty
\left[(L_0-\Lambda)^{-1}w^2\right]w^4\, dz$.  Then, we define
$R(\Lambda)$ as
\begin{align*}
  R(\Lambda):= \int_{-\infty}^\infty w^4\, dz-\alpha
  \int_{-\infty}^\infty \left[(L_0-\Lambda)^{-1}w^2\right]w^4\, dz \,.
\end{align*}
Since $R(0)=(1-\alpha)\int_{-\infty}^\infty w^4\, dz<0$ and
$R(\Lambda)\rightarrow+\infty$ as $\Lambda \rightarrow 1^{-},$ we
conclude that there exists a positive $\Lambda\in (0,1)$ such that
$R(\Lambda)=0.$ This finishes the proof of our claim.

By comparing this result and (\ref{alphacritical}), it follows that
there is still a gap region between $1-3\bar\chi^{-2}v_{\max 0}^{-2}$
and $1$.  To eliminate this gap, and obtain a refined
prediction of the threshold $\alpha_c$, we shall rewrite the solution
to (\ref{nlepequiv}) in terms of the hypergeometric function and
perform a detailed asymptotic expansion of it similar to that in
\cite{wei2002critical}.

To do so, we first recall the definition and some properties of
generalized hypergeometric functions \cite{slater1966generalized}.
The generalized hypergeometric functions
${}_pF_q(a_1,\cdots,a_p;b_1,\cdots,b_q;z)$ are defined by the
following series:
\begin{align}\label{hyperdef}
  {}_pF_{q}(a_1,\cdots,a_p;b_1,\cdots,b_q;z)=1+\frac{a_1\cdots a_p}
  {b_1\cdots b_q}\frac{z}{1!}+\frac{(a_1+1)\cdots (a_p+1)}{(b_1+1)\cdots
  (b_q+1)}\frac{z^2}{2!}+\cdots\,.
\end{align}
Their derivatives satisfy a recursion formula, given by
\begin{align}\label{hyperdefdiff}
  \frac{d}{dz}{}_pF_{q}(a_1,\cdots,a_p;b_1,\cdots,b_q;z)=
  \frac{\Pi_{i=1}^Pa_i}{\Pi_{i=1}^q b_i}{}_pF_{q}(a_1+1,\cdots,a_p+1;b_1+1,\cdots,
  b_q+1;z)\,.
\end{align}
In addition, the relationship between
${}_{p+1}F_{q+1}(a_1,\cdots,a_p;b_1,\cdots,b_q;z)$ and
${}_pF_{q}(a_1,\cdots,a_p;b_1,\cdots,b_q;z)$ is
\begin{align}\label{property1}
&{}_{p+1}F_{q+1}(a_1,\cdots,a_p,a_{p+1};b_1,\cdots,b_q,b_{q+1};z)\nonumber\\
  &\qquad =\frac{\Gamma(b_{q+1})}{\Gamma(a_{p+1})\Gamma(b_{q+1}-a_{p+1})}
     \int_{0}^1 t^{a_{p+1}-1}(1-t)^{b_{q+1}-a_{p+1}-1}{}_{p}
     F_{q}(a_1,\cdots,a_p;b_1,\cdots,b_q;tz)\, dt\,,
\end{align}
where $\Gamma$ is the Gamma function
$\Gamma(z):=\int_0^{\infty }t^{z-1}e^{-t} dt$.  In particular, when
$p=2$ and $q=1$, (\ref{hyperdef}) becomes the ordinary hypergeometric
function, which satisfies
\begin{align}\label{property2}
  {}_2F_1(a_1,a_2;b_1;1)=\frac{\Gamma(b_1)\Gamma(b_1-a_2-a_1)}
  {\Gamma(b_1-a_1)\Gamma(b_1-a_2)}\,, \qquad b_1>a_1+a_2\,.
\end{align}
In addition, for $\vert z\vert<1$, ${}_2F_1(a_1;b_1,b_2;z)$ has the following
recursion formula:
\begin{align}\label{property3}
  {}_2F_1(a_1,a_2;b_1;z)=(1-z)^{b_1-a_2-a_1}{}_2
  F_1(b_1-a_1;b_1-a_2,b_1;z)\,, \qquad b_1<a_1+a_2\,.
\end{align}

With this preliminary background, we return to the NLEP
(\ref{nlepequiv}) and use generalized hypergeometric functions to
calculate the critical value of $\alpha$, labeled by $\alpha_c$, for
which $\lambda_0=0$ is an eigenvalue. This implies that
$\Lambda=\delta^2$ in (\ref{nlepequiv}). By defining $\bar z:=2z$,
(\ref{nlepequiv}) can be written when $\lambda_0=0$ and
$\Lambda=\delta^2$ as
\begin{align}\label{psi0barzbarz}
  \Psi_{0\bar z\bar z}+\frac{w^2}{4}\Psi_0-
  \frac{\bar{\kappa}}{4}\frac{\int_{-\infty}^\infty w^2\Psi_0 \, d\bar z}
  {\int_{-\infty}^\infty w^2 \, d\bar z}w^2=\frac{\delta^2}{4}\Psi_0\,,
  \qquad \bar{\kappa} := \frac{\alpha (4-\Lambda)}{2+\alpha}
  \frac{\int_{-\infty}^{\infty} w^2 \, d\bar{z}}
  {\int_{-\infty}^{\infty} w^4 \, d\bar{z}}\,.
\end{align}
To use the standard results in \cite{wei2002critical}, we define
$\bar w:=\frac{3}{2}\sech^2(\bar z/2)$ and $\delta_1:=\delta/2$, so that
(\ref{psi0barzbarz}) becomes
\begin{align}\label{psi0barzbarz1}
  \Psi_{0\bar z\bar z}+\frac{\bar w}{3}\Psi_0-\frac{\bar{\kappa}}{3}
  \frac{\int_{-\infty}^\infty \bar w\Psi_0 \, d\bar z}
  {\int_{-\infty}^\infty \bar w \, d\bar z}\bar w=\delta_1^2\Psi_0\,.
\end{align}
Next, as was shown in \cite{wei2002critical}, (\ref{psi0barzbarz1})
can be transformed into a local problem with an integral
constraint:
\begin{align}\label{localwithnonlocal}
  \Psi_{0\bar z\bar z}+\frac{\bar w}{3}\Psi_0=\delta_1^2\Psi_0+\bar w\,, \qquad
  \int_{0}^\infty \bar w\Psi_0\, d\bar z=\frac{3}{\bar{\kappa}}\int_0^\infty \bar w \,
  d\bar z\,.
\end{align}
Upon defining $G$ by $\Psi_0=\bar w^{\delta_1}G$, we substitute this
relation into (\ref{localwithnonlocal}) to obtain
\begin{align}\label{hypergeometric1}
  G_{\bar z\bar z}-2\delta_1\frac{\bar w_{\bar z}}{\bar w}G_{\bar z}+
  \left[\frac{1}{3}-\frac{\delta_1}{3}(1+2\delta_1)\right]\bar w G=
  \bar w^{1-\delta_1}\,.
\end{align}
We next define $\tilde z:={2\bar w/3}$ and rewrite
(\ref{hypergeometric1}) as
\begin{align}\label{hypergeometric2}
  \tilde z(1-\tilde z)G_{\tilde z\tilde z}+\left[c-(a+b+1)\tilde z\right]G_{\tilde z}
  -abG=\left(\frac{3}{2}\right)^{1-\delta_1}\tilde z^{-\delta_1}\,,
\end{align}
where we have labeled $a$, $b$, and $c$ by
$a=\delta_1+1$, $b=\delta_1-{1/2}$ and $c=1+2\delta_1$.

With this reformulation, we now solve (\ref{hypergeometric2}) in terms
of hypergeometric functions.  To begin, we recall from
\cite{kolokolnikov2009existence} that the two linear
independent solutions to the homogeneous counterpart of
(\ref{hypergeometric2}) are
\begin{align}\label{kernels}
  {}_2F_1(a,b;c;\tilde z)\,, \qquad
  \tilde z^{1-c}{}_2F_1(a-c+1,b-c+1;2-c;\tilde z)\,.
\end{align}
As such, we need only find a particular solution, labeled by $G_1$, of
(\ref{hypergeometric2}).  To do so, we write $G_1$ in the form
$G_1(\tilde z)=\tilde z^{i}\sum_{k=0}^{\infty} c_{k}\tilde z^k$,
where the constants $i$ and $c_k$ need to be determined.  Upon substituting
this infinite series into (\ref{hypergeometric2}), we solve the
resulting recursion equations for $i$ and $c_k$ to get
\begin{align}\label{G1}
  G_1=\left(\frac{3}{2}\right)^{1-\delta_1}(1-\delta_1^2)^{-1}\tilde
z^{1-\delta_1} {}_3F_2\bigg(1,\frac{1}{2},2;2-\delta_1,2+\delta_1;\tilde z\bigg)\,.
\end{align}
It is verify that $\Psi_0=\bar w^{\delta_1}G_1\rightarrow 0$ as
$\bar z\rightarrow +\infty.$ However, we must have
$\Psi_{0\bar z}(0)=0$ since $\Psi_0$ is even.  To enforce this condition,
we write $\Psi_0$ as a linear combination of $G_1$ and the
first homogeneous solution $G_2$ in (\ref{kernels}) as
\begin{align}\label{finalsec630}
  \Psi_0=\bar w^{\delta_1}(G_1+AG_2)\,, \qquad \mbox{where} \quad
  G_2:={}_2F_1\Big(\delta_1+1,\delta_1-\frac{1}{2};2\delta_1+1;\tilde z\Big)\,,
\end{align}
where the constant $A$ will be determined below. To determine
$A$, we apply (\ref{hyperdefdiff}) on (\ref{G1}) to get
\begin{align*}
  \frac{dG_1}{d\tilde z}= \left(\frac{3}{2}\right)^{1-\delta_1}
  (1-\delta_1^2)^{-1}{}_3
  F_2\bigg(2,\frac{3}{2},3;3-\delta_1,3+\delta_1;\tilde z\bigg)\,.
\end{align*}
By using (\ref{property1}), together with (\ref{property2}) and
(\ref{property3}), we further calculate for
$\vert\tilde z\vert\rightarrow 1^-$, that
\begin{align}\label{G1behavior}
  \frac{dG_1}{d\tilde z}\sim \left(\frac{3}{2}\right)^{1-\delta_1}
  \frac{(1-\tilde z)^{-1/2}}{4}{}_2F_1\left(1,\frac{3}{2};3;1\right)
   \sim \left(\frac{3}{2}\right)^{1-\delta_1} (1-\tilde z)^{-1/2}\,.
\end{align}
Similarly, from (\ref{hyperdefdiff}), (\ref{property2}) and
(\ref{property3}), we get that the asymptotic behavior of $G_2$ in
(\ref{finalsec630}) as $\vert\tilde z\vert\rightarrow 1^-$ is
\begin{align}\label{G2behavior}
\frac{dG_2}{d\tilde z}\sim \frac{(1+\delta_1)(\delta_1-\frac{1}{2})}{2\delta_1+1}(1-\tilde z)^{-\frac{1}{2}}\frac{\Gamma(2\delta_2+2)\Gamma(\frac{1}{2})}{\Gamma(2\delta_1+1)\Gamma(\delta_1+\frac{1}{2})}.
\end{align}
Upon combining (\ref{G1behavior}) and (\ref{G2behavior}), we conclude that
$\Psi_{0\bar z}(0)=0$ holds when
\begin{align}\label{Aconstant}
  A=\left(\frac{3}{2}\right)^{1-\delta_1}\frac{\Gamma(1+\delta_1)\Gamma
  \big(\frac{1}{2}+\delta_1\big)}{\big(\frac{1}{2}-\delta_1\big)
  \Gamma(1+2\delta_1)\Gamma\big(\frac{1}{2}\big)}\,.
\end{align}
This gives us an explicit form for $\Psi_0$ in (\ref{finalsec630}).

Next, we focus on the integral constraint in (\ref{localwithnonlocal}).
To begin, we calculate for $\delta_1\ll 1$ that
\begin{align}\label{constraint1}
  \int_0^\infty \bar w^{1+\delta_1}G_1\, d\bar z=&-\frac{3}{2}\int_0^1 \bar
     w^{1+\delta_1}\frac{G_1}{\bar w_{\bar z}}\, d\tilde z =
   \bigg(\frac{3}{2}\bigg)^2(1-\delta_1^2)^{-1}
     \frac{\Gamma(2)\Gamma\big(\frac{1}{2}\big)}{\Gamma\big(\frac{5}{2}\big)}
 {}_4F_3\Big(1,\frac{1}{2},2,2;2-\delta_1,2+\delta_1,\frac{5}{2};1\Big)\nonumber\\
\sim &3(1-\delta_1^2)^{-1}\,,
\end{align}
and
\begin{align}\label{constraint2}
  \int_0^\infty \bar w^{1+\delta_1}G_2 \, d\bar z=&-\frac{3}{2}\int_0^1 \bar
    w^{1+\delta_1}\frac{G_2}{\bar w_{\bar z}}\, d\tilde z\nonumber\\
  =&\bigg(\frac{3}{2}\bigg)^{1+\delta_1}(1-\delta_1^2)^{-1}
     \frac{\Gamma(1+\gamma_1)\Gamma\big(\frac{1}{2}\big)}
     {\Gamma\big(\frac{3}{2}\big)}{}_3F_2\Big(1+\delta_1,\delta_1-\frac{1}{2},
     1+\delta_1;2\delta_1+1,\frac{3}{2}+\delta_1;1\Big)\,.
\end{align}
Moreover, we calculate that
\begin{align}\label{constraint3}
  \int_0^\infty \bar w \, d\bar z=-\frac{3}{2}\int_0^1 \frac{\bar w}
  {\bar w_{\bar z}}\, d\tilde z=\frac{3}{2}\int_0^{1}\frac{1}
  {\sqrt{1-\tilde z}}\, d\tilde z=3\,.
\end{align}
Upon collecting (\ref{constraint1}), (\ref{constraint2}) and
(\ref{constraint3}), we use the constraint in
(\ref{localwithnonlocal}), with $A$ as in (\ref{Aconstant}), to obtain
\begin{align}\label{thresholdeqappen}
  &(1-\delta_1^2)^{-1}{}_4F_3\Big(1,\frac{1}{2},2,2;2-\delta_1,2+\delta_1,
    \frac{5}{2};1\Big)\nonumber\\
  &\qquad + \frac{A}{3}\Big(\frac{3}{2}\Big)^{1+\delta_1}
    \frac{\Gamma(1+\delta_1)\Gamma\big(\frac{1}{2}\big)}
    {\Gamma\big(\frac{3}{2}+\delta_1\big)}{}_3F_2\Big(1+\delta_1,\delta_1
    -\frac{1}{2},1+\delta_1;2\delta_1+1,\frac{3}{2}+\delta_1;1\Big)
    =\frac{3}{\bar{\kappa}} \,.
\end{align}

As a partial verification of our computation, if we let $\delta_1=0$
then (\ref{thresholdeqappen}) yields that $\bar{\kappa}=1$. This
agrees precisely with our leading order threshold $\alpha_{c}\sim 1.$
To seek a refined approximation of this threshold, as obtained by the
next order term of $\alpha_c$, we expand (\ref{thresholdeqappen}) up
to $O(\delta_1).$ To do so, we use the standard result in
\cite{buhring1987behavior} to find
\begin{align}\label{F323}
  {}_3F_2\Big(1+\delta_1,\delta_1-\frac{1}{2},1+\delta_1;2\delta_1+1,
    \frac{3}{2}+\delta_1;1\Big)
   = \frac{\Gamma(b_1)\Gamma(b_2)}{\Gamma(a_3)\Gamma(a_1+1)
    \Gamma(a_2+1)}{}_3F_2\Big(\delta_1,\frac{1}{2},1;2+\delta_1,\frac{1}{2}+
    \delta_1;1\Big)\,, 
\end{align}
where $a_1=1+\delta_1$, $a_2=\delta_1-{1/2}$, $a_3=1+\delta_1$,
$b_1=2\delta_1+1$, and $b_2=\delta_1+{3/2}$.

Next, we expand
\begin{subequations}\label{F321}
\begin{align}
   \frac{\Gamma(b_1)\Gamma(b_2)}{\Gamma(a_3)\Gamma(a_1+r)\Gamma(a_2+r)} &=
  \frac{\Gamma(1+2\delta_1)\Gamma(\frac{3}{2}+\delta_1)}
  {\Gamma(1+\delta_1)\Gamma(2+\delta_1)\Gamma(\frac{1}{2}+\delta_1)}
  = \frac{\frac{1}{2}+\delta_1}{1+\delta_1}+{\mathcal O}(\delta_1^2)\,,\\
  {}_3F_2\Big(\delta_1,\frac{1}{2},1;2+\delta_1,\frac{1}{2}+\delta_1;1\Big)&=
  1+\delta_1+{\mathcal O}(\delta_1^2)\,.
\end{align}
\end{subequations}
Upon substituting (\ref{F321}) into (\ref{F323}), we conclude that
\begin{align}\label{F324}
  {}_3F_2\Big(1+\delta_1,\delta_1-\frac{1}{2},1+\delta_1;2\delta_1+1,
  \frac{3}{2}+\delta_1;1\Big)=\frac{1}{2}
  \left[1+2\delta_1+{\mathcal O}(\delta_1^2)\right] =\frac{1}{2}+\delta_1 +
  {\mathcal O}(\delta_1^2)\,.
\end{align}
Then, by using the identity
${\Gamma^2(1+\delta_1)/\Gamma(1+2\delta_1)}=1+{\mathcal O}(\delta_1^2)$.
we substitute (\ref{F324}) into (\ref{thresholdeqappen}), and recall
that $\delta_1={\delta/2}$ where $\delta={2/(v_{\max 0}\bar\chi)}$. This
yields
\begin{align}\label{appendix643}
  {\bar \kappa}=1-\delta_1+{\mathcal O}(\delta_1^2)=1-\frac{\delta}{2}+
  {\mathcal O}(\delta^2) = 1-\frac{1}{\bar\chi v_{\max 0}} +
  {\mathcal O}(v_{\max 0}^{-2}) \,.
\end{align}
Finally, by relating $\bar{\kappa}$ to $\alpha$ using
(\ref{psi0barzbarz}), and noting the identity
$4\int_{-\infty}^\infty w^2\, d\bar{z}=3\int_{-\infty}^\infty w^4\, d\bar{z}$,
we conclude that (\ref{appendix643}) provides the following
refined threshold at which $\lambda_0=0$, which completes the proof of
Theorem \ref{theorem1}:
\begin{align}
\alpha_{c}\sim 1-\frac{3}{2\bar\chi v_{\max 0}}\,.
\end{align}

\section{Computation of Partial Derivatives for Quasi-Equilibria}\label{appendix:vderiv}

In this appendix, we derive an approximation for ${d v_{\max j}/ds_j}$
from our quasi-equilibrium construction, and we calculate some related
partial derivatives that are needed in our analysis. From
(\ref{algebraicreal}), $v_{\max j}$ and $C_j$ satisfy
\begin{align}\label{largejacobi}
  v_{\max j}^2=\frac{2 C_{j}}{\bar\chi } e^{\bar\chi v_{\max j}} - 
  \frac{2 s_j}{\bar\chi}+s^2_j\,, \qquad C_je^{\bar\chi s_j}=s_j \,.
\end{align}
Upon differentiating the equation for $v_{\max j}$ with respect to $s_j$,
and labeling  $v_{\max j}^{\prime}:= {d v_{\max j}/d s_j}$, we get
\begin{align}\label{jacobian1}
  2v_{\max k}v_{\max j}^{\prime}=(2e^{-\bar\chi s_j}-2\bar\chi s_je^{-\bar\chi s_j})
  \frac{e^{\bar\chi v_{\max j}}}{\bar\chi}+2s_je^{-\bar\chi s_j}e^{\bar\chi v_{\max j}}
  v_{\max j}^{\prime}-\frac{2}{\bar\chi}+2s_j \,,
\end{align}
We solve for $v_{\max j}^{\prime}$ in (\ref{jacobian1}), while eliminating $C_j$ in
(\ref{largejacobi}). After some algebra we obtain
\begin{align}\label{diffvmaxj}
  v_{\max j}^{\prime}=\frac{
  \left( {v_{\max j}^2/s_j} - \bar{\chi} v_{\max j}^2 +\bar{\chi} s_j^2 -s_j\right)}
  {2 v_{\max j} - \bar{\chi} v_{\max j}^2 - 2 s_j + \bar{\chi} s_j^2}\,.
\end{align}
Since $s_j={\mathcal O}(\epsilon |\log\epsilon|^3)$ and
$v_{\max j}={\mathcal O}(|\log\epsilon|)$, we obtain upon retaining only
the first term in the numerator and the first two terms in the denominator
that for $\epsilon \to 0$
\begin{align}\label{diffvmaxj_f}
  v_{\max j}^{\prime}=\frac{dv_{\max j}}{d s_j} \sim -\frac{\zeta_{\max j}}
  {\bar{\chi} s_j}  \,, \qquad
  \zeta_{\max j}:= \left(1 - \frac{2}{\bar{\chi} v_{\max j}}\right)^{-1}\,.
\end{align}
This result (\ref{diffvmaxj_f}) is needed in (\ref{jac:balance}) for
analyzing the Jacobian of the quasi-equilibrium construction.

In a similar way, by taking the partial derivative of $v_{\max k}$
with respect to the location $x_i$ of the $i^{\mbox{th}}$ spike in the
quasi-equilibrium pattern, we readily derive the following result
for $\epsilon \to 0$ that is needed in (\ref{Fj;xi}) and (\ref{dae3:vk_xi}):
\begin{align}\label{diffvmaxi}
  \partial_{x_i} v_{\max k} \sim -\frac{\zeta_{\max k}}
  {\bar{\chi} s_k}\partial_{x_i} s_{k}  \,, \qquad \zeta_{\max k}=
  \left(1 - \frac{2}{\bar{\chi} v_{\max k}}\right)^{-1}\,.
\end{align}

\section{Calculation of $\mathcal G_g$ and $\mathcal P_g$}\label{appensecB}
  
In this appendix, for $d_1\in {\mathcal T}_e$, we calculate the matrix
spectrum of ${\mathcal G}_g$, as given in \eqref{matrixGg}, as well as the
matrix ${\mathcal P}_g$ that was defined in (\ref{matrixPg}).  To do so, for
$d_1\in {\mathcal T}_e$, we introduce the auxiliary BVP
\begin{align}\label{bvpydipole}
         \frac{d_1}{\mu}y^{\prime\prime} +\bar u y=0\,, \quad 1<x<1\,; \quad
       y^{\prime}(\pm 1)=0\,;\quad \Big[\frac{d_1}{\mu}y\Big]_j=b_j\,,\quad
      \Big[\frac{d_1}{\mu}y^{\prime}\Big]_j=0\,,
\end{align}
for $j=1,\ldots,N$. Here $[y]_j:=y(x_j^+)-y(x_j^-)$ with $x_j=x_j^0$
as given by (\ref{locations}).  The solution to (\ref{bvpydipole}) is
$y=\sum_{k=1}^N b_k g(x;x_k)$, where the dipole Green's function
$g(x;x_k)$ satisfies (\ref{green:dipole_small}). Upon defining
$\boldsymbol{y}^{\prime}:=\left(y^{\prime}_{1},\ldots,y^{\prime}_{N}\right)^T$,
$\boldsymbol{\langle y\rangle}:=\left(\langle y\rangle_1,\ldots,
  \langle y\rangle_N\right)^T$, and $\boldsymbol{b}:=\left(b_1,\ldots,b_N
  \right)^T$, where $y^{\prime}_j=y^{\prime}(x_j)$ and
$\langle y\rangle_j=\big(y(x_j^+)+y(x_j^-)\big)/2$, we conclude that
\begin{align}\label{yprimeandomegadipole}
  \boldsymbol{y}^{\prime}=\mathcal G_g\boldsymbol{b}\,, \qquad
  \langle\boldsymbol{y}\rangle =\mathcal P_g\boldsymbol{b}\,.
\end{align}

The inverses of $\mathcal G_g$ and $\mathcal P_g$ exist and are
tridiagonal when $d_1\in {\mathcal T}_e$. To show this, we solve
(\ref{bvpydipole}) on each subinterval where we impose the continuity
conditions on $y^{\prime}$ across $x_j$. This yields that
\begin{align}\label{analyticydipole}
y=\left\{\begin{array}{ll}
   -\frac{y_1^{\prime}}{\theta}\frac{\cos[\theta(1+x)]}{\sin[\theta(1+x_1)]}\,,
           &-1<x<x_1\,,\\
  \frac{y_j^{\prime}}{\theta}\frac{\cos[\theta(x_{j+1}-x)]}
    {\sin[\theta(x_{j+1}-x_j)]}-\frac{y_{j+1}^{\prime}}{\theta}
     \frac{\cos[\theta(x-x_j)]}{\sin[\theta(x_{j+1}-x_j)]}\,,&x_j<x<x_{j+1}\,,
             \quad j=1,\ldots,N-1\,,\\
  \frac{y_{N}^{\prime}}{\theta}\frac{\cos[\theta(1-x)]}{\sin[\theta(1-x_{N})]}
           \,,&x_{N}<x<1\,,
\end{array}
\right.
\end{align}
where $\theta=\sqrt{\mu \bar u/d_1}.$  By using
(\ref{analyticydipole}), we satisfy the jump conditions in
(\ref{bvpydipole}) to get
\begin{align}\label{appb:dgtogg}
  \mathcal D_g\boldsymbol{y}^{\prime}=\frac{\mu\theta}{d_1}\boldsymbol{b}\,,
  \qquad \mathcal G_g=\frac{\mu\theta}{d_1} \mathcal D_g^{-1}\,.
\end{align}
Here, for $d_1\in {\mathcal T}_e$,  $\mathcal D_g$ is the invertible
tridiagonal matrix defined by
\begin{align}\label{mathcalDg}
\mathcal D_g=\left( \begin{array}{ccccccc}
  d_g & f_g&0&\cdots &0 &0& 0\\
  f_g & e_g&f_g&\cdots &0 &0& 0\\
   0 & f_g&e_g&\ddots &0 &0& 0\\
     \vdots & \vdots&\ddots&\ddots &\ddots &\vdots& \vdots\\
    0 & 0&0&\ddots &e_g &f_g& 0\\
  0 & 0&0&\cdots &f_g &e_g& f_g\\
   0 & 0&0&\cdots &0 &f_g& d_g\\
    \end{array}
  \right)\,.
\end{align}
where $d_g=\cot\left({2\theta/N}\right)+\cot\left({\theta/N}\right)$,
$e_g=2\cot\left({2\theta/N}\right)$ and
$f_g=-\csc\left({2\theta/N}\right)$, for which the identity
$d_g=e_g-f_g$ holds.  When $d_1\in {\mathcal T}_e$ (see
(\ref{d1:admiss})), i.e.~${2\theta/N}<\pi$, we see that $e_g$, $d_g$
and $f_g$ are well-defined.

Similarly, we rewrite $\boldsymbol{\langle y\rangle}$ in terms of
$\boldsymbol{y}^{\prime}$ as $\boldsymbol{\langle y\rangle }=-
(2\theta)^{-1}\csc\left({2\theta/N}\right)\mathcal C\boldsymbol{y}^{\prime}$.
where ${\mathcal C}$ was defined in (\ref{app:matc}).  By combining
the second equation in (\ref{yprimeandomegadipole}) with this result
we obtain for $d_1\in {\mathcal T}_e$ that
\begin{align}\label{matrixPgrepeat}
    \mathcal P_g=-\frac{\mu}{2d_1}\csc\left(\frac{2\theta}{N}\right)
    \mathcal C\mathcal D_g^{-1}\,.
  \end{align}
  The matrix spectrum of the tridiagonal matrix ${\mathcal D}_g$,
  labeled by $\mathcal D_g \boldsymbol{v}=\xi \boldsymbol{\nu}$ where
  $\boldsymbol{\nu}=(\nu_1,\ldots,\nu_N)^T$, is readily calculated as
  in \cite{iron2001stability} and the result is summarized in
  Proposition \ref{prop41}.

Finally, when $\lambda_0=0$, we establish a key identity
\begin{align}\label{app:pt_pg}
  {\mathcal P}^T = - {\mathcal P}_g \,,
\end{align}
which relates (\ref{matrixGg}) for ${\mathcal P}$ when $\lambda_0=0$
to (\ref{matrixPg}). One way to derive this identity is to observe
from (\ref{def}) that when $\lambda_0=0$, we have $e_g = - e$,
$f_g = -f$, and $d_g=-e+f$. By using these expressions in
(\ref{mathcalDg}) a direct matrix multiplication yields the identity
${\mathcal C} {\mathcal D}_g = - {\mathcal D} {\mathcal C}$, where
${\mathcal D}$ and ${\mathcal C}$ are defined in (\ref{inverseG}) and
(\ref{app:matc}), respectively. The result (\ref{app:pt_pg}) follows
by comparing (\ref{matrixPgrepeat}) and (\ref{matrixP}), and noting
that ${\mathcal D}$ and ${\mathcal D}_g$ are symmetric.

\section{Diagonalization of the Matrix ${\mathcal M}$ for
  the Small Eigenvalues}\label{appensecC}

In this appendix, when $d_1\in {\mathcal T}_e$, we show how to
diagonalize the matrix ${\mathcal M}$ in (\ref{jacobianbutinsmallep})
to obtain the result given in Proposition \ref{prop:small} for the
small eigenvalues. From (\ref{jacobianbutinsmallep}), the
matrix for the small eigenvalues is
\begin{align}\label{appc:magain}
    \mathcal M=\frac{2\bar\chi}{3}v_{\max 0}^3\mathcal G_g-
  \frac{2v_{\max 0}^2\zeta_0}{a_g}\mathcal P\left(I+\frac{3\zeta_0}
  {\bar\chi a_gv_{\max 0}}
  \mathcal G\right)^{-1}\mathcal P_g+\frac{s_0 \bar u\mu}{\epsilon d_1} I\,,
  \qquad \zeta_0:=\left(1-\frac{2}{\bar{\chi}v_{\max 0}}\right)^{-1}\,.
\end{align}

We begin by focusing on the middle term in ${\mathcal M}$. We first
introduce the matrix decomposition of ${\mathcal D}$ by
${\mathcal D} = {\mathcal Q} {\mathcal K} {\mathcal Q}^T$, where
${\mathcal K}=\mbox{diag}(\kappa_1,\ldots,\kappa_N)$ and
${\mathcal Q}$ is the orthogonal matrix formed from the eigenvectors
$\boldsymbol{q}_j$ in Proposition \ref{prop31} when $\tau=0$. For
$\tau=0$, the eigenvalues $\kappa_j$ of ${\mathcal D}$ are related
to the eigenvalues $\xi_j$ of ${\mathcal D}_g$ by
\begin{equation}\label{map:ktoxi}
  \kappa_1 =2\tan\left({\theta/N}\right) \,, \quad
  \kappa_j=-\xi_j= -2\cot\left({2\theta/N}\right) +
  2 \csc\left({2\theta/N}\right)\cos\left({\pi (j-1)/N}\right)
  \,, \quad j=2,\ldots,N\,.
\end{equation}
By using (\ref{key:G_to_D}) with $\lambda_0=0$, we obtain that
${\mathcal G} = \sqrt{\frac{\mu}{d_1\bar{u}}} {\mathcal Q}
\mathcal{K}^{-1} {\mathcal Q}^T$, which yields
\begin{align}\label{appc:ginv}
  {\mathcal P}  \left(I+\frac{3\zeta_0}{\bar\chi a_gv_{\max 0}}\mathcal
  G\right)^{-1} {\mathcal P}_g =
 {\mathcal P} {\mathcal Q} \left( I+\frac{3\zeta_0}{\bar\chi a_gv_{\max 0}}
  \sqrt{\frac{\mu}{\bar u d_1}}{\mathcal K}^{-1} \right)^{-1}Q^T
  {\mathcal P}_g\,.
\end{align}

Next, we use (\ref{matrixPgrepeat}) and (\ref{matrixP}) to conclude
that ${\mathcal P}{\mathcal D}=\left({\mathcal P}_g{\mathcal D}_g\right)^T$
so that
\begin{align}\label{appc:ptopg}
  {\mathcal P} = \left({\mathcal P}_g{\mathcal D}_g\right)^T {\mathcal D}^{-1}
  =\left({\mathcal P}_g{\mathcal D}_g\right)^T {\mathcal Q}
  {\mathcal K}^{-1} {\mathcal Q}^T \,.
\end{align}
By combining (\ref{appc:ptopg}) and (\ref{appc:ginv}), and using
${\mathcal Q}{\mathcal Q}^T=I$, we get
\begin{align}\label{appc:ginv_2}
  {\mathcal P}  \left(I+\frac{3\zeta_0}{\bar\chi a_gv_{\max 0}}\mathcal
  G\right)^{-1} {\mathcal P}_g = {\mathcal R}  {\mathcal D_g}^{-1}\,,
  \qquad \mbox{where} \quad   {\mathcal R} :=
  \left({\mathcal P}_g{\mathcal D}_g\right)^T
  {\mathcal Q} {\mathcal H} {\mathcal Q}^T
  \left({\mathcal P}_g{\mathcal D}_g\right)\,.
\end{align}
Here ${\mathcal R}$ is defined in terms of a diagonal matrix
${\mathcal H}$ given by
\begin{align}\label{appc:hmat}
  {\mathcal H} :=
  \left(\frac{3\zeta_0}{\bar\chi a_gv_{\max 0}}
  \sqrt{\frac{\mu}{\bar u d_1}} I + {\mathcal K}\right)^{-1}=
  \mbox{diag}(h_1,\ldots,h_N) \,.
\end{align}
Therefore, by using ${\mathcal G}_g=\frac{\mu\theta}{d_1}{\mathcal D}_g^{-1}$
from (\ref{appb:dgtogg}), together with (\ref{appc:ginv_2}),
we can write (\ref{appc:magain}) as
\begin{align}\label{appc:mnew}
  \mathcal M=\left( \frac{2\bar\chi}{3}v_{\max 0}^3 \frac{\mu\theta}{d_1}
  I + \frac{s_0 \bar u\mu}{\epsilon d_1} {\mathcal D_g} -
  \frac{2v_{\max 0}^2\zeta_0}{a_g} {\mathcal R} \right) {\mathcal D}_g^{-1}
  \,.
\end{align}

Next, we must focus on analyzing the matrix ${\mathcal R}$ defined by
(\ref{appc:ginv_2}). By using (\ref{matrixPgrepeat}), we obtain
\begin{align*}
  \left( {\mathcal P}_g {\mathcal D}_g\right)^T =
  -\frac{\mu}{2d_1} \csc\left(\frac{2\theta}{N}\right)
    \mathcal C\,,
\end{align*}
where ${\mathcal C}$ is given in (\ref{app:matc}). In this way,
it is convenient to write ${\mathcal R}$ as
\begin{align*}
  {\mathcal R} = \frac{\mu^2}{4d_1^2} \csc^{2}\left(\frac{2\theta}{N}\right)
                 {\mathcal C}^T {\mathcal Q} {\mathcal H} {\mathcal Q}^T
                 {\mathcal C} =
     \frac{\mu^2}{4d_1^2} \csc^{2}\left(\frac{2\theta}{N}\right)
      {\mathcal Q}_g {\mathcal Q}_g^T
      {\mathcal C}^T {\mathcal Q} {\mathcal H} {\mathcal Q}^T
      {\mathcal C} {\mathcal Q}_g {\mathcal Q}_g^T \,,
\end{align*}
where ${\mathcal Q}_g$ are the normalized eigenvectors of ${\mathcal D}_g$
(see Proposition \ref{prop41}), arising in the matrix
decomposition
\begin{align}\label{appc:dg_decom}
  {\mathcal D}_g = {\mathcal Q}_g {\mathcal K}_g {\mathcal Q}_g^T \,,
  \qquad {\mathcal K}_g =\mbox{diag}(\xi_1,\ldots,\xi_N)\,,
\end{align}
where $\xi_j$ are the eigenvalues of ${\mathcal D}_g$ as given in
Proposition \ref{prop41}. In this way, we can write
${\mathcal R}$ as 
\begin{align}\label{appc:rfinal}
  {\mathcal R} = {\mathcal Q}_g \Sigma {\mathcal Q}_g^T \,, \qquad
  \mbox{where} \quad    \Sigma :=  \frac{\mu^2}{4d_1^2}
  \csc^{2}\left(\frac{2\theta}{N}\right)
  {\mathcal S} {\mathcal H} {\mathcal S}^T \,, \quad
  {\mathcal S}:={\mathcal Q}_g^{T} {\mathcal C}^T
  {\mathcal Q} \,.
\end{align}

The key step in the analysis is the calculation of
$\Sigma$ in (\ref{appc:rfinal}) using the explicit forms for the matrices
${\mathcal Q}_g$, ${\mathcal C}$, and ${\mathcal Q}$, as was done in
section 4.2 of \cite{iron2001stability}. This calculation in
\cite{iron2001stability} showed that $\Sigma$ is a diagonal matrix
given by
\begin{align}\label{appc:sigma_eig}
  \Sigma = \mbox{diag}(\omega_1,\ldots,
  \omega_N) \,, \qquad \mbox{where} \qquad
  \omega_j := \frac{\mu^2}{d_1^2}\csc^2\left(\frac{2\theta}{N}\right)
  \sin^2\left(\frac{(j-1)}{N}\pi\right) h_j \,, \quad j=1,\ldots,N \,.
\end{align}
Here $h_j$, for $j=1,\ldots,N$, are the diagonal entries of
${\mathcal H}$ that can be identified from (\ref{appc:hmat}).

Upon substituting (\ref{appc:rfinal}) and
${\mathcal D}_g^{-1}= {\mathcal Q}_g {\mathcal K}_g^{-1} {\mathcal Q}_g^T$ into
(\ref{appc:mnew}), and recalling (\ref{433insection4}), we obtain that the
matrix eigenvalue problem for the small eigenvalues reduces to
\begin{align}\label{appc:lam_mat}
  \lambda \boldsymbol{c}\sim -\epsilon^3\beta_0\mathcal M\boldsymbol{c}\,,
  \qquad \mbox{where} \quad
  \mathcal M={\mathcal Q}_g \left( a {\mathcal K}_g^{-1} + b I -
  \frac{2v_{\max 0}^2\zeta_0}{a_g} \Sigma {\mathcal K}_g^{-1} \right)
  {\mathcal Q}_g^{T} \,.
\end{align}
This key result shows that ${\mathcal M}$ is diagonalizable by
the eigenspace ${\mathcal Q}_g$ of the Green's dipole matrix.  In
(\ref{appc:lam_mat}), 
\begin{align}\label{appc:ab}
  a:= \frac{2 \bar\chi }{3} v_{\max 0}^3
  \left(\frac{\mu\theta}{d_1}\right) \,, \qquad b: =
  \frac{s_0\bar u \mu}{\epsilon d_1} =
  \frac{2\bar\chi}{3} v_{\max 0}^3 \left(\frac{a_g{\bar u}\mu}{d_1}\right)\,,
\end{align}
where we have used the result
$s_0\sim {2\bar{\chi}a_gv_{\max 0}^3\epsilon/3}$ from (\ref{ag}) to
simplify $b$.

Finally, by introducing
$\boldsymbol{\tilde c}=\mathcal Q_{g}^T \boldsymbol{c}$ in
(\ref{appc:lam_mat}), we readily obtain from (\ref{appc:ab}) that
the small eigenvalues are given explicitly as
in (\ref{prop:small_eig}) of Proposition \ref{prop:small}. The
constants $\omega_j$, as given in (\ref{prop:small_omegaj}), are
obtained from (\ref{appc:sigma_eig}) by using the diagonal entries of
${\mathcal H}$ that can be identified from (\ref{appc:hmat}) and
(\ref{map:ktoxi}).

\section{Bifurcation Point for the Emergence of Asymmetric Steady-States}
\label{app:asymmetric}

In this appendix we verify that the simultaneous zero-eigenvalue
crossing threshold for the small eigenvalues, as given in
(\ref{hj:zero}), coincides with the bifurcation point at which
asymmetric steady-state solution branches bifurcate from the symmetric
steady-state branches constructed in \S \ref{sec2}.

To do so, we proceed in a similar way as in \cite{asymm} by constructing 
a steady-state solution of (\ref{timedependent}) on a canonical domain
$|x|\leq \ell$, with $u_x=v_x=0$ at $x=\pm \ell$ and with a spike centered at
$x=0$. On this domain, the leading-order outer solution $u_{o\ell}(x)$ satisfies
(see (\ref{outerproblem}))
\begin{align}\label{app:outerproblem}
  {\mathcal L}_{0\ell} u_{o\ell} := \frac{d_1}{\mu}u_{o\ell xx}+\bar u
  u_{o\ell}=\frac{2\bar\chi\epsilon}{3 } v_{\max \ell}^3 \, \delta(x) \,, \quad
  |x|\leq \ell\,; \qquad u_{o\ell x}(\pm \ell)=0 \,,
\end{align}
where, in analogy with (\ref{vm:dominant}), $v_{\max \ell}$ satisfies the
dominant balance
\begin{equation}\label{app:vm_eq}
  \frac{1}{2}v_{\max \ell}^2 \sim \frac{s_{\ell}}{\bar{\chi}}
  e^{\bar{\chi}v_{\max \ell}} \,,
  \qquad \mbox{with} \quad s_{\ell}=u_{o\ell}(0) \,.
\end{equation}
To solve (\ref{app:outerproblem}) we let $G_{\ell}(x)$ be the Green's
function satisfying ${\mathcal L}_{0\ell} G_{\ell}=\delta(x)$, with
$G_{\ell x}(\pm \ell)=0$.
For $\theta\neq {m\pi/\ell}$ with $m=1,2,\ldots$, where
$\theta=\sqrt{{\mu \bar{u}/d_1}}$, we obtain that
\begin{equation}\label{app:u0l}
  u_{o\ell}(x)=\frac{2\bar\chi}{3}\epsilon v_{\max \ell}^3 \, G_{\ell}(x) \,, \qquad
  \mbox{where} \quad
  G_{\ell}(x)=\frac{\mu \cos\left[\theta(\ell-|x|)\right]}{2\theta d_1
    \sin(\theta \ell)}\,.
\end{equation}
By evaluating (\ref{app:u0l}) at $x=0$ we can calculate $s_{\ell}$, which
is needed in (\ref{app:vm_eq}) for determining $v_{\max \ell}$. In this way,
we obtain after some algebra that at $x=\ell$
\begin{subequations}\label{app:final}
\begin{equation}\label{app:final_1}
  u_{o\ell}(\ell) = c {\mathcal B}(\ell) \,, \qquad \mbox{where} \quad
  {\mathcal B}(\ell):=
  \frac{v_{\max \ell}^3}{\sin(\theta \ell)} \,, \quad c: = \frac{\eps \bar{\chi}}
  {3\bar{u}} \sqrt{\frac{\mu \bar{u}}{d_1}} \,.
\end{equation}
Here $v_{\max \ell}$ as a function of
$\ell$ satisfies the nonlinear algebraic equation
\begin{equation}\label{app:final_2}
  v_{\max \ell} e^{\bar{\chi} v_{\max \ell}} \cot(\theta \ell) =
  {\bar{\chi}/(2c)}\,.
\end{equation}
\end{subequations}

As similar to the analysis in \cite{asymm} for the GM model, the
construction of asymmetric steady-state patterns for
(\ref{timedependent}) relies on determining $\ell_1$ and
$\ell_2$ for which ${\mathcal B}(\ell_1)={\mathcal
  B}(\ell_2)$. As a result, we have $u_{o\ell}(\ell_1)=u_{o\ell}(\ell_2)$, which
allows for the construction of a $C^{1}$ global solution on $|x|\leq 1$
with $M_1$ and $M_2$ small and large spikes, respectively, when the length
constraint $\ell_1 M_1 + \ell_2 M_2=1$ is satisfied (cf.~\cite{asymm}).

The bifurcation point along the steady-state symmetric branch where
such asymmetric equilibria emerge is determined by setting
${\mathcal B}^{\prime}(\ell)=0$ with $\ell={1/N}$. From (\ref{app:final_1})
and the logarithmic derivative of (\ref{app:final_2}) we get
\begin{equation}\label{app:bp}
  {\mathcal B}^{\prime}(\ell) = \frac{v_{\max \ell}^2}{\sin(\theta \ell)}
  \left[ 3 v_{\max \ell}^{\prime} - \theta v_{\max \ell} \cot(\theta \ell)\right]\,,
  \qquad   v_{\max \ell}^{\prime}\left(1 + \frac{1}{\bar{\chi} v_{\max \ell}}\right)
  = \frac{\theta}{\bar{\chi} \sin(\theta \ell)\cos(\theta \ell)} \,.
\end{equation}
Upon combining these two equations we conclude that
\begin{equation}\label{app:bprime}
  {\mathcal B}^{\prime}(\ell) = \frac{\theta \, v_{\max \ell}^3}{
    \sin^2(\theta \ell)\cos(\theta \ell)} \left[
    \frac{3}{1+ \bar{\chi} v_{\max \ell}} - \cos^2(\theta \ell)\right]
  \,.
\end{equation}
By setting ${\mathcal B}^{\prime}(\ell)=0$ with $\ell={1/N}$, and using the
double-angle formula for $\cos^{2}(\theta \ell)$, we readily obtain that
the threshold value of $\theta$ is
\begin{equation}
  \cos\left( \frac{2\theta}{N}\right) = \frac{1-a_1}{1+a_1} \,, \qquad
  \mbox{where} \quad a_1 = \frac{1}{3} \left(\bar{\chi} v_{\max} -2\right)\,.
\end{equation}
This threshold agrees precisely with the zero-eigenvalue crossing
result (\ref{hj:zero}) for the small eigenvalues.

\section{Computation of $\beta_0$ and $\beta_j$}\label{appendixFfinal}

In this appendix, we show how to obtain the estimate (\ref{betajvalue})
for $\beta_j$, where $\beta_j$ was defined in (\ref{solvability1}) of \S
\ref{sec5}. For simplicity, in the analysis below we will drop the
subscript $j$ in $V_{0 j}$, $v_{\max j}$, $C_j$, $s_j$, and
$v_{\max j}$.

We begin by recalling from (\ref{Hamiltonian}) that the leading order
steady state $v$-equation for the spike profile is
\begin{align}\label{appendixFfinaleq1}
  V_0^{\prime\prime}-V_0+C e^{\bar\chi V_0}=0\,, \qquad -\infty<y<+\infty\,;
  \qquad V_0(0)=v_{\max} \,, \quad V_0(\infty)=s \,,
\end{align}
where $v_{\max }^2= 2C e^{\bar\chi v_{\max}}-2 s+ s^2$ and
$C=se^{-\bar\chi s}$.

From the results in Proposition \ref{prop1} for the sub-inner region,
we conclude that there exists a positive constant
$y_0={\mathcal O}\left({1/v_{\max}}\right)\ll 1$ such that 
\begin{equation*}
  V_0 \sim v_{\max}+\frac{1}{\bar\chi}
        \log\left[ \sech^2\left( \frac{v_{\max}\bar\chi y}{2} \right)
        \right] \,, \quad 0<y<y_0 \,; \qquad
  U_0 \sim \frac{\bar\chi }{2} v_{\max}^2\sech^2\Big(\frac{v_{\max}\bar\chi y}{2}
        \Big) \,, \quad 0<y<y_0 \,.
\end{equation*}
The decay behavior of $U_0$ and $V_0$ is obtained by noting that
$V_0^{\prime\prime}-V_0+\bar{\chi}s V_0\approx 0$ for $y>y_0$. Since
$s\ll 1$, this yields $V_{0}^{\prime\prime}-V_0\approx 0$. With this
observation, and by enforcing continuity across $y=y_0$, we estimate
\begin{align}\label{appF:V0}
V_0\sim\left\{\begin{array}{ll}
                v_{\max}+\frac{1}{\bar\chi}\log\left[\sech^2
                \left(\frac{v_{\max}\bar\chi y}{2}\right)
          \right]\,,&y<y_0\,,\\
v_{\max}e^{-(y-y_0)}+\frac{1}{\bar\chi}\log \left[\sech^2\Big(\frac{v_{\max}\bar
              \chi y_0}{2}\Big)\right]\,,&y>y_0\,,
\end{array}
                                           \right.\,, \qquad
V_{0}^{\prime}\sim\left\{\begin{array}{ll}
-v_{\max}\tanh(\frac{v_{\max}\bar\chi y}{2})\,,&y<y_0\,,\\
-v_{\max}e^{-(y-y_0)}\,,&y>y_0\,.
\end{array}
\right.
\end{align}
Moreover, since $U_0=Ce^{\bar\chi V_0},$ we obtain in a similar way that
\begin{align}\label{appF:U0}
U_0\sim\left\{\begin{array}{ll}
     \frac{\bar\chi }{2}v_{\max}^2\sech^2\Big(\frac{v_{\max}\bar\chi y}{2}\Big)
           \,,&y<y_0\,,\\
    Ce^{\bar\chi v_{\max}e^{-(y-y_0)}}\Big(\sech^2\Big(\frac{v_{\max}\bar\chi y_0}
            {2}\Big)\Big)\, ,&y>y_0\,.
\end{array}
\right.
\end{align}

By using (\ref{appF:U0}) we calculate that
\begin{align*}
\int_0^y \frac{1}{U_0}\, d\xi\sim \left\{\begin{array}{ll}
       \frac{2}{\bar\chi v_{\max}^2}\Big(\frac{y}{2}+
     \frac{\sinh(v_{\max}\bar\chi y)}{2\bar\chi v_{\max}}\Big)\,, &y<y_0\,,\\
\frac{2}{\bar\chi v_{\max}^2}\Big(\frac{y_0}{2}+\frac{\sinh(v_{\max}\bar\chi y_0)}{2\bar\chi v_{\max}}\Big)+\frac{1}{s}(y-y_0)\,, &y>y_0\,.
\end{array}
\right.
\end{align*}
Then, upon multiplying by $U_0$, we obtain
\begin{align}\label{appF:t1}
U_0\int_0^y \frac{1}{U_0}\, d\xi\sim \left\{\begin{array}{ll}
\frac{y}{2} \sech^2(\frac{v_{\max}\bar\chi y}{2})+\frac{1}{2v_{\max}\bar\chi}\tanh(\frac{v_{\max}\bar\chi y}{2})\sech(\frac{v_{\max}\bar\chi y}{2})\,, &y<y_0\,,\\
(y-y_0)+\frac{2C}{\bar\chi v_{\max}^2}e^{\bar\chi v_{\max}e^{-(y-y_0)}}\sech^2\Big(\frac{v_{\max}\bar\chi y_0}{2}\Big)\Big(\frac{y_0}{2}+\frac{\sinh(v_{\max}\bar\chi y_0)}{2\bar\chi v_{\max}}\Big)\,, &y>y_0\,.
\end{array}
\right.
\end{align}

By multiplying (\ref{appF:t1}) with $V_0^{\prime}$ from (\ref{appF:V0})
and integrating, we observe that the dominant contribution to the
integrand arises from multiplying the $y-y_0$ term in (\ref{appF:t1})
with the $-v_{\max}e^{-(y-y_0)}$ term in (\ref{appF:V0}). In this way,
\begin{align*}
  \int_0^\infty U_{0}V_{0}^{\prime}\Big(\int_0^y\frac{1}{U_{0}}\, d\xi\Big)\, dy\sim
  -v_{\max}\int_{y_0}^{+\infty} e^{-(y-y_0)} (y-y_0)\, dy\sim -v_{\max}\,.
\end{align*}
In a similar way, we estimate that
$\int_{0}^{+\infty}\left(V_{0}^{\prime}\right)^2\, dy\sim v^2_{\max}
  \int_{y_0}^{+\infty} e^{-2(y-y_0)}\, dy\sim {v_{\max}^2/2} \,.
$
We conclude from (\ref{solvability1}) that $\beta_j\sim {2/v_{\max }}$, as
was claimed in (\ref{betajvalue}).

Next, we recall from (\ref{433insection4}) in
our analysis of the small eigenvalues that
$\beta_0=-{\int_0^\infty yV_0^{\prime}\, dy/\int_0^\infty
  \left(V_{0}^{\prime} \right)^2 \, dy}$. By using (\ref{appF:t1}) and
(\ref{appF:V0}), we can readily verify that
$$\int_0^\infty yV_0^{\prime}\, dy\sim \int_0^\infty U_{0}V_{0}^{\prime}
\Big(\int_0^y\frac{1}{U_{0}}\, d\xi\Big) \, dy\,,$$
which establishes that $\beta_j\sim \beta_0$ when evaluated at the
steady-state solution.

\section{The Equivalence Between Some Matrices}\label{appendixG}

In this appendix, we show the relationship between the matrices
\begin{align}\label{appG:g_mats}
  \nabla \mathcal G:=(\partial_{x_j} G(x_j^0;x_k^0))_{N\times N}\,, \qquad
  (\nabla \mathcal G)^T:=(\partial_{x_k}G(x_j^0;x_k^0))_{N\times N}\,, \qquad
  \nabla^2 \mathcal G:=(\partial_{x_j}\partial_{x_k} G(x_j^0;x_k^0))_{N\times N}\,,
\end{align}
used in the linearization of the DAE system and the matrices
${\mathcal P}$, ${\mathcal P}_g$, and ${\mathcal G}_g$, as defined in
(\ref{matrixGg}), (\ref{matrixPg}), and (\ref{matrixGg}),
respectively, that were used in \S \ref{sec4} in our analysis of the
small eigenvalues. Recall that the diagonal entries in the matrices in
(\ref{appG:g_mats}) were defined in (\ref{computation533insec51new})
in terms of the regular part $R$ of the Green's function (see
(\ref{g:decomp})).

We first show that $\nabla G={\mathcal P}$. To establish this, we use
the decomposition (\ref{g:decomp}) to obtain
\begin{equation}\label{appG:1}
  G_{x}(x;x_k) = \left\{\begin{array}{ll}
                          \frac{\mu}{2d_1} + R_x(x;x_k) \,,&x>x_k\,,\\
                          -\frac{\mu}{2d_1} + R_x(x;x_k) \,,&x<x_k\,.\\
\end{array}
\right.
\end{equation}
As such, we identify that the average across the $k^{\mbox{th}}$ spike is
simply
$\langle
G_x\rangle_k={\left(G_x(x_k^+;x_k)+G_x(x_k^-;x_k)\right)/2}=R_x(x_k;x_k)$.
By comparing (\ref{appG:g_mats}) and (\ref{matrixGg}), and
recalling (\ref{computation533insec51new}) for $j=k$, we conclude that
$\nabla G={\mathcal P}$.

Next, we show that $\left(\nabla \mathcal G\right)^T=-{\mathcal P}_g$.
We first differentiate the BVP (\ref{greenequation}) for $G(x;x_k)$ with
respect to $x_k$ to get
\begin{equation*}
  \frac{d_1}{\mu}  \left( \partial_{x_k} G(x;x_k)\right)_{xx}  + \bar{u}
  \left(\partial_{x_k} G(x;x_k) \right) = - \delta^{\prime}(x-x_k)\,; \qquad
  \partial_{x} \left(\partial_{x_k} G(x;x_k) \right)\vert_{x=\pm 1}=0 \,.
\end{equation*}
By comparing this result with the BVP (\ref{green:dipole_small})
satisfied by the dipole Green's function, we conclude that
\begin{equation}\label{appG:Gtog}
    \partial_{x_k} G(x;x_k) = - g(x;x_k)\,, \qquad -1<x<1 \,,
\end{equation}
so that for $j\neq k$ we have
$\partial_{x_k} G(x_j;x_k) = - g(x_j;x_k)$.  It follows that the
off-diagonal entries in $\left(\nabla {\mathcal G}\right)^T$ and
${\mathcal P}_g$ are identical. For the diagonal entries, where $j=k$,
we use (\ref{appG:Gtog}) and the decomposition (\ref{g:decomp}) to obtain
\begin{align}\label{appG:gexp}
  g(x;x_k) =
 \left\{\begin{array}{ll}
          \partial_{x_k}\left( \frac{\mu}{d_1}(x-x_k) + R(x;x_k)\right)=
-\frac{\mu}{d_1} - \partial_{x_k} R(x;x_k)     \,,&x>x_k\,,\\ 
          \partial_{x_k}\left( -\frac{\mu}{d_1}(x-x_k) + R(x;x_k)\right)=
\frac{\mu}{d_1} - \partial_{x_k} R(x;x_k)    \,,&x< x_k\,,\\
\end{array}
\right.
\end{align}
Upon defining $\langle g\rangle_k=\frac{1}{2} \left(
  g(x_{k}^{+};x_k) +g(x_{k}^{-};x_k)\right)$, we conclude from (\ref{appG:gexp})
and the reciprocity $R(x;y)=R(y;x)$ of the Green's function that
$\langle g\rangle_k=-\partial_{x_k}R(x;x_k)\vert_{x=x_k}=
-\partial_{x_k}R(x_k;x)\vert_{x=x_k}$. This implies that the diagonal
entries of $\mathcal P_g$ in (\ref{matrixPg}) are the same as those of
$(\nabla \mathcal G)^T$ in (\ref{appG:g_mats}). It follows that
$\left(\nabla \mathcal G\right)^T={\mathcal P}^T = -{\mathcal P}_g$. We
remark that the relation ${\mathcal P}^T = -{\mathcal P}_g$ was also
derived using an alternative approach in (\ref{app:pt_pg}) at the end
of Appendix \ref{appensecB}.

Our next identity is to establish that
$\nabla^2 \mathcal G=-{\mathcal G}_g$.  The equivalence between the
off-diagonal entries in these matrices, where $j\neq k$, is
established by setting $x=x_j$ in (\ref{appG:Gtog}) and
differentiating in $x_j$ to obtain
\begin{equation*}
  \partial_{x_j} \left[ \partial_{x_k} G(x_j;x_k)\right] = -
  \partial_{x_j} g(x_j;x_k) = - \partial_{x} g(x;x_k)\vert_{x=x_j} \,.
\end{equation*}
Next, we differentiate (\ref{appG:gexp}) with
respect to $x$ and upon evaluating at $x=x_k$, we compare the
resulting expression with (\ref{computation533insec51new}) to obtain
that
\begin{equation*}
  g_{x}(x:x_k)\vert_{x=x_k} = -\frac{\partial}{\partial x}\vert_{x=x_k}
  \frac{\partial}{\partial y} \vert_{y=x_k} R(x,y) =
  -\partial^2_{x_k} G(x_j;x_k) \,, \quad j=k\,.
\end{equation*}
We conclude that the diagonal entries in $\nabla^2 \mathcal G$ and
$-{\mathcal G}_g$ are also identical. It follows that
$\nabla^2 \mathcal G=-{\mathcal G}_g$.

Finally, we calculate $R_{xx}(x;x_j)\vert_{x=x_j}$ as needed in
(\ref{daes:term2_more}). By using the decomposition (\ref{g:decomp})
we write (\ref{appG:1}) as
\begin{equation*}
  G_x(x;x_j) = -\frac{\mu}{2d_1} + \frac{\mu}{d_1} H(x-x_j) +
  R_x(x;x_j)\,,
\end{equation*}
where $H(z)$ is the Heavyside function. Therefore,
$G_{xx}(x;x_j) = \frac{\mu}{d_1} \delta(x-x_j) + R_{xx}(x;x_j)$ on
$|x|<1$.  Upon substituting this expression into the BVP
(\ref{greenequation}) for $G$, we conclude that
$R_{xx}(x;x_j) = \frac{\bar{u}\mu}{d_1} G(x;x_j)$, so that
\begin{equation}\label{appendixG:rxx}
  R_{xx}(x_j;x_j) = \frac{\bar{u}\mu}{d_1} G(x_j;x_j) \,.
\end{equation}

\end{appendix}
\bibliographystyle{plain}
\bibliography{ref}
\end{document}